\documentclass{amsart}
\usepackage[
	pdfauthor={Bin Guan},
    pdftitle={Averages and Nonvanishing of Central Values of
				Triple Product L-Functions},
				pdfencoding=auto
]{hyperref}

\usepackage{amssymb,amsmath,amscd,amsthm,amsfonts}
\usepackage{graphicx,bbding,pifont,wasysym,skull}

\usepackage{chemarrow} 
\usepackage{mathrsfs} 
\usepackage{enumerate}
\usepackage{multicol}

\newcommand{\fD}{\mathfrak{D}}
\newcommand{\fd}{\mathfrak{d}}
\newcommand{\fp}{\mathfrak{p}}
\newcommand{\CC}{\mathbb{C}}

\newcommand{\A}{\mathbb{A}}
\newcommand{\Q}{\mathbb{Q}}
\newcommand{\bH}{\mathbb{H}}
\newcommand{\R}{\mathbb{R}}
\newcommand{\Z}{\mathbb{Z}}
\newcommand{\cA}{\mathcal{A}}
\newcommand{\cB}{\mathcal{B}}
\newcommand{\cF}{\mathcal{F}}
\newcommand{\cH}{\mathcal{H}}
\newcommand{\cO}{\mathcal{O}}

\newcommand{\id}{\operatorname{id}}

\newcommand{\ord}{\operatorname{ord}}
\newcommand{\Hom}{\operatorname{Hom}}
\newcommand{\Tr}{\operatorname{Tr}}
\newcommand{\Ad}{\operatorname{Ad}}
\newcommand{\JL}{\operatorname{JL}}
\newcommand{\Ram}{\operatorname{Ram}}
\newcommand{\fin}{\operatorname{fin}}
\newcommand{\alg}{\operatorname{alg}}

\newcommand{\ONB}{\mathcal{ONB}}
\newcommand{\1}{\mathbf 1}

\DeclareMathOperator{\iso}{\xrightarrow{\sim}}
\DeclareMathOperator{\Sym}{Sym}
\DeclareMathOperator{\vol}{vol}
\DeclareMathOperator{\disc}{disc}
\DeclareMathOperator{\sgn}{sgn}

\newcommand{\GL}{\operatorname{{GL}}}
\newcommand{\PGL}{\operatorname{{PGL}}}
\newcommand{\SU}{\operatorname{{SU}}}

\newtheorem{Theorem}{Theorem}[section]
\newtheorem{Lemma}[Theorem]{Lemma}
\newtheorem{Corollary}[Theorem]{Corollary}
\newtheorem{Proposition}[Theorem]{Proposition}

\newtheorem{Example}[Theorem]{Example}
\newtheorem{remark}[Theorem]{Remark}

\begin{document}

\title
[Averages and Nonvanishing of 
Triple Product
Central $L$-Values]
{
Averages and Nonvanishing of 
Central Values of Triple Product \texorpdfstring{$L$}{L}-Functions
}
\author{Bin Guan}
\address{Data Science Institute\\Shandong University \\Jinan\\ China}
\email{bguan.math@sdu.edu.cn}

\date{\today}

\begin{abstract}
We use Ichino's period formula combined with a
relative trace formula to obtain exact formulas for  
the central values of triple product $L$-functions
$L(3k-1,f\times g\times h)$,
averaged over Hecke-normalized cusp newforms
of weight $2k$ on $\Gamma_0(N)$,
with $f$ and $g$ varying and $h$ fixed.
We also present some applications of 
the average formulas to the nonvanishing problem,
giving a lower bound on the number 
of nonvanishing central $L$-values when 
one of the forms is fixed.
\end{abstract}


\thanks{}

\maketitle

\tableofcontents

\section{Introduction}

\subsection{Main results}

The aim of this paper is to 
establish exact average formulas
of central values of 
triple product $L$-functions 
associated to three normalized cusp newforms,
while one of the three forms is fixed.
We also give some applications of 
the average formulas to the nonvanishing problems.

Let $N,k$ be positive integers and $N$ be square-free.
Let $\cF_{2k}(N)$ denote the set of normalized cusp newforms
of weight $2k$ on $\Gamma_0(N)$
which are eigenforms of Hecke operators.
Normalizing $f(z)=\sum_{n\geq 1} a_n(f)e^{2\pi inz},g,h\in\cF_{2k}(N)$
such that $a_1(f)=a_1(g)=a_1(h)=1$,
we can define the triple product $L$-function as the Euler product
\[L_{\fin}(s,f\times g\times h)=\prod_p L_p(s,f\times g\times h)\]
(see Section \ref{section.L-functions}
for the data of local $L$-factors).
B{\"o}cherer and Schulze--Pillot 
\cite{bocherer1993central} have shown that
\[\Lambda(s,f\times g\times h)=
(2\pi)^{6k-3-4s}\Gamma(s)\Gamma(s+1-2k)^3
L_{\fin}(s,f\times g\times h)\]
has an analytic continuation to the entire $s$-plane
and satisfies the functional equation
\[\Lambda(s,f\times g\times h)=\varepsilon
N^{5(3k-1-s)}
\Lambda(6k-2-s,f\times g\times h),\quad
\text{where }\varepsilon=-\prod_{p\mid N}\varepsilon_p=\pm 1.\]
Here the local root number is defined by 
$\varepsilon_p=-a_p(f)a_p(g)a_p(h)p^{-3(k-1)}=\pm 1$ when $p\mid N$.
(One can observe that the central value is at $s=3k-1$.
But after a translation $s\mapsto s+3k-\frac 32$,
the functional equation may be written in the form
\[\Lambda(s,f\times g\times h)=\varepsilon
N^{5(\frac 12-s)}
\Lambda(1-s,f\times g\times h)\]
so that the central value is $\Lambda(\frac 12,f\times g\times h)$.
See \eqref{adelizedL}.)

In this paper we will develop a relative trace formula (RTF)
for automorphic forms on a specific quaternion algebra.
We use this RTF, along with the Jacquet--Langlands correspondence
and Ichino's triple product formula,
to translate the triple product average value
into a sum of orbital integrals.
Finally we evaluate this sum explicitly
and apply Waldspurger's formula
to rewrite certain orbital integrals
as central values of Rankin--Selberg $L$-functions.

Before we state the main theorem, we recall 
the CM modular forms arising from Hecke characters. 
For an imaginary quadratic extension $E/\Q$ with discriminant $-d$, 
consider a character $\Omega$
on $E^\times\backslash\A_E^\times$,
whose restriction on $\A^\times=\A_{\Q}^\times$ is trivial.
Assume that it is unramified everywhere at finite places.
(We denote by 
$\widehat{[E^\times]}$ 
the set of such characters,
while we define $[E^\times]=\A^\times E^\times\backslash\A_E^\times$.)
At infinity we have $\Omega_\infty=\sgn^{2m}$ for some $m\in\Z$
where $\sgn(z)=z/|z|$.
Recall (cf. \cite[Theorem 12.5]{iwaniec1997topics} for example) 
that, when $\Omega$ does not factor 
through the norm map $N_E:\A_E^\times\to\A^\times$,
there is a modular form $\Theta_\Omega$
of level $d$, weight $2|m|+1$
and nebentypus $\chi_{-d}$
such that $L(s,\Theta_\Omega)=L(s,\Omega)$.
For $h\in \cF_{2k}(N)$ 
let $L(s, h\times \Theta_\Omega)$ denote 
the completed Rankin--Selberg $L$-function 
which satisfies a functional equation 
relating the value at $s$ to that at $2k+2|m|+1-s$.

\begin{Theorem}[Main Theorem]\label{mainthm}
Let $N$ be a square-free integer with an odd number of prime factors.
For any $h\in\cF_{2k}(N)$,
\begin{multline}\label{maineq1}
\frac{N}{2^{12k-4} \pi^{6k-1}}
\sum_{\substack{f,g\in\cF_{2k}(N)\\
\varepsilon_p=-1,\ \forall p\mid N}}
\frac{L_{\fin}(3k-1,f\times g\times h)}
{(f,f)(g,g)(h,h)}
=\frac{1-24\delta(k)/\varphi(N)}
{2^{\omega(N)}\Gamma(2k)\Gamma(2k-1)^2}\\
+\frac {\Gamma(2k-1)}
{\Gamma(k)^3\Gamma(3k-1)} 
\frac{4L_0 \cdot 2^{\ord_2(N)}
\prod\limits_{p\mid N}\frac{1-\chi_{-4}(p)}{2}
+6\sqrt 3 L_1 \cdot 2^{\ord_3(N)}
\prod\limits_{p\mid N}\frac{1-\chi_{-3}(p)}{2}}
{(4\pi)^{2k}(h,h)},
\end{multline}
where
\[\begin{split}
L_0 &= I_0(\1)\cdot L_{\fin}(k, h)L_{\fin}(k, h\otimes\chi_{-4})
+ \sum_{\1\neq\Omega\in\widehat{[E_0^\times]}}
I_0(\Omega)\cdot L_{\fin}(k+|m|+\tfrac 12, h\times \Theta_\Omega),
\\
L_1&= I_1(\1)\cdot L_{\fin}(k, h)L_{\fin}(k, h\otimes\chi_{-3})
+\sum_{\1\neq\Omega\in\widehat{[E_1^\times]}}
I_1(\Omega)\cdot L_{\fin}(k+|m|+\tfrac 12, h\times \Theta_\Omega).
\end{split}\]
Here $\varepsilon_p$ is the local root number,
$(\cdot,\cdot)$ is the Petersson inner product
on $\cF_{2k}(N)$ defined by
\begin{equation}\label{def.Petersson}
(f_1,f_2)=\int_{\Gamma_0(N)\backslash \cH}
f_1(z)\overline{f_2(z)}y^{2k}\ \frac{dx\ dy}{y^{2}};
\end{equation}
 $E_0=\Q(\sqrt{-1})$, $E_1=\Q(\sqrt{-3})$;
$I_0(\1)$, $I_1(\1)$, $I_0(\Omega)$, $I_1(\Omega)$
(defined in Theorem \ref{maingeometric})
are constants depending only on $k$ and $\Omega$
(and on the Fourier coefficients
$a_2(h)$, $a_3(h)$ when $2$ or $3\mid N$ respectively);
$\varphi(N)$ is the Euler's totient function;
$\omega(N)=\sum_{p\mid N}1$
is the number of distinct prime factors of $N$; and
\[\delta(k)=\begin{cases}1,&\text{if }k=1,\\
0,&\text{otherwise.}\end{cases}\]
\end{Theorem}
\begin{remark}\label{remark.mainthm}
1)
For an automorphic representation
$\Pi=\pi_1\otimes\pi_2\otimes\pi_3$
for $\PGL(2,\Q)^3$,
Prasad \cite{prasad1990trilinear} shows that 
$\varepsilon_v(\frac 12,\Pi)=-1$
if and only if all $\pi_{i,v}$ are 
discrete series representations and 
$\Hom_{D_v^\times}(\Pi_v',\CC)\neq 0$,
where $\Pi_v'$ is a representation of 
the group $(D_v^\times)^3$ of quaternions,
which relates to $\Pi_v$ via the Jacquet--Langlands 
correspondence.
According to this result and 
the discussion in Section \ref{RepresentationSU2}
we know, if all three representations have weight $2k$,
$\varepsilon_\infty(\frac 12,\Pi)=-1$.
Under the assumption that
$N$ has an odd number of prime factors,
the global $\varepsilon$-factor
(the global root number) for the triple product $L$-function
is $1$ in the case of the Main Theorem.

If, in addition, $N$ is prime,
then the condition 
of root numbers $\varepsilon_p=-1$, $\forall p\mid N$
can be removed,
otherwise the global root number is $-1$ 
and the central $L$-values vanish. 
That is to say, 
\begin{multline*}
\frac{N}{2^{12k-4} \pi^{6k-1}}
\sum_{f,g\in\cF_{2k}(N)}
\frac{L_{\fin}(3k-1,f\times g\times h)}
{(f,f)(g,g)(h,h)}
=\frac{1-\frac{24}{N-1}\delta(k)}
{2\Gamma(2k)\Gamma(2k-1)^2}\\
+\frac {\Gamma(2k-1)}
{\Gamma(k)^3\Gamma(3k-1)} 
\frac{4L_0 \cdot \frac{1-\chi_{-4}(N)}{2}
+6\sqrt 3 L_1 \cdot \frac{1-\chi_{-3}(N)}{2}}
{(4\pi)^{2k}(h,h)}
\end{multline*}
when $N$ is a prime $\neq 2,3$.

2) 
Here $\chi_d$ is the Dirichlet character
defined by the Kronecker symbol $(\frac{d}{\cdot})$,
where $d\equiv 0,1\pmod 4$ is a fundamental discriminant.
The product over $p\mid N$ can be seen as
a congruency condition.
For example, when $N$ is square-free,
\[2^{\ord_2(N)}\prod_{p\mid N}\frac{1-\chi_{-4}(p)}2=\begin{cases}
0,&\text{if $N$ has a prime factor $\equiv 1\pmod {4}$;}\\
1,&\text{otherwise.}
\end{cases}\]
If, in addition, 
$N$ has a prime factor $\equiv 1\pmod 4$ and one $\equiv 1\pmod 3$
(or if $N$ has a prime factor $\equiv 1\pmod {12}$),
for any $h\in\cF_{2k}(N)$, the Main Theorem simplifies to
\begin{equation}\label{maineq2}
\frac{N}{2^{12k-4}\pi^{6k-1}}
\sum_{\substack{f,g\in\cF_{2k}(N)\\
\varepsilon_p=-1,\ \forall p\mid N}}
\frac{L_{\fin}(3k-1,f\times g\times h)}{(f,f)(g,g)(h,h)}=
\frac{1-24\delta(k)/\varphi(N)}
{2^{\omega(N)}\Gamma(2k-1)^2\Gamma(2k)}.
\end{equation}

3) 
Feigon and Whitehouse \cite{feigon2010exact}
have shown an exact average formula of
triple product central $L$-values
associated to three newforms of weight $2$ and
of the same prime level $p$.
Their approach does not use the relative trace formula.
Rather,
using the classical period formula of Gross--Kudla \cite{gross1992heights},
they write $L_{\fin}(2,f\times g\times h)$
as a finite sum of functions defined on a finite set.

Our Main Theorem is a generalization of \cite{feigon2010exact}'s result
to the case of general weight and level.
In Section \ref{examples246} we give average formulas
\eqref{maineqwt2} \eqref{maineqwt4} and \eqref{maineqwt6}
when $f,g,h$ all have small weights ($\leq 6$)
with the constants calculated explicitly. 
In particular, when $2k=2$ and
$N$ is a prime with $N=11$ or $N>13$, 
our result \eqref{maineqwt2} reproves the main theorem
of \cite{feigon2010exact}.

4) 
Feigon and Whitehouse \cite{feigon2009averages} also
have shown an exact average formula for central values 
of certain twisted quadratic base change $L$-functions 
averaged over Hilbert modular forms of a fixed weight and level. 
With their results one can obtain an exact average formula 
of $L_{\fin}(3k-1,f\times g\times h)$
as all three forms run through $\cF_{2k}(N)$.
See Section \ref{section.sumover3}.

\end{remark}

In Section \ref{section.nonvanishing} we apply the above theorem
to the nonvanishing problem.
In particular we give a lower bound on the number 
of nonvanishing $L_{\fin}(3k-1,f\times g\times h)$ when 
$h$ is fixed.

\begin{Corollary}
Let $N$ be a square-free integer with an odd number of prime factors.
Then, for any $h\in\cF_{2k}(N)$,
\[\#\{(f,g)\in\cF_{2k}(N)\times\cF_{2k}(N):
L_{\fin}(3k-1,f\times g\times h)\neq 0\}\gg_{k,\epsilon} N^{3/4-\epsilon}.\]
\end{Corollary}

Moreover, when $N$ is a prime $\equiv 1 \pmod {12}$, 
we can apply the Main Theorem to derive a Lindel\"of-on-average upper bound
(cf. Corollary \ref{LindelofOnAverage}):
\[\sum_{f,g\in\cF_{2k}(N)}
L_{\fin}(3k-1,f\times g\times h)
\ll_{\epsilon}(kN)^{2+\epsilon},\]
while the convexity bound for an individual 
$L_{\fin}(3k-1,f\times g\times h)
\ll_{\epsilon}(k^8N^5)^{1/4+\epsilon}$
will only give an upper bound $O(k^{13/4+\epsilon}N^{4+\epsilon})$.

The RTF, originally introduced by Jacquet to study periods integrals
(cf. \cite{jacquet2005guide} for an overview),
along with Ichino's period formula \cite{ichino2008trilinear},
plays an important role in the proof of the adelic version of
the Main Theorem (Theorem \ref{mainadelic}) in
the case of general weight and level.
This method could also be applied to the case
of triple product $L$-functions attached to Hilbert modular forms over
a totally real number field.
We now give an overview of these tools.

\subsection{Ichino's period formula}

From an adelic point of view, one can consider the
triple product $L$-function $L(s,\pi_1\otimes\pi_2\otimes\pi_3)$
(defined in Section \ref{section.L-functions})
associated to three irreducible unitary
cuspidal automorphic representations of $\PGL(2,\A)$,
where $\A$ is the adele ring over $\Q$.
Harris and Kudla \cite{harris2004conjecture} 
proved a conjecture of Jacquet,
that the central value $L(1/2,\pi_1\otimes\pi_2\otimes\pi_3)\neq 0$
if and only if there exists a quaternion
algebra $D$ over $\Q$ such that the period integral
\[\int_{\A^\times D^\times(\Q)\backslash D^\times(\A)}
\phi_1(x)\phi_2(x)\phi_3(x)\ d^\times x\neq 0\]
for some $\phi_i\in\pi_i'$,
where $\A^\times$ is diagonally embedded in $D^\times(\A)$ as its center,
and $\pi'$ is the irreducible unitary automorphic representation
of $D^\times(\A)$ associated to $\pi$ by the Jacquet--Langlands correspondence.

Moreover,
Gross and Kudla \cite{gross1992heights}
established an explicit identity relating
central $L$-values and period integrals
(which are finite sums in their case),
when the cusp forms are of prime levels and weight $2$.
This Gross--Kudla period formula is the key ingredient
when \cite{feigon2010exact} proves 
their average formula for 
the case of weight $2$ and a prime level.
B{\"o}cherer, Schulze--Pillot \cite{bocherer1993central}
and Watson \cite{watson2008rankin} generalized this identity
to more general levels and weights.
At last, Ichino \cite{ichino2008trilinear}
proved an adelic version of this period formula which work for
all the cases:
\[\frac
{\left|\int_{[D^\times]}\phi_1(h)\phi_2(h)\phi_3(h)\ dh\right|^2}
{\prod_{i=1}^3\int_{[D^\times]}\phi_i(h)\overline{\phi_i(h)}\ dh}
\sim
\frac{L(\frac 12,\pi_1\otimes\pi_2\otimes\pi_3)}
{L(1,\pi_1\otimes\pi_2\otimes\pi_3,\Ad)}.\]
Here $[D^\times]=\A^\times D^\times(\Q)\backslash D^\times(\A)$.
The exact formula can be found in Theorem \ref{Ichino}.

Many authors have derived several explicit versions of Ichino's formula
in various cases. An incomplete list includes 
\cite{woodbury2012explicit,
nelson2011equidistribution,
nelson2014bounds,
hu2017triple,
chen2018deligne,
humphries2020random,
collins2020anticyclotomic,
Hsieh2021Hida}.
In this paper, we continue the work 
of Woodbury \cite{woodbury2012explicit} 
and Chen--Cheng \cite{chen2018deligne}
to establish Theorem \ref{explicitIchino}, 
an explicit version of Ichino's formula for 
$L(1/2,\pi_1\otimes\pi_2\otimes\pi_3)$, 
in terms of the period integrals 
appearing in the spectral side of the RTF, 
which we use to prove the adelic Main Theorem \ref{mainadelic}. 

\subsection{Jacquet's relative trace formula}\label{RTF}

Here we consider a general version of the RTF
(cf. \cite{feigon2009averages,zhang2017periods}).
Let $G$ be an anisotropic algebraic group
defined over a global field $F$
and $H_1$, $H_2$ be closed subgroups of $G$.
Let $f\in C^\infty_c(G(\A_F))$.
Integrating $f$ against the action of $G(\A_F)$ gives a linear map
\[R(f) : L^2(G(F)\backslash G(\A_F)) \to L^2(G(F)\backslash G(\A_F))\]
defined by
\[(R(f)\phi)(x) = \int_{G(\A_F)} f(g) \phi(xg) \ dg.\]
One sees that $R(f)$ is an integral operator with kernel
\[K_f(x,y)=\sum_{\gamma\in G(F)}f(x^{-1}\gamma y),\quad x,y\in G(\A_F).\]
Let $\cA(G)$ denote the set of automorphic representations on $G(\A_F)$.
Fixing automorphic forms $\phi_{H_1}$, $\phi_{H_2}$
in $\pi_1\in\cA(H_1)$ and $\pi_2\in\cA(H_2)$ respectively,
we define a distribution
\[
I(f) = \int_{H_1(F) \backslash H_1(\A_F)} \int_{H_2(F) \backslash H_2(\A_F)}
K_f(h_1, h_2) \phi_{H_1}(h_1)\overline{\phi_{H_2}(h_2)} \ dh_1 \ dh_2.
\]
The RTF for the case $H_1\backslash G/H_2$ gives two expressions of $I(f)$.

From the spectral decomposition of $L^2(G(F)\backslash G(\A_F))$, we have
\[K_f(x, y) = \sum_{\pi\in\cA(G)} \sum_{\phi\in\ONB(\pi)}
(\pi(f)\phi)(x) \overline{\phi(y)},\]
where for each $\pi\in\cA(G)$, $\ONB(\pi)$
denotes an orthonormal basis of $V_{\pi}$.
Under the assumption that $G$ is anisotropic,
for any test function $f\in C^\infty_c(G(\A_F))$,
the operator $R(f)$ is of trace class,
that means the kernel $K_f(x, y)$ is absolutely convergent,
independent of the choice of basis.
The spectral expansion for $K_f(x, y)$ gives
\[I(f) = \sum_{\pi\in\mathcal A(G)} I_{\pi}(f),\]
where, for each $\pi\in\cA(G)$,
\[I_{\pi}(f) = \sum_{\phi\in\ONB(\pi)}
\int_{H_1(F) \backslash H_1(\A_F)} (\pi(f)\phi)(h_1)\phi_{H_1}(h_1) \ dh_1
\overline{\int_{H_2(F) \backslash H_2(\A_F)} \phi(h_2)\phi_{H_2}(h_2) \ dh_2}.\]

From the geometric expansion of the RTF,
\[
I(f) = \int_{H_1(F) \backslash H_1(\A_F)} \int_{H_2(F) \backslash H_2(\A_F)}
\sum_{[\gamma]}\sum_{(\theta_1,\theta_2)}
f(h_1^{-1}\theta_1^{-1}\gamma \theta_2 h_2)
\phi_{H_1}(h_1)\overline{\phi_{H_2}(h_2)}
\ dh_1 \ dh_2.\]
Here $[\gamma]$ runs through representatives of $H_1(F)\backslash G(F)/H_2(F)$;
and $(\theta_1,\theta_2)$ runs through
$H_1(F)\times H_2(F)/(H_1(F)\times H_2(F))_\gamma$,
where we define \[(H_1(F)\times H_2(F))_\gamma=
\{(\theta_1,\theta_2)\in H_1(F)\times H_2(F):
\theta_1^{-1}\gamma\theta_2=\gamma\}.\]
Let $\theta_i h_i$ be the new $h_i$ ($i=1,2$). Then we have
\[I(f)=\sum_{[\gamma]\in H_1(F)\backslash G(F)/H_2(F)}I_{[\gamma]}(f)\]
where
\[I_{[\gamma]}(f)=\int_{(H_1(F)\times H_2(F))_\gamma
\backslash H_1(\A_F)\times H_2(\A_F)}
f(h_1^{-1}\gamma h_2)
\phi_{H_1}(h_1)\overline{\phi_{H_2}(h_2)} \ d(h_1,h_2).\]

As a generalization of the Arthur--Selberg trace formula,
Jacquet's relative trace formula
is a powerful tool in the study of period integrals.
With a period formula like Ichino's,
the average of central $L$-values
appears in the spectral decomposition of a certain distribution.
In the compact quotient case one can get an explicit
orbital decomposition of the same distribution.
For example, Feigon and Whitehouse \cite{feigon2009averages}
have considered the RTF for the case
$E^\times \backslash D^\times / E^\times$
(where $D$ is a quaternion algebra over a totally real number field $F$
and $E/F$ is a quadratic field extension embedded in $D$)
and, using a period formula of Waldspurger,
obtained an exact formula for averages of central values of
twisted quadratic base change $L$-functions
associated to Hilbert modular forms.
In this paper an analogous method is applied to the case
$D^\times \backslash (D^\times\times D^\times) / D^\times$
to obtain exact formulas for averages of central values of
triple product $L$-functions.

\subsection{A sketch of the proof and the structure of the paper}

Let $D$ be the definite quaternion algebra 
over $\Q$ with discriminant $N$.
In Section \ref{section.JLcorrespondence} 
we use the Jacquet--Langlands correspondence
to associate a newform $f\in\cF_{2k}(N)$
with a cuspidal automorphic representation $\pi'$
of $D^\times(\A)$ with the same level and weight.
With the definitions of $L$-functions in Section \ref{section.L-functions},
we translate the Main Theorem to an adelic version,
which is Theorem \ref{mainadelic}.

In Sections \ref{section.spectral} and \ref{section.geometric} 
we apply the RTF for the case
\[D^\times \backslash (D^\times\times D^\times) / D^\times\]
to prove the adelic Main Theorem \ref{mainadelic}.
Briefly, let $G'$ be the algebraic group defined over $\Q$
with $G'(\Q) =Z(\Q)\backslash D^\times(\Q)$.
We take $G'\times G'$ as the ``big'' group $G$ 
as in the previous section, 
and take $H_1=H_2=G'$ embedded in $G$ diagonally.
On one hand, we choose a suitable test function
$f\in C^\infty_c(G'(\A)\times G'(\A))$
(as in Section \ref{ChooseTest}),
so that (almost) all terms $I_\pi(f)$
on the spectral side of the RTF
are associated to representations 
of level $N$ and weight $2k$.
Moreover $I_\pi(f)$,
when it is nonzero, is essentially 
a period integral
\[\left|\int_{\A^\times D^\times(\Q)\backslash D^\times(\A)}
\phi_1(h)\phi_2(h)\phi_3(h)\ dh\right|^2,\]
and it can be written as the central value 
of corresponding triple product $L$-function
via Ichino's formula \cite{ichino2008trilinear}.
(An explicit Ichino's formula is given in Section \ref{section.Ichino}.)

On the other hand,
there are at most three terms $I_{[\gamma]}(f)$ 
which do not vanish
on the geometric side of the RTF
(see Theorem \ref{orbdecomp}).
For a nontrivial orbit $[\gamma]$, the nonzero term
$I_{[\gamma]}(f)$ is of the form
\[[\text{some congruency condition}]\times 
\int_{\A^\times E_\gamma^\times\backslash \A_{E_\gamma}^\times}
\phi'(t)\overline{\phi''(t)}\ dt,\]
where $E_\gamma$ is a quadratic extension of $\Q$.
In Section \ref{section.harmonic} we use harmonic analysis 
and write the integral over 
$\A^\times E_\gamma^\times\backslash \A_{E_\gamma}^\times$
as a sum of period integrals,
and the period integrals over the torus $E_\gamma^\times$
can be related to $L(1/2, \pi_{E_\gamma} \otimes\Omega)$ 
by Waldspurger's period formula \cite{waldspurger1985valeurs}.
(An explicit Waldspurger's formula is given 
in Section \ref{section.waldspurger}.)


At last in Section \ref{section.nonvanishing}
we apply the Main Theorem to the nonvanishing problems.
With the weight $2k$ fixed,
a lower bound on the number 
of nonvanishing central values of triple product $L$-functions 
is given in Section \ref{section.nonvanishing},
when one of the forms is fixed.
We also have a Lindel\"of-on-average result 
(Corollary \ref{LindelofOnAverage})
for the upper bound of $\sum_{f,g}L_{\fin}(3k-1,f\times g\times h)$,
and a result (Corollary \ref{nonvanishmodp}) 
on the nonvanishing modulo 
suitable primes $p$ of the algebraic part of triple product $L$-values. 


\section{Preliminaries}

\subsection{Quaternion Algebras}

For any field $F$ of characteristic $\neq 2$, and $a, b \in F^\times$,
let
\[D=\left(\frac{a,b}{F}\right)=F\{i,j\}/(i^2-a,~j^2-b,~ij+ji),\]
denote the quaternion algebra
with $F$-basis $1_D, i_D, j_D, k_D$ 
(or $1, i, j, k$ for short) such that $i^2 = a$, $j^2 = b$
and $ij =-ji = k$ (so $k^2 = -ab$).
We know that either $D\cong M(2,F)$ is split 
or $D$ is a division algebra.
When $F$ is global, let $\Ram(D)$ be a set of places $v$ in $F$
such that $D$ is ramified at $v$,
i.e. such that $D_v = D \otimes_F F_v$ is not split.
In fact, $\Ram(D)$ is a finite set, and 
when $F=\Q$ it always has even cardinality.

Denote by $\Tr_D$ and $N_D$ the (reduced) trace and norm in $D$.
It is well known that, in a non-split quaternion algebra, 
elements are conjugate to each other if and only if they have
the same trace and norm.
That is to say, 
conjugacy classes in $D^\times$ can be parametrized by
traces and norms.
Instead of $D^\times$, in this paper we deal with
the group $G'(F)
=F^\times\backslash D^\times$,
for which we have the following two lemmas.

\begin{Lemma}\label{parametrization}
Fix a set $\Sigma$ of representatives in $F^\times/(F^\times)^2$.
(For example, when $F=\Q$, $\Sigma$ can be the set of square-free integers.)
Then $[\bar x]\mapsto (\Tr_D(x),N_D(x))$ is a well-defined injection from
the set of conjugacy classes of $G'(F)=F^\times\backslash D^\times$ to
$(\{\pm 1\}\backslash F)\times\Sigma$.
\end{Lemma}
\begin{proof}
For any two representatives $x_1,x_2\in D^\times$ of $\bar x\in G'(F)$,
there exists $\lambda\in F^\times$ so that $x_2=\lambda x_1$,
and we have $\Tr_D(x_2)=\lambda \Tr_D(x_1)$, $N_D(x_2)=\lambda^2N_D(x_1)$.
Fixing $\Sigma$,
moreover, with $N_D(\bar x)$ fixed, we can only take $\lambda^2=1$.
So the traces of representatives in $\bar x$ might differ by a sign.
\end{proof}

For a fixed quaternion,
one can check the following result about its centralizer
by direct calculation.

\begin{Lemma}\label{centralizer}
Suppose $D=(\frac{a,b}{F})$ is a division algebra.
Consider the centralizer of $x\in G'(F)
=F^\times\backslash D^\times$
given by $G'_x(F)=\{g\in G'(F): gx=xg\}$.
Then \begin{itemize}
\item $G'_x(F)=G'(F)$ when $x\in F^\times$;
\item when $\Tr_D(x)\neq 0$ and $x\notin F^\times$, 
$G'_x(F)$ is the image of 
$\{\lambda+\mu x\in D^\times:\lambda,\mu\in F\}$
under the quotient map $D^\times\twoheadrightarrow G'(F)$;
\item and when $\Tr_D(x)= 0$ (we write $x=\beta i+\gamma j+\delta k$), 
$G'_x(F)$ is the union of 
the image of 
$\{\lambda+\mu x\in D^\times:\lambda,\mu\in F\}$ 
and that of 
$\{x_1i+x_2j+x_3k\in D^\times:x_1,x_2,x_3\in F; \
a\beta x_1+b\gamma x_2=ab\delta x_3\}$.
\end{itemize}
\end{Lemma}

When $F=\Q$, a quaternion algebra over $\Q$ is called definite
if $D_\infty= D \otimes_\Q \R$ is not split
(i.e. isomorphic to the algebra $\bH$ of
Hamilton quaternions).
We define the discriminant of $D$ by 
$\disc(D)=\prod_{p\in\Ram(D)}p$.
The quaternion algebra $D$
corresponding to a fixed square-free discriminant
can be constructed explicitly.
In this paper we only consider the following
two kinds of discriminants.

\begin{Lemma}\label{AnotherPresentation}
Let $N$ be a square-free integer with an odd number of prime divisors.
\begin{enumerate}[(1)]
\item If $N$ has no prime divisor of the form $4n+1$,
$(\frac{-1,-N}{\Q})$ is the definite quaternion algebra
over $\Q$ with discriminant $N$;
\item If $N$ has no prime divisor of the form $3n+1$,
$(\frac{-3,-N}{\Q})$ is the definite quaternion algebra
over $\Q$ with discriminant $N$.
\end{enumerate}
\end{Lemma}

\begin{proof}
It is easy to prove this lemma for $p\nmid N$ and $p\neq 2,3$,
since $(\frac{a,b}{\Q_p})$ is split
if $p$ is unramified in $\Q(\sqrt{a})$ and $v_p(b)=0$
(cf. \cite[Corollary 13.4.1]{voight2017quaternion}).
For $p\mid N$ and $p\neq 2,3$ the lemma also holds,
noticing that
$(\frac{a,\varpi}{\Q_p})$ is the only 
non-split quaternion algebra over $\Q_p$
(up to isomorphism)
if $\varpi$ is a uniformizer of $\Q_p$
and $a\in\Z_p^\times$ is an element such that
$\Q_p(\sqrt{a})$ is the unramified quadratic extension of $\Q_p$
(cf. \cite[Theorem 13.3.11]{voight2017quaternion}).

A more detailed proof can be given 
by calculating the Hilbert symbol.
Recall that, for a local field $F$,
the Hilbert symbol $(\cdot,\cdot)_F:
F^\times/(F^\times)^2\times F^\times/(F^\times)^2
\to\{\pm 1\}$ can be defined by 
that $(a,b)_F=1$ if $(\frac{a,b}{F})\cong M(2,F)$ is split,
and $=-1$ if $(\frac{a,b}{F})$ is a division algebra.
It can be calculated by the following algorithm
(cf. \cite{serre2012course}):

\begin{itemize}
	\item For $F=\R$, 
	$(a,b)_{\R}=1$ if $a$ or $b$ is $>0$, 
and $=-1$ if $a$ and $b$ are $<0$.
\item
For $F=\Q_p$, if we write $a,b$ in the form $p^\alpha u,p^\beta v$ 
where $u,v\in\Z_p^\times$, we have
\[(a,b)_{\Q_p}=\begin{cases}
(-1)^{\alpha\beta\varepsilon(p)}(\frac up)^\beta(\frac vp)^\alpha
&\text{if }p\neq 2,\\
(-1)^{\varepsilon(u)\varepsilon(v)+\alpha\omega(v)+\beta\omega(u)}
&\text{if }p= 2.\end{cases}\]
Here $(\frac up)$ denotes the Legendre symbol, 
and $\varepsilon$, $\omega$ are defined by
\[\varepsilon(z)=
\begin{cases}
0&\text{if }z\equiv 1\pmod 4,\\1&\text{if }z\equiv 3\pmod 4;
\end{cases}\qquad
\omega(z)=
\begin{cases}
0&\text{if }z\equiv \pm 1\pmod 8,\\1&\text{if }z\equiv \pm 5\pmod 8.
\end{cases}\]
\end{itemize}

We will complete the proof of (1) for example,
that is to show $(-N,-1)_{\Q_2}=(-1)^{\ord_2(N)}$.
By our assumption on $N$, if $2\nmid N$, 
$N$ is the product of an odd number of primes of the form $4n+3$,
and therefore $-N\equiv 1 \pmod 4$,
\[(-N,-1)_{\Q_2}=(-1)^{\varepsilon(-N)\varepsilon(-1)}=1;\]
if $2\mid N$, $\frac N2$ is the product of 
an even number of primes of the form $4n+3$,
and then $-\frac N2\equiv -1 \pmod 4$,
\[(-N,-1)_{\Q_2}=(-1)^{\varepsilon(-N/2)\varepsilon(-1)+\omega(-1)}=-1.\]
\end{proof}

For a quaternion algebra $D$ defined over $\Q$,
the maximal orders of quaternion algebras 
$D_p=D \otimes_\Q \Q_p$
can be described as follows.

\begin{Proposition}
[\cite{voight2017quaternion}]\label{maxorder}\
When $D_p = M(2,\Q_p)$, the maximal orders of $D_p$ are the
$\GL(2,\Q_p)$-conjugates of $M(2,\Z_p)$. 
When $D_p$ is not split,
there is a unique maximal order
\[
\cO_p=\{x\in D_p:N_{D_p}(x)\in\Z_p\}\]
which contains all $\Z_p$-integral elements of $D_p$.
\end{Proposition}
In this paper we fix a maximal order
$\cO$ of $D$ such that
$\cO_p=\cO\otimes_{\Z}\Z_p$ is 
the unique maximal order for $p\mid\disc(D)$,
and $\cO_p$ is the preimage of $M(2,\Z_p)$ under a certain isomorphism
$\cO_p\iso M(2,\Q_p)$ for $p\nmid\disc(D)$
(cf. Sections \ref{section.Waldspurger.split} and \ref{section.compatibility}).

\subsection{Normalization of measures}\label{section.measure}

We recall the local and global Haar measures on 
number fields, quaternion algebras and their quotients.
The normalization constants of the measures follow 
\cite[Chapter 29]{voight2017quaternion} and \cite[Section 1.6]{yuan2013gross}.

Let $F$ be a number field.
For a finite place $v$ of $F$,
the ring of integers in $F_v$ is denoted by $\cO_{F_v}$.
Let $\varpi_v$ denote a uniformizer in $F_v$
and $q_v=\#(\cO_{F_v}/(\varpi_v))$ be the order of the residue field.

Fix an additive character $\psi$ of $F\backslash\A_F$. For
a place $v$ of $F$ we take the additive Haar measure
$dx_v$ on $F_v$ which is self-dual with respect
to $\psi_v$.
On $F_v^\times$ we take the measure
\[d^\times x_v = \zeta_{F_v}(1) \frac{dx_v}{|x|_v}
= \begin{cases}
\dfrac{dx}{|x|},&F_v=\R;\\
\pi^{-1}\dfrac{2 dx_0 dx_1}{x\bar x},&F_v=\CC,\ x=x_0+x_1i;\\
(1-q_v^{-1})^{-1} \dfrac{dx_v}{|x|_v},&v<\infty.
\end{cases}\]
Let $E$ be a quadratic field extension over $F$.
We define measures on $E_v=F_v\otimes_F E$ and $E_v^\times$ 
similarly with respect to the additive character $\psi \circ \Tr_{E/F}$.
We note that with these choices of measures we have,
for $v<\infty$,
\[\vol(\cO^\times_{F_v}; d^\times x_v) =\vol(\cO_{F_v}; dx_v)
= |\fd_{v}|^{1/2}\]
with $\fd_{v}\in F_v$
such that $\fd_{v}\cO_{F_v}$ is the different of $F_v$ over $\Q_p$;
and for a quadratic field extension $E_v/F_v$,
we have
\[\vol(F_v^\times\backslash E_v^\times) =
\begin{cases}2,&\text{if }F_v=\R,\ E_v=\CC;\\
|\fd_v|^{1/2},&\text{if $E_v/F_v$ is the unramified field extension};\\
2|\fD_v\fd_v|^{1/2},&\text{if $E_v/F_v$ is ramified},\end{cases}\]
with $\fD_v\in\cO_{F_v}$ such that
$\fD_{v}\cO_{F_v}$ is the relative discriminant of $E_v/F_v$.

For a quaternion algebra $D$ defined over a number field $F$,
fix a maximal order $\cO\subset D$.
For a finite place $v$ of $F$ we take the Haar measure
$dg_v$ on $D^\times_v$ as
\[dg_v=\zeta_{F_v}(1)|N_{D_v}(g_v)|^{-2}\ d\mu_v(g_v)\]
where $\mu_v$ is the additive Haar measure on $D_v$
which is self-dual with respect to $\psi_v$.
Let $K_v$ be the image of $Z(F_v) \cO_v^\times $ in 
$G_v'=Z(F_v)\backslash D^\times_v$.
Then, with the quotient measure defined on $G_v'$,
for $v<\infty$, we have
\[\vol(K_v) =|\fd_{v}|^{\frac 32}\zeta_{F_v}(2)^{-1}\cdot\begin{cases}
(q_v-1)^{-1},&\text{if }v\in\Ram(D),\\
1,&\text{if }v\notin\Ram(D).\end{cases}\]

For $v \mid\infty$ and a definite quaternion algebra $D$,
$D_v\cong\mathbb{H}$ and
$G'_v=Z(\R)\backslash \mathbb{H}^\times\cong
\{\pm 1\}\backslash \SU(2)$.
We parametrize
$h=\left(\begin{matrix}\alpha&-\beta\\\bar\beta&\bar\alpha\end{matrix}\right)
\in\SU(2)$
by setting
\[\alpha=re^{i\theta}\cos\gamma,\quad \beta=re^{i\varphi}\sin\gamma,\qquad
r=1,\ 0\leq\gamma\leq\frac\pi 2,\ 0\leq \theta,\varphi<2\pi.\]
So $d\alpha d\beta=2\sin2\gamma\ d\gamma d\theta d\varphi$
(notice that $d\alpha d\beta$ is the self-dual additive measure on $\mathbb{H}$),
and for a function $\Phi\in L^1(\SU(2))$,
\[\int_{\SU(2)}\Phi(h)\ dh=\int_0^{2\pi}\int_0^{2\pi}\int_0^{\frac\pi 2}
\Phi(\gamma, \theta ,\varphi)\cdot 2\sin2\gamma\ d\gamma\ d\theta\ d\varphi.\]
This choice of Haar measure on $\SU(2)$ implies 
$\vol(G'_v)=\vol(G'_\infty)=4\pi^2$ by direct calculation.

Globally we take the product of these local measures
and give discrete subgroups the counting measures,
and define the Tamagawa measure on
\[[E^\times]
=\A_F^\times E^\times \backslash \A_E^\times
\quad\text{and}\quad
[D^\times]
=Z(\A_F)D^\times(F)\backslash D^\times(\A_F)\]
as the quotient measure.
In this way we get
\begin{equation}\label{quadextvol}
\vol([E^\times]) = 2 L(1, \eta_{E/F})\quad
\text{and}\quad
\vol([D^\times]) = 2.
\end{equation}
Here $\eta_{E/F}$ is the quadratic character of $F^\times\backslash \A_F^\times$
associated to $E/F$ by class field theory.
For example, when $E=\Q(\sqrt{d})$ such that $d$ is a fundamental discriminant,
$\eta_{E/F}$
is the Hecke character corresponding to
the Dirichlet character $\chi_d$ such that
\[\chi_{d}(p)=\left(\frac{d}{p}\right)=\begin{cases}
1,&\text{$p$ splits in $E$;}\\
-1,&\text{$p$ remains prime in $E$;}\\
0,&\text{$p$ ramifies in $E$.}\end{cases}\]

We end this section by recalling a useful lemma about the quotient measure.

\begin{Lemma}[{\cite[Corollary 7.14]{knightly2006traces}}]\label{productmeasure}
Let $G$ be a unimodular group and
suppose $G=KH$ for closed unimodular subgroups $H$ and $K$.
Suppose further that $K\cap H$ is unimodular.
Let $dh$ denote a right $H$-invariant measure on $(K\cap H)\backslash H$.
Then, for $f\in C_c(KH)$,
\[\int_{(K\cap H)\backslash H}\int_{K}f(kh) \ dk \ dh\]
defines a Haar measure on $G=KH$.
Moreover, with this measure on $G$,
\[\int_{K\backslash G}f(g)\ dg=\int_{(K\cap H)\backslash H}f(h)\ dh\]
for all $f\in C_c(K\backslash G)$.
\end{Lemma}

\subsection{Jacquet--Langlands correspondence}\label{section.JLcorrespondence}

Let $D$ be a quaternion algebra over $\Q$.
Define an algebraic group $D^\times$ over $\Q$ by $D^\times(A)=(A\otimes_{\Q} D)^\times$
for any $\Q$-algebra $A$.
Thus $D^\times$ is a reductive algebraic group and
we therefore have a theory of automorphic forms and representations of $D^\times$.
We will be more interested in the forms that correspond to
some automorphic forms on $\GL(2,\Q)$ via the Jacquet--Langlands correspondence.

Let $N$ be a square-free integer
with an odd number of prime factors.
Fix a positive integer $k$.
Denote by $\cF_{2k}(N)$ the set of normalized cusp newforms
of weight $2k$ on $\Gamma_0(N)$
which are Hecke eigenforms,
and by $\cF(N,2k)$ the set of cuspidal automorphic representations
of $\PGL(2,\A)$ of level $N$ and weight $2k$.
The following theorem shows the 1-1 correspondence between
$\cF_{2k}(N)$ and $\cF(N,2k)$.

\begin{Theorem}[\cite{gelbart1975automorphic,
watson2008rankin}]\label{cuspformtorep}
Suppose $N=\prod p_i$ is a product of distinct primes.
If $f(z)=\sum_{n\geq 1}a_n(f)e^{2\pi inz}\in \cF_{2k}(N)$
(normalized such that $a_1=1$), 
its corresponding cuspidal automorphic representation
$\pi_f=\otimes_v\pi_v$ of $\PGL(2,\A)$ can be described as follows:
\begin{itemize}
\item $\pi_\infty\cong\pi_{\mathrm{dis}}^{2k}
=\sigma(|\cdot|^{k-1/2},|\cdot|^{-(k-1/2)})$
is the discrete series representation of weight $2k$;
\item if $p\nmid N$, $\pi_p$ is the spherical representation $\pi(\mu_1,\mu_2)$
such that $\mu_1\mu_2$ is trivial
and $a_p(f)=p^{\frac{2k-1}{2}}(\mu_1(p)+\mu_2(p))$; and
\item if $p\mid N$, $\pi_p$ is the special representation
$\sigma_\delta$ of $\GL(2,\Q_p)$
with trivial central character,
where $\delta$ is the unramified character of $\Q_p^\times$
with $\delta(p)=a_p(f)p^{-(k-1)}=\pm 1$.
\end{itemize}
\end{Theorem}

Let $D$ be the definite quaternion algebra with discriminant $N$
(i.e. the quaternion algebra defined over $\Q$ which is ramified
precisely at the infinite place of $\Q$ and the primes dividing $N$).
We have taken $G'$ to be the algebraic group defined over $\Q$
with $G'(\Q) =Z(\Q)\backslash D^\times(\Q)$.
Denote by $\cA(G')$ the set of
irreducible automorphic representations of $G'(\A)$.
Since the quotient $G'(\Q)\backslash G'(\A)$ is compact,
we have the decomposition
\[L^2(G'(\Q)\backslash G'(\A))
= \widehat{\bigoplus}_{\pi'\in\mathcal A(G')} V_{\pi'},\]
where $V_{\pi'}$ denotes the space of $\pi'$.

Clearly $\cA(G')$ contains all the characters of $G'(\Q)\backslash G'(\A)$,
which are of the form $\delta \circ N_D$ where
$\delta:\Q^\times \backslash \A^\times \to \{\pm 1\}$
is a quadratic Hecke character.
Let $\cA_{res}(G')$ be the set of these characters.
Then its orthogonal complement, denoted by $\cA_{cusp}(G')$,
contains all the infinite-dimensional irreducible
automorphic representations of $G'(\A)$.
Every representation $\pi'\in\cA_{cusp}(G')$,
according to Jacquet--Langlands \cite{jacquet2006automorphic},
corresponds to a cuspidal automorphic representation of $\PGL(2, \A)$.

Let $\cF'(N, 2k)$ be the set of representations $\pi' \in \cA_{cusp}(G')$
which map to representations in $\cF(N, 2k)$
under the Jacquet--Langlands correspondence.
The compatibility between the local and global
Jacquet--Langlands correspondence gives the following theorem,
which describes explicitly $\pi'=\otimes\pi_v'\in\cF'(N, 2k)$.

\begin{Theorem}
[\cite{jacquet2006automorphic} Jacquet--Langlands correspondence,
as recalled in {\cite[Fact 3.1]{feigon2009averages}}]\label{factJL}
Under the Jacquet--Langlands correspondence $\JL: \mathcal A(G')
\hookrightarrow \mathcal A(\PGL(2))$,
the image of $\cA_{cusp}(G')$
is equal to the set of cuspidal automorphic representations
$\pi = \otimes_v \pi_v$ of $\PGL(2, \A)$
such that $\pi_v$ is a square-integrable representation
of $\PGL(2, \Q_v)$ at all places $v$ where $D$ is ramified.
In particular, when $D$ is definite and has discriminant $N$,
for $\pi\in\cF(N, 2k)$,
there exists $\pi'=\otimes\pi_v'\in\cA_{cusp}(G')$
such that $\JL(\pi')=\pi$ and
\begin{enumerate}
	\item $\pi_\infty\cong\pi_{\mathrm{dis}}^{2k}$,
	$\pi'_\infty \cong \pi_{2k}'$ is a $(2k-1)$-dimensional
	irreducible representation of
	$G'_\infty=Z(\R)\backslash{\bH}^\times$
	(which is defined in Section \ref{RepresentationSU2});
	\item for $v=p\mid N$, $\pi_p$ is the special representation
	$\sigma_{\delta_p}$
	where $\delta_p : \Q_p^\times \to \{\pm 1\}$
	is an unramified character,
	$\pi'_p \cong \delta_p \circ N_{D_p}$ is a character of $G_p'$; and
	\item for all the other $v$, $\pi_v$ is unramified,
	$\pi'_v\cong\pi_v$.
\end{enumerate}
\end{Theorem}

For $\pi'\in \cF'(N, 2k)$, the following lemma defines
a new-line vector $\phi\in V_{\pi'}$.

\begin{Lemma}\label{autoset}
Fix a maximal order $\cO\subset D$ such that
$\cO_p=M(2,\Z_p)$ whenever $D$ splits at $p$.
For any $p<\infty$,
let $K_p$ be the image of $Z(\Q_p)\cO_p^\times$ in $G'_p=G'(\Q_p)$,
and $K_{\fin}=\prod_p K_p$ be
an open subgroup of $G'_{\fin}=\prod_p G'_p$.
Then (with $X^{2k-2}\in V_{\pi'_{2k}}$
being the unit highest weight vector defined in Section \ref{RepresentationSU2})
\[\CC X^{2k-2}\otimes(\pi'_{\fin})^{K_{\fin}}\]
is a one-dimensional subspace of $V_{\pi'}$ for $\pi'\in \cF'(N, 2k)$.
We call any nonzero vector $\phi$ in this subspace
a new-line vector in $\pi'$, and write it as
\[\phi=\otimes_v\phi_v\quad\text{with}\quad
\phi_\infty=\|\phi\|X^{2k-2}\]
and $\phi_p$ being the unit spherical vector in $\pi'_p$
(we fix a $G_p'$-invariant bilinear form
on $\pi'_p\otimes \tilde\pi'_p$).
\end{Lemma}

\begin{proof}
When $p\mid N$, the unramifiedness of $\delta_p$ implies that
$\delta_p\circ N_{D_p}$ is $K_p$-invariant.
Then this Lemma is a direct result from Theorem \ref{factJL}.
\end{proof}

In particular, when $2k=2$,
$\pi'_\infty\cong \Sym^0 V \otimes \det^0$ is trivial.
So every automorphic form in $\pi'\in\cF'(N, 2)$,
in particular the new-line vector,
can be seen as a function defined on $G'(\A_{\fin})$.

\subsection{Representation theory of \texorpdfstring{$\SU(2)$}{SU(2)}}
\label{RepresentationSU2}

Notice that $G_\infty'=\R^\times\backslash \mathbb{H}^\times
\cong \SU(2)/\{\pm 1\}$. 
Set $\pi'_{2k}=\Sym^{2k-2} V \otimes {\det}^{-k+ 1}$, 
where $V$ denotes the irreducible 2-dimensional representation 
of $G'_\infty$ coming from the isomorphism 
$D^\times(\overline{\R}) \iso \GL(2, \CC)$ (see \eqref{archi.embed}).
We have $\dim \pi'_{2k}=2k-1$. 
More explicitly,
$\pi'_{2k}$ can be realized on the space 
of homogeneous polynomials in $X,Y$ of degree $2k-2$, i.e.
\[V_{\pi'_{2k}}=\bigoplus_{n=0}^{2k-2}\CC X^nY^{2k-2-n}\]
with
\[\pi'_{2k} (g)P(X,Y)=P((X,Y)g)\det(g)^{1-k} 
\quad \text{for}\quad g\in 
\left\{\left(\begin{matrix}\alpha&-\beta\\\bar\beta&\bar\alpha\end{matrix}\right)\in\GL(2,\CC)\right\}
\cong D^\times(\R).\]
One can check that 
\begin{equation}\label{innerproduct}
\langle X^iY^{2k-2-i},X^jY^{2k-2-j}\rangle_{2k}
=\begin{cases}\binom{2k-2}{i}^{-1},&\text{if }i=j;\\
0,&\text{otherwise}
\end{cases}\end{equation}
defines a $G'_\infty$-invariant 
inner product on $V_{\pi'_{2k}}$ 
(cf. \cite[Eqn (5.4)]{chen2018deligne} or \cite[Lemma B.0.1]{guan2020averages}).

One can check that
$\pi'_{2k}$ is an irreducible representation with highest weight $2k-2$,
and $X^{2k-2}$ is a highest weight vector.
As in Lemma \ref{autoset},
let $\phi_\infty$ be the highest weight vector $\|\phi\|X^{2k-2}$.
We have $\langle\phi,\phi\rangle=\prod_v\langle\phi_v,\phi_v\rangle_v$
since the length of $\phi_v$ is assumed to be 1 for any $v<\infty$.

Denote by $\Delta_2$ (resp. $\Delta_3$) 
the diagonal embedding from $G'_\infty$ 
to two (resp. three) copies of itself.
One can view ${\pi'_{2k}}^{\otimes 2}\circ\Delta_2$
and ${\pi'_{2k}}^{\otimes 3}\circ\Delta_3$
as representations of $G'_\infty$.
Denote by \[\{X_1^iY_1^{2k-2-i}\otimes X_2^jY_2^{2k-2-j}\}\quad 
(\text{resp. }\{X_1^iY_1^{2k-2-i}\otimes X_2^jY_2^{2k-2-j}\otimes 
X_3^rY_3^{2k-2-r}\})\] 
a basis of ${\pi'_{2k}}^{\otimes 2}=\pi'_{2k}\otimes\pi'_{2k}$
(resp. ${\pi'_{2k}}^{\otimes 3}$).
\cite{chen2018deligne} shows that,
\[\mathbb{P}_{2k}=
\det\left(\begin{matrix}X_1&X_2\\Y_1&Y_2\end{matrix}\right)^{k-1}\otimes
\det\left(\begin{matrix}X_2&X_3\\Y_2&Y_3\end{matrix}\right)^{k-1}\otimes
\det\left(\begin{matrix}X_3&X_1\\Y_3&Y_1\end{matrix}\right)^{k-1}\]
is the only $G'_\infty$-invariant vector in 
${\pi'_{2k}}^{\otimes 3}\circ\Delta_3$ 
up to a constant multiple. 

Let $\langle\cdot,\cdot\rangle$ be the $D^\times(\Q)^2$ or 
$D^\times(\Q)^3$-invariant pairing 
on ${\pi'_{2k}}^{\otimes 2}$
or ${\pi'_{2k}}^{\otimes 3}$
given by \begin{equation}\label{innerproduct2}
\langle\cdot,\cdot\rangle=
\langle\cdot,\cdot\rangle_{2k}
\otimes\langle\cdot,\cdot\rangle_{2k}\quad\text{or}
\quad\langle\cdot,\cdot\rangle_{2k}
\otimes\langle\cdot,\cdot\rangle_{2k}\otimes\langle\cdot,\cdot\rangle_{2k}.
\end{equation}
We can calculate the lengths of some particular vectors.

\begin{Lemma}\label{length}
Let \[w^\circ_{2k}=
\left(-Y_1Y_2\left|\begin{matrix}X_1&X_2\\
Y_1&Y_2\end{matrix}\right|\right)^{k-1};\qquad
\mathbb{P}_{2k}=
\left|\begin{matrix}X_1&X_2\\Y_1&Y_2\end{matrix}\right|^{k-1}\otimes
\left|\begin{matrix}X_2&X_3\\Y_2&Y_3\end{matrix}\right|^{k-1}\otimes
\left|\begin{matrix}X_3&X_1\\Y_3&Y_1\end{matrix}\right|^{k-1}.\]
Then \[\|w^\circ_{2k}\|^2=\langle w^\circ_{2k},w^\circ_{2k}\rangle
=\frac{\Gamma(k)^3\Gamma(3k-1)}{\Gamma(2k-1)^2\Gamma(2k)},
\quad 
\|\mathbb{P}_{2k}\|^2=\langle \mathbb{P}_{2k},\mathbb{P}_{2k}\rangle
=\frac{\Gamma(k)^3\Gamma(3k-1)}{\Gamma(2k-1)^3}.\]
In particular
$\|\mathbb{P}_{2k}\|^2=(2k-1)\|w^\circ_{2k}\|^2$.
\end{Lemma}
\begin{proof}
By definition
\[\begin{split}
\langle w^\circ_{2k},w^\circ_{2k}\rangle
=&\ \left\langle\sum_{r=0}^{k-1}\binom{k-1}{r}
X_1^r(-Y_1)^{2k-2-r}X_2^{k-1-r}Y_2^{k-1+r},\right.\\
&\phantom{\sum_{r=0}^{k-1}\binom{k-1}{r}X_1^rX_1^r(-Y_1)^{2k-2-r}}
\left.\sum_{r'=0}^{k-1}\binom{k-1}{r'}
X_1^{r'}(-Y_1)^{2k-2-r'}X_2^{k-1-r'}Y_2^{k-1+r'}\right\rangle\\
=&\ \sum_{r=0}^{k-1}\binom{k-1}{r}^2
\langle X_1^rY_1^{2k-2-r},X_1^rY_1^{2k-2-r}\rangle_{2k}
\langle X_2^{k-1-r}Y_2^{k-1+r},X_2^{k-1-r}Y_2^{k-1+r}\rangle_{2k}\\
=&\ \sum_{r=0}^{k-1}\binom{k-1}{r}^2\binom{2k-2}{r}^{-1}\binom{2k-2}{k-1-r}^{-1}.
\end{split}\]
One can verify that
\[\binom{k-1}{r}^2\binom{2k-2}{r}^{-1}\binom{2k-2}{k-1-r}^{-1}
=\binom{2k-2}{k-1}^{-2}\binom{k-1+r}{k-1}\binom{2k-2-r}{k-1}.\]
Therefore, with Lemma \ref{sum.binom}
we have \[
\langle w^\circ_{2k},w^\circ_{2k}\rangle
=
\frac{\sum_{r=0}^{k-1}\binom{k-1+r}{k-1}\binom{2k-2-r}{k-1}}
{\binom{2k-2}{k-1}^{2}}
=
\frac{\binom{3k-2}{k-1}}{\binom{2k-2}{k-1}^{2}}
=\frac{\Gamma(k)^3\Gamma(3k-1)}{\Gamma(2k-1)^2\Gamma(2k)}.
\]

The length of $\mathbb{P}_{2k}$ is calculated 
in \cite[Proposition 5.1]{chen2018deligne}.
\end{proof}

\begin{Lemma}\label{sum.binom}
For any $m,n\geq 1$
\[\sum_{r=0}^n\binom{n+r}{n}\binom{m+n-r}{n}=\binom{2n+m+1}{m}.\]
\end{Lemma}
\begin{proof}
Assume that a point moves from $(0,0)$ to $(2n+1,m)$ 
by moving up or to the right by one unit each time. 
Then the point has $\binom{2n+m+1}{m}$ possible paths.
But the path can intersect the vertical line $x=n+1/2$ only once, say it passes 
$(n,r)$ and $(n+1,r)$ where $0\leq r\leq m$.
Then $\binom{n+r}{n}\binom{n+m-r}{n}$ is the number of 
all possible paths that the point moves from 
$(0,0)$ to $(n,r)$ and then from $(n+1,r)$ to $(2n+1,m)$.
While $r$ varies from $0$ to $m$, all possible paths are counted.
\end{proof}

Now we recall two lemmas
in the representation theory of compact groups.
The first lemma can be obtained through direct calculation.

\begin{Lemma}\label{Schur}
Let $K$ be a compact topological group with Haar measure $dk$,
$\Pi$ be a unitary representation (might not be irreducible) of $K$.
Define \[P_\Pi(v)=\frac 1{\vol(K;dk)}\int_K\Pi(k)v\ dk,\quad v\in V_\Pi.\]
Then $P_\Pi$ is the projection map from $V_\Pi$ to
its $K$-invariant subspace $V_\Pi^K$,
and \[\int_K\langle \Pi(k)u,v\rangle\ dk=\vol(K;dk)
\langle P_\Pi(u),P_\Pi(v)\rangle\]
where $\langle \ , \ \rangle$ is a $K$-invariant 
inner product defined on $V_\Pi$.
\end{Lemma}

\begin{Lemma}[\cite{knapp2001representation}
Schur Orthogonality Relations]\label{SchurOrtho}
Let $K$ be a compact Lie group,
$\pi,\pi'$ be two finite-dimensional irreducible unitary representations of $K$,
$\langle \ , \ \rangle$ be a $K$-invariant inner product of $\pi$ or $\pi'$.
Then, for $u,v\in V_\pi$, $u',v'\in V_{\pi'}$,
\[\int_K\langle \pi(k)u,v\rangle \overline{\langle \pi'(k)u',v'\rangle}\ dk=
\begin{cases}
0,&\text{if }\pi\not\cong \pi';\\
\vol(K;dk)\dfrac{\langle u,u'\rangle\overline{\langle v,v'\rangle}}
{\dim(V_\pi)},&\text{if }\pi=\pi'.
\end{cases}\]
\end{Lemma}

At last we prove a lemma 
which we use to motivate our choice of the test function 
in Section \ref{ChooseTest} 
(essentially the choice of $w^\circ_{2k}$).
The following lemma implies that
\[\int_{G'_\infty}
\langle\pi'_{2k}\otimes\pi'_{2k}(g,g)w^\circ_{2k},
w^\circ_{2k}\rangle\pi'_{2k}(g)X^{2k-2}\ dg\]
is a constant multiple of $X^{2k-2}$; more explicitly,
\[\int_{G'_\infty}
\langle\pi'_{2k}\otimes\pi'_{2k}(g,g)w^\circ_{2k},
w^\circ_{2k}\rangle\pi'_{2k}(g)X^{2k-2}\ dg=
\vol(G'_\infty)\frac{\|w^\circ_{2k}\|^4}{\|\mathbb{P}_{2k}\|^2}X^{2k-2}.\]
One can understand the above identity in this way:
the representation ${\pi'_{2k}}^{\otimes 2}\circ \Delta_2$ of $G_\infty'$
can be decomposed as 
\[\pi'_{4k-2}\oplus\pi'_{4k-4}\oplus\cdots\oplus\pi'_{4}\oplus\pi'_{2}\]
(with highest weight $4k-4,4k-6,\ldots,2,0$ respectively).
In particular there is exactly one copy of $\pi'_{2k}$ 
in this decomposition, with $w^\circ_{2k}$
being its lowest weight vector.

\begin{Lemma}\label{Choosew}
With $w^\circ_{2k}$ and $\mathbb{P}_{2k}$ defined in Lemma \ref{length},
we have \\(1) 
$\langle w^\circ_{2k}\otimes X_3^{2k-2-i}Y_3^i,\mathbb{P}_{2k}\rangle$
vanishes when $i\neq 0$, and
\[\langle w^\circ_{2k}\otimes X_3^{2k-2},\mathbb{P}_{2k}\rangle
=\langle w^\circ_{2k},w^\circ_{2k}\rangle.\]
(2)\[
\left\langle \int_{G'_\infty}
\langle\pi'_{2k}\otimes\pi'_{2k}(g,g)w^\circ_{2k},w^\circ_{2k}\rangle\pi'_{2k}(g)X^{2k-2}\ dg,
X^{2k-2-i}Y^i\right\rangle
\]
vanishes unless $i=0$, in which case it is equal to $\vol(G'_\infty)
\|w^\circ_{2k}\|^4/\|\mathbb{P}_{2k}\|^2$
\end{Lemma}

\begin{proof}
(1) Here $w^\circ_{2k}$ has nothing to do with $X_3$, $Y_3$. 
So only the terms in $\mathbb{P}_{2k}$
with $X_3^{2k-2-i}Y_3^i$ contribute to the inner product 
$\langle w^\circ_{2k}\otimes X_3^{2k-2-i}Y_3^i,\mathbb{P}_{2k}\rangle$,
i.e. the inner product is equal to
\[\begin{split}
&\ \left\langle \left(-Y_1Y_2\left|\begin{matrix}X_1&X_2\\Y_1&Y_2\end{matrix}\right|\right)^{k-1} X_3^{2k-2-i}Y_3^i,\right.\\
&\left.\phantom{\sum}
\left|\begin{matrix}X_1&X_2\\Y_1&Y_2\end{matrix}\right|^{k-1}X_3^{2k-2-i}Y_3^i
\sum_{r_1+r_2=i}
\binom{k-1}{r_1}(X_2)^{r_1}(-Y_2)^{k-1-r_1}\binom{k-1}{r_2}(Y_1)^{k-1-r_2}(-X_1)^{r_2}
\right\rangle\\
=&\ \left\langle \left(-Y_1Y_2\left|\begin{matrix}X_1&X_2\\Y_1&Y_2\end{matrix}\right|\right)^{k-1} \right.,\\
&\left.\phantom{\sum X_3^{2k-2-i}Y_3^i}
\left|\begin{matrix}X_1&X_2\\Y_1&Y_2\end{matrix}\right|^{k-1}
\sum_{r_1+r_2=i}
\binom{k-1}{r_1}(X_2)^{r_1}(-Y_2)^{k-1-r_1}\binom{k-1}{r_2}(Y_1)^{k-1-r_2}(-X_1)^{r_2}
\right\rangle\\
&\ \cdot\langle X_3^{2k-2-i}Y_3^i,X_3^{2k-2-i}Y_3^i\rangle.
\end{split}\]
Notice that the sum of exponents of $X_1,X_2$ in the first term 
$w^\circ_{2k}=\big(-Y_1Y_2(X_1Y_2-X_2Y_1)\big)^{k-1} $
is always $k-1$,
while that in the second term is always $k-1+i$.
So the inner product is 0 unless $i=0$.

When $i=0$, we see that 
\[\begin{split}
&\ \langle w^\circ_{2k}\otimes X_3^{2k-2},\mathbb{P}_{2k}\rangle\\
=&\ \left\langle \left(-Y_1Y_2\left|\begin{matrix}X_1&X_2\\Y_1&Y_2\end{matrix}\right|\right)^{k-1} ,\right.\\
&\left.\phantom{X_3^{2k-2-}Y_3^i}
\left|\begin{matrix}X_1&X_2\\Y_1&Y_2\end{matrix}\right|^{k-1}
\sum_{r_1+r_2=0}
\binom{k-1}{r_1}(X_2)^{r_1}(-Y_2)^{k-1-r_1}\binom{k-1}{r_2}(Y_1)^{k-1-r_2}(-X_1)^{r_2}
\right\rangle\\
&\ \cdot\langle X_3^{2k-2},X_3^{2k-2}\rangle\\
=&\ \left\langle \left(-Y_1Y_2\left|\begin{matrix}X_1&X_2\\Y_1&Y_2\end{matrix}\right|\right)^{k-1} ,
\left|\begin{matrix}X_1&X_2\\Y_1&Y_2\end{matrix}\right|^{k-1}
(-Y_2)^{k-1}(Y_1)^{k-1}\right\rangle\cdot 1=\langle w^\circ_{2k},w^\circ_{2k}\rangle.
\end{split}\]

(2) We have that
\[\begin{split}
& \ \left\langle \int_{G'_\infty}
\langle\pi'_{2k}\otimes\pi'_{2k}(g,g)w^\circ_{2k},w^\circ_{2k}\rangle\pi'_{2k}(g)X^{2k-2}\ dg,
X^{2k-2-i}Y^i\right\rangle\\
=&\ \int_{G'_\infty}
\langle\pi'_{2k}\otimes\pi'_{2k}(g,g)w^\circ_{2k},w^\circ_{2k}\rangle
\langle \pi'_{2k}(g)X^{2k-2},
X^{2k-2-i}Y^i\rangle\ dg\\
=&\ \int_{G'_\infty}
\langle{\pi'_{2k}}^{\otimes 3}\circ\Delta_3(g)w^\circ_{2k}\otimes X_3^{2k-2},
w^\circ_{2k}\otimes X_3^{2k-2-i}Y_3^i\rangle\ dg.\end{split}\]
Recall that $\mathbb{P}_{2k}$ is
the only $G'_\infty$-invariant vector in 
${\pi'_{2k}}^{\otimes 3}\circ\Delta_3$ 
up to a constant multiple. 
Hence, by Lemma \ref{Schur}, 
the above integral is equal to
\[\vol(G'_\infty)
\langle w^\circ_{2k}\otimes X_3^{2k-2},
\frac{\mathbb{P}_{2k}}{\|\mathbb{P}_{2k}\|}\rangle
\overline{\langle w^\circ_{2k}\otimes X_3^{2k-2-i}Y_3^i,
\frac{\mathbb{P}_{2k}}{\|\mathbb{P}_{2k}\|}\rangle}.\]
The value of
$\langle w^\circ_{2k}\otimes X_3^{2k-2-i}Y_3^i,\mathbb{P}_{2k}\rangle$
completes the proof.
\end{proof}

\subsection{\texorpdfstring{$L$}{L}-functions and adelic version of Main Theorem}
\label{section.L-functions}

Let $F$ be a local field.
According to the local Langlands correspondence,
for every irreducible admissible representation $\pi$ of $\GL(2,F)$,
there is a representation $\rho:W_F\to\GL(2,\CC)$ of the Weil group
such that $L(s,\rho)=L(s,\pi)$.
The triple product local $L$-factor can be defined by
(cf. \cite[Section 3.1]{watson2008rankin},
\cite[Section 2.4]{woodbury2012explicit})
\begin{gather*}
L(s,\pi_1\otimes\pi_2\otimes\pi_3)=
L(s,\rho_1\otimes\rho_2\otimes\rho_3),\\
L(s,\pi_1\otimes\pi_2\otimes\pi_3,\Ad)=
L(s,\oplus_i\Ad(\rho_i))=\prod_{i}L(s,\Ad(\rho_i))
=\prod_{i}L(s,\pi_i,\Ad),
\end{gather*}
where $\Ad(\rho_i):W_F\to\GL(3,\CC)$ is the adjoint representation.
For the cases at hand, we can define the local $L$-factors explicitly.

Let $F=\R$ or $\CC$ be an Archimedean local field.
Recall that, for $s\in\CC$,
\[\zeta_{\R}(s)=\pi^{-s/2}\Gamma(s/2),\quad
\zeta_{\CC}(s)=\zeta_{\R}(s)\zeta_{\R}(s+1)=2(2\pi)^{-s}\Gamma(s),\]
where $\Gamma(s)$ is the standard $\Gamma$-function.
For a character $\mu:F^\times\to\CC^\times$,
define
\[L(s,\mu)=\begin{cases}
\zeta_{\R}(s+r+m),&\text{when }F=\R,\
\mu(x)=|x|_{\R}^r\sgn^m(x),\ r\in\CC,\ m\in\{0,1\};\\
\zeta_{\CC}(s+r+|m|),&\text{when }F=\CC,\
\mu(z)=|z|_{\CC}^r(z/\bar z)^m,\ r\in\CC,\ m\in\frac 12 \Z.\\
\end{cases}\]
(In this paper we denote $\sgn^2(z)=z/\bar z$ for $z\in\CC$.)
For a discrete series representation $\pi_{\mathrm{dis}}^{2k}$
of $\GL(2,\R)$ with weight $2k$,
one can define
\begin{gather*}
L(s,\pi_{\mathrm{dis}}^{2k})=\zeta_\CC(s+k-\tfrac 12),\quad
L(s,\pi_{\mathrm{dis}}^{2k},\Ad)=\zeta_\R(s+1)\zeta_\CC(s+2k-1);\\
L(s,\pi_{\mathrm{dis}}^{2k}\otimes
\pi_{\mathrm{dis}}^{2k}\otimes \pi_{\mathrm{dis}}^{2k})
=\zeta_\CC(s+3k-\tfrac 32)\zeta_\CC(s+k-\tfrac 12)^3;\\
L(s,(\pi_{\mathrm{dis}}^{2k})_{\CC}\otimes\sgn^{2m})
=\zeta_{\CC}(s+|k+m-\tfrac 12|)\zeta_{\CC}(s+|-k+m+\tfrac 12|).
\end{gather*}

Now let $F$ be a non-Archimedean local field with uniformizer $\varpi$,
and $q=\#(\cO_F/(\varpi))$.
For an unramified character $\mu$
(perhaps with superscripts and subscripts),
\[L(s,\mu)=(1-\mu(\varpi)q^{-s})^{-1},\quad
\zeta_F(s)=L(s,\1_F)=(1-q^{-s})^{-1}.\]
For a spherical representation $\pi(\mu_1,\mu_2)$
with $\mu_1,\mu_2$ unramified,
\begin{gather*}L(s,\pi(\mu_1,\mu_2))=L(s,\mu_1)L(s,\mu_2)
=(1-\mu_1(\varpi)q^{-s})^{-1}(1-\mu_2(\varpi)q^{-s})^{-1};\\
L(s,\pi(\mu_1,\mu_2),\Ad)=\zeta_F(s)L(s,\mu_1\mu_2^{-1})L(s,\mu_1^{-1}\mu_2);\\
L(s,\pi(\mu_1^{(1)},\mu_2^{(1)})\otimes\pi(\mu_1^{(2)},\mu_2^{(2)})\otimes\pi(\mu_1^{(3)},\mu_2^{(3)}))
=\prod_{i_1,i_2,i_3\in\{1,2\}}L(s,\mu_{i_1}^{(1)}\mu_{i_2}^{(2)}\mu_{i_3}^{(3)}).\end{gather*}
For a special representation $\sigma_\mu$ with $\mu$ unramified,
\begin{gather*}L(s,\sigma_\mu)=L(s+\tfrac 12,\mu)
=(1-\mu(\varpi)q^{-s-1/2})^{-1};\\
L(s,\sigma_\mu,\Ad)=\zeta_F(s+1)=(1-q^{-s-1})^{-1};\\
L(s,\sigma_{\mu_1}\otimes\sigma_{\mu_2}\otimes\sigma_{\mu_3})
=L(s+\tfrac 32, \mu_1\mu_2\mu_3)L(s+\tfrac 12,\mu_1\mu_2\mu_3)^2.\end{gather*}
In this case the local root numbers are
\[\varepsilon(\tfrac 12,\sigma_\mu)=-\mu(\varpi),\quad
\varepsilon(\tfrac 12,\sigma_{\mu_1}\otimes\sigma_{\mu_2}
\otimes\sigma_{\mu_3})
=-\mu_1\mu_2\mu_3(\varpi).\]
For a quadratic extension $E/F$,
the base change $L$-factors can be defined in the same way as above,
noticing that (cf. \cite[Appendix E.6]{getz2012hilbert})
\[(\pi(\mu_1,\mu_2))_E=\pi(\mu_1\circ N_{E/F},\mu_2\circ N_{E/F}),
\quad (\sigma_\mu)_E=\sigma_{\mu\circ N_{E/F}}.\]

Globally, for a number field $F$, a Hecke character $\mu$ on $\A_F^\times$,
automorphic representations $\pi,\pi_1,\pi_2,\pi_3$ of $\GL(2,\A_F)$,
a quadratic extension $E/F$
and a character $\Omega:E^\times\backslash\A_E^\times/
\A_F^\times\to\CC^\times$,
we define the
completed $L$-functions
\begin{gather*}
\zeta_F^*(s),\quad
L(s,\mu),\quad
L(s,\pi),\quad
L(s,\pi_1\otimes\pi_2\otimes\pi_3),\\
L(s,\pi_E\otimes\Omega),\quad
L(s,\pi,\Ad),\quad
L(s,\pi_1\otimes\pi_2\otimes\pi_3,\Ad)
\end{gather*}
as Euler products for large $\Re(s)$
of corresponding local $L$-factors over all places of $F$.

These $L$-functions can also be defined classically
associated to cuspidal modular forms.
For example,
normalizing $f(z)=\sum_{n\geq 1} a_n(f)e^{2\pi inz},g,h\in\cF_{2k}(N)$
such that $a_1=1$,
the local factors of triple product $L$-functions are defined as follows.
When $p\nmid N$,
\[L_p(s,f\times g\times h)=\prod_{i_1,i_2,i_3\in\{1,2\}}
(1-\alpha_p^{(i_1)}(f)\alpha_p^{(i_2)}(g)\alpha_p^{(i_3)}(h)p^{-s})^{-1}\]
where $\alpha_p^{(1)}(f)$, $\alpha_p^{(2)}(f)$ are defined
to be the roots of $X^2-a_p(f)X+p^{2k-1}=0$;
when $p\mid N$,
noticing that $a_p(f)p^{-(k-1)}=\pm 1$,
we define
\[L_p(s,f\times g\times h)=
(1+\varepsilon_pp^{3(k-1)}p^{-s})^{-1}
(1+\varepsilon_pp^{3(k-1)}p^{1-s})^{-2}\]
where $\varepsilon_p=-a_p(f)a_p(g)a_p(h)p^{-3(k-1)}$.
Then $L_{\fin}(s,f\times g\times h)$ is absolutely convergent
in the half plane $\Re(s)>3k-\frac 12$.

We recall some facts in order to 
translate Theorem \ref{mainthm} to adelic language.
Let $\pi_f,\pi_g,\pi_h$ be 
the cuspidal automorphic representations of $\GL(2,\A)$
generated by $f,g,h$ respectively 
(see Theorem \ref{cuspformtorep}). 
One can check by direct calculation that
\begin{equation}\label{adelizedL}
L_{\fin}(\tfrac 12,\pi_f\otimes\pi_g\otimes\pi_h)
=L_{\fin}(3k-1,f\times g\times h), \quad\text{and}
\end{equation}
\begin{equation}\label{adelizedL2}
L_{\fin}(\tfrac 12,(\pi_h)_E\otimes\Omega)=
\begin{cases}
L_{\fin}(k, h)L_{\fin}(k, h\otimes\chi_{-d}),&\Omega=\1,\\
L_{\fin}(k+|m|+\frac 12, h\times \Theta_\Omega),&\Omega\neq \1
\end{cases}
\end{equation}
($\chi_{-d}$, $m$ and $\Theta_\Omega$ defined as 
in Theorem \ref{mainthm}).
Moreover, 
Shimura \cite[Eqn (2.5)]{shimura1976special} 
and Hida \cite[Theorem 5.1]{hida1981congruence}
proved an identity that relates
a special value of the adjoint $L$-function 
and the Petersson norm of a normalized newform.
Particularly when $f\in \cF_{2k}(N)$ has square-free level $N>2$,
\begin{equation}L(1,\pi_f,\Ad)=\frac {2^{2k}}{N}(f,f).\end{equation}
Here $(\cdot,\cdot)$ is the Petersson inner product on $\cF_{2k}(N)$
defined in \eqref{def.Petersson}.
One can also find a proof of this in \cite[Lemma 5]{watson2008rankin}
or \cite[Proposition 1.11]{cai2014explicit}.

With the above identities and the definition of Archimedean $L$-factors,
one can easily check that the following theorem
is equivalent to Theorem \ref{mainthm}.
Recall that for $h\in\cF_{2k}(N)$, when $p\mid N$,
$(\pi_h)_p\cong\sigma_{\delta_p}$ is a special representation
with \[\delta_p(p)=a_p(h)p^{-(k-1)}=
-\varepsilon_p(\tfrac 12,\pi_h)=\pm 1.\]

\begin{Theorem}[Main Theorem, adelic version]\label{mainadelic}
Let $N$ be a square-free integer with an odd number of prime factors,
and $\cF(N,2k)$ be the set of cuspidal automorphic representations
of $\PGL(2,\A)$ of level $N$ and weight $2k$.
For any $\pi_3\in\cF(N,2k)$,
\begin{multline}\label{maineqnadelic}
\frac{1}{2 N^2}
\sum_{\substack{\pi_1,\pi_2\in\cF(N,2k)\\
\varepsilon_p=-1,\ \forall p\mid N}}
\frac{L(\frac 12,\pi_1\otimes\pi_2\otimes\pi_3)}
{L(1,\pi_1\otimes\pi_2\otimes\pi_3,\Ad)}\\
=\left(1-\frac{24\delta(k)}{\varphi(N)}\right)
\frac{\Gamma(k)^{3}\Gamma(3k-1)}{2^{\omega(N)}\Gamma(2k)\Gamma(2k-1)^2}+
\frac {(2\pi)^{-2k}\Gamma(2k-1)}{N\cdot L(1,\pi_3,\Ad)} \\
\cdot \left( 4\cdot 2^{\ord_2(N)}
\prod_{p\mid N}\frac{1-\chi_{-4}(p)}{2}
\sum_{\Omega\in\widehat{[E_0^\times]}}
I_0(\Omega)\cdot L_{\fin}(\frac 12,(\pi_3)_{\chi_{-4}}\otimes \Omega)
\right.\\
\left.
+6\sqrt{3}\cdot 2^{\ord_3(N)}
\prod_{p\mid N}\frac{1-\chi_{-3}(p)}{2}
\sum_{\Omega\in\widehat{[E_1^\times]}}
I_1(\Omega)\cdot L_{\fin}(\frac 12,(\pi_3)_{\chi_{-3}}\otimes \Omega)
\right),
\end{multline}
where $I_0$ and $I_1$ 
(defined in Theorem \ref{maingeometric}) 
depend only on $k$ and $\Omega$.
\end{Theorem}
In the following three sections 
we are going to prove Theorem \ref{mainadelic} 
using the relative trace formula (RTF).


\section{Explicit Period Formulae}
\label{section.explicit}

Waldspurger's period formula and Ichino's formula
relate some period integrals and
central values of certain $L$-functions.
In this section we will 
establish an explicit version of these two formulas,
regarding the period integrals appearing 
in each side of the RTF.

\subsection{Ichino's period formula}
\label{section.Ichino}

Let $F$ be a number field, $E=F\times F\times F$,
$\pi_i$ be a unitary cuspidal automorphic representation
of $\GL(2,\A_F)$ for $i=1,2,3$.
Assume that the product of the central characters of $\pi_i$ is trivial.
Assume that $D$ is a quaternion algebra over $F$ so that
there exists an irreducible unitary automorphic representation
$\Pi'=\pi_1'\otimes\pi_2'\otimes\pi_3'$ of $(D^\times(\A_F))^3$
associated to $\Pi=\pi_1\otimes\pi_2\otimes\pi_3$
by the Jacquet--Langlands correspondence.
Let $[D^\times]=Z(\A_F)D^\times(F)\backslash D^\times(\A_F)$.

\begin{Theorem}[\cite{ichino2008trilinear}]\label{Ichino}
For $\Phi=\otimes_v \Phi_v\in\Pi'$,
\[\frac
{\left|\int_{[D^\times]}\Phi(g,g,g)\ dg\right|^2}
{\langle\Phi,\Phi\rangle}
=\frac{\zeta_F^*(2)^2}{2^3}
\frac{L(\frac 12,\pi_1\otimes\pi_2\otimes\pi_3)}
{L(1,\pi_1\otimes\pi_2\otimes\pi_3,\Ad)}
\prod_v I_v,\]
where $dg=\prod_v d g_v$
is the Tamagawa measure on $[D^\times]$,
\[I_v=\frac 1{\zeta_{F_v}(2)^2}
\frac{L_v(1,\Pi,\Ad)}{L_v(1/2,\Pi)}
\int_{Z(F_v)\backslash D^\times(F_v)}
\frac{\cB_v(\Pi_v'(g_v)\Phi_v,\overline{\Phi_v}) }
{\cB_v(\Phi_v,\overline{\Phi_v})}\ dg_v,\]
the $\cB_v$'s are $(D^\times(\A_F))^3$-invariant pairings
between $\Pi_v'$ and its contragredient $\tilde\Pi_v'$ so that
$\prod_v \cB_v(\Phi_v,\overline{\Phi_v'})$
coincides with the inner product
$\langle\Phi,\Phi'\rangle$
defined by
\[\iiint_{[D^\times]^3}\Phi(g_1,g_2,g_3)
\overline{\Phi'(g_1,g_2,g_3)}\ dg_1\ dg_2\ dg_3\]
and $\cB_v(\Phi_v,\overline{\Phi_v'})=1$ for almost all $v$.
Moreover, for $v < \infty$,
\[I_v=\vol(K_v;dg_v)=
\vol(\cO_{F_v}^\times\backslash\GL(2,\cO_{F_v}))\]
when $D_v\cong M(2, F_v)$, 
$\Pi'_v\cong\Pi_v$ is unramified,
and $\Phi_v\in \Pi_v$ is a unit spherical vector.
\end{Theorem}

Now we calculate the local constants $I_v$ for the following case.
Assume $F=\Q$, $D$ is definite, $N=\disc(D)$ is
square-free with an odd number of prime factors,
$\pi_1,\pi_2,\pi_3\in\cF(N,2k)$ be 
cuspidal automorphic representations of $\PGL(2, \A)$,
and $\pi_1',\pi_2',\pi_3'\in\cF'(N,2k)$ be their 
Jacquet--Langlands correspondences
on $G'(\A)=Z(\A)\backslash D^\times(\A)$.
Let $\Phi$ be of the form $\Phi_{\pi_1'\otimes \pi'_2}\otimes\phi_3$
which contributes to the sum
$\sum_{\pi_1',\pi_2'\in\cA_{cusp}(G')}I_{\pi_1',\pi_2'}(f)$
(cf. \eqref{specterms}), where 
\[\Phi_{\pi_1'\otimes \pi'_2}=
\frac{w^\circ_{2k}}{\|w^\circ_{2k}\|}\otimes
\big(\otimes_p(\phi_{1,p}\otimes\phi_{2,p})\big)\]
is the normalized vector in the $1$-dimensional space
$\CC w^\circ_{2k}\otimes 
(\pi'_{1,\fin})^{K_{\fin}}(\pi'_{2,\fin})^{K_{\fin}}$ 
(defined in Lemma \ref{spec.archi.term}),
and $\phi_3=\|\phi_3\|X_3^{2k-2}\otimes(\otimes_p\phi_{3,p})$ 
is the new-line vector in $\pi_3'$ 
(defined in Lemma \ref{autoset})\footnote{
In this paper we use $\phi_3$ as an automorphic form in $\pi_3'$,
instead of the $3$-adic part of $\phi$.}.

\begin{Theorem}[An explicit Ichino formula]
\label{explicitIchino}
With $\Phi=\Phi_{\pi_1'\otimes \pi'_2}\otimes\phi_3$ defined above,
\[\left|\int_{[D^\times]}
\Phi_{\pi_1'\otimes \pi'_2}(h)\phi_3(h)\ dh\right|^2\]
vanishes unless $\varepsilon_p=-1$ for every $p\mid N$,
in which case it equals
\[\|\phi_3\|^2\frac{\Gamma(2k)^{2}\Gamma(2k-1)}
{\Gamma(k)^{3}\Gamma(3k-1)}
\frac{2^{\omega(N)}\varphi(N)}{ 48 N^2}
\frac{L(\frac 12,\pi_1\otimes\pi_2\otimes\pi_3)}
{L(1,\pi_1\otimes\pi_2\otimes\pi_3,\Ad)}.\]
\end{Theorem}

\begin{proof}
Now we have
\begin{equation}
\label{choice.archi}
\Phi_\infty=\frac{w^\circ_{2k}}{\|w^\circ_{2k}\|}
\otimes \|\phi_3\|X_3^{2k-2}
\end{equation}
and on the non-Archimedean places $\Phi_p$ are
tensor products of three $K_p$-invariant unit vectors.

When $p\nmid N$, 
$I_p=\vol(K_p;dg_p)=\zeta_p(2)^{-1}$ according to
\cite[Lemma 2.2]{ichino2008trilinear}.

When $p\mid N$, $(\pi_i)_p$ is the special representation $\sigma_{\delta_i}$
for some unramified character
$\delta_i:\Q_p^\times\to\{\pm 1\}$.
The corresponding $\phi_{i,p}$ is the normalization of
$\delta_i\circ N_{D_p}$.
\cite[Proposition 5.5]{woodbury2012explicit} shows that
\[I_p=(1-\varepsilon_p)\tfrac 1p(1-\tfrac 1{p})\zeta_p(2)^{-1},\quad
p\mid N\]
with $\varepsilon_p = -(\delta_1\delta_2\delta_3)(p)
=\varepsilon_p(\frac 12,\Pi)$.
(Note that the measure in \cite{woodbury2012explicit}
differs by a factor $\zeta_p(2)$.)

At the Archimedean place, recall that 
$\cB_\infty(\cdot,\bar\cdot)=\langle\cdot,\cdot\rangle$
is the inner product defined in \eqref{innerproduct2}.
By \eqref{choice.archi},
$\langle\Phi_\infty,\Phi_\infty\rangle_\infty=\|\phi_3\|^2$,
and, applying Lemma \ref{Schur}, we have
\[\int_{G'_\infty}
\cB_\infty(\Pi_\infty'(g)\Phi_\infty,\overline{\Phi_\infty})\ dg
=\frac{\|\phi_3\|^2}{\|w^\circ_{2k}\|^2}\int_{G'_\infty}
\langle\pi_{2k}'^{\otimes^3}\circ\Delta_3(g)w^\circ_{2k}\otimes X_3^{2k-2},
w^\circ_{2k}\otimes X_3^{2k-2}\rangle
\ dg\]
Recall that $\mathbb{P}_{2k}$, as defined in Lemma \ref{length},
is the only $D^\times(\R)$-invariant vector
in $\pi_{2k}'^{\otimes^3}\circ\Delta_3$
up to a constant multiple.
Then Schur Orthogonality Relation (Lemma \ref{SchurOrtho})
implies that the above integral is equal to
\[\frac{\|\phi_3\|^2}{\|w^\circ_{2k}\|^2}\vol(G'_\infty)
\frac{|\langle w^\circ_{2k}\otimes X_3^{2k-2}, 
\mathbb{P}_{2k}\rangle|^2}{\|\mathbb{P}_{2k}\|^2}.\]
Lemma \ref{Choosew} shows that
$\langle w^\circ_{2k}\otimes X_3^{2k-2}, \mathbb{P}_{2k}\rangle
=\|w^\circ_{2k}\|^2$.
Therefore, with Lemma \ref{length} 
and the definitions of $L$-factors in Section \ref{section.L-functions}, 
one can check that
\[I_\infty=\frac{\Gamma(2k)^{2}\Gamma(2k-1)}
{\Gamma(k)^{3}\Gamma(3k-1)}.\]

Now we see that $\prod_vI_v=0$ unless $\varepsilon_p=-1$ for any $p\mid N$,
in which case
\[\prod_vI_v=\zeta^*_{\Q}(2)^{-1}\frac{\Gamma(2k)^{2}\Gamma(2k-1)}
{\Gamma(k)^{3}\Gamma(3k-1)}
\frac{2^{\omega(N)}\varphi(N)}{ \pi N^2}.\]
With $\zeta^*_{\Q}(2)=\pi/6$ and Theorem \ref{Ichino} we complete the proof.
\end{proof}

\subsection{Waldspurger's period formula}\label{section.waldspurger}

Let $F$ be a number field, 
$\pi$ be a unitary cuspidal automorphic representation
of $\GL(2,\A_F)$.
Assume that $D$ is a quaternion algebra over $F$ so that
there exists an irreducible unitary automorphic representation
$\pi'$ of $D^\times(\A_F)$
associated to $\pi$ by the Jacquet--Langlands correspondence.

Let $E/F$ be a quadratic field extension embedded in $D$,
with $\eta_{E/F}:
F^\times\backslash\A_F^\times\to\{\pm 1\}$ 
the quadratic character
associated to $E/F$ by class field theory.
Let $[E^\times]=\A_F^\times E^\times \backslash \A_E^\times$
and $\pi_E$ denote the base change of $\pi$ to $E$.

Assume that $\Omega$ is a unitary character 
on $E^\times\backslash\A_E^\times$
such that $\Omega^{-1}\mid_{\A_F^\times}$ 
agrees with the central character of $\pi$.
We define a period integral $P_\Omega:\pi'\to\CC$ by
\begin{equation}\label{def.periodintegral}
P_\Omega(\phi) =\int_{[E^\times]}
\phi(t)\Omega(t)\ dt,\quad \phi\in\pi',
\end{equation}
where the Haar measure gives 
total volume $2L(1,\eta_{E/F})$ on $[E^\times]$.

\begin{Theorem}[\cite{waldspurger1985valeurs,
yuan2013gross}]
\label{Waldspurger}
For any $\phi',\phi''\in\pi'$, we have
\[\frac{P_\Omega(\phi')\cdot P_{\Omega^{-1}}(\overline{\phi''})}
{\langle\phi',\phi''\rangle}
=\frac{\zeta_F^*(2)L(\frac 12,\pi_E\otimes \Omega)}{2L(1,\pi,\Ad)}
\prod_v\alpha_v(\phi'_v,\overline{\phi''_v};\Omega_v),\]
with \[\alpha_v=
\frac{L_v(1,\eta_{E/F})L_v(1,\pi,\Ad)}
{\zeta_{F_v}(2)L_v(\frac 12,\pi_E\otimes \Omega)}
\int_{E_v^\times/F_v^\times}
\frac{B_v(\pi_v'(t_v)\phi'_v,\overline{\phi''_v})}
{B_v(\phi'_v,\overline{\phi''_v})}\Omega_v(t_v)\ dt_v.\]
Here the $B_v$'s
are $D_v^\times$-invariant bilinear forms
on $\pi_v'\otimes\tilde\pi_v'$ such that
\[\prod_vB_v(\phi'_v,\overline{\phi''_v})
=\langle\phi',\phi''\rangle
=\int_{[D^\times]}\phi'(g)\overline{\phi''(g)}\ dg.\]
Moreover, for $v<\infty$,
\[\alpha_v(\phi'_v,\overline{\phi_v''};\Omega_v)
=\vol(\cO^\times_{E_v}/\cO^\times_{F_v};dt_v)\]
when $D_v\cong M(2,F_v)$, $E_v/F_v$ is unramified,
$\pi_v'\cong\pi_v$ and $\Omega_v$ are both unramified, and
$\phi_v'\in\pi_v$, $\overline{\phi_v''}\in\tilde\pi_v$ 
are unit spherical vectors.
\end{Theorem}

Again we calculate the local constants $\alpha_v$ for the following case.
Assume $F=\Q$, $D$ is definite, $N=\disc(D)$ is
square-free with an odd number of prime factors,
$\pi\in\cF(N,2k)$ be a
cuspidal automorphic representation of $\PGL(2, \A)$,
and $\pi'$ be its Jacquet--Langlands correspondences.
Take $\phi^*,\phi^{**}\in\pi'$
such that 
\begin{equation}\label{phistar.general}
\phi^*_{\fin}=\phi^{**}_{\fin}=\otimes_p\phi_p\in(\pi'_{\fin})^{K_{\fin}}
\end{equation}
is the product of normalized spherical vectors.

Let $E/\Q$ be an imaginary quadratic field extension embedded in $D$.
Assume that $h=\otimes h_v\in D^\times(\A_\Q)$
satisfies that $h_p\in E_p^\times \GL(2,\Z_p)$
whenever $v=p<\infty$ and $p\nmid \disc(D)$.
Under the assumptions of $\phi^*,\phi^{**}$ we have
\begin{equation}\label{eqn.innerproduct}
\langle R(h)\phi^*,R(h)\phi^{**}\rangle
=\langle\phi^*,\phi^{**}\rangle
=B_\infty(\phi^*_\infty,
\overline{\phi^{**}_\infty})
=B_\infty(\pi_{2k}'(h_\infty)\phi^*_\infty,
\overline{\pi_{2k}'(h_\infty)\phi^{**}_\infty}).
\end{equation}

\begin{Theorem}[An explicit Waldspurger formula]
\label{explicitWaldspurger}
With notations and assumptions above, 
$P_\Omega(R(h)\phi^*)\cdot P_{\Omega^{-1}}(\overline{R(h)\phi^{**}})$
vanishes unless all the following conditions hold for $\Omega$:
\begin{itemize}
\item $\Omega=\otimes_v\Omega_v$ is unramified at every finite place,
\item $\Omega_\infty(z)=\sgn^{2m}(z)=(z/\bar z)^m$ with $|m|\leq k-1$, and 
\item if $p\mid \disc(D)$ and $E_p/\Q_p$ is ramified,
$(\tau/\bar\tau)^m=-\varepsilon_p(\frac 12,\pi)$ 
for any $\tau\in E$ such that $\tau$ is a uniformizer in $E_p$. 
\end{itemize}
In this case it equals 
$\prod_{p\nmid N}\dfrac{\vol(\cO_{E_p}^\times)}{\vol(\Z_p^\times)}
\prod_{p\mid N} \vol(E_p^\times/\Q_p^\times)$ times
\[\frac{L_{\fin}(\frac 12,\pi_E\otimes \Omega)}{L(1,\pi,\Ad)}
\frac{\Gamma(2k)}{6(2\pi)^{2k}} \frac{\varphi(N)}{N}
\int_{\CC^\times/\R^\times}
B_\infty(\pi_{2k}'(th_\infty)\phi^*_\infty,
\overline{\pi_{2k}'(h_\infty)\phi^{**}_\infty})
\sgn^{2m}(t)\ dt.\]
Here the measures are given in Section \ref{section.measure},
and the integral over $\CC^\times/\R^\times$ 
is calculated in Lemma \ref{geom.archi.vanish}.

In particular, when the weight of $\pi$ is $2k=2$,
we notice that
$\pi_{2k}'=\Sym^{2k-2} V \otimes {\det}^{-k+ 1}$ 
is the trivial representation of $D^\times(\R)$
and $\Omega_\infty$ can only be trivial.
The Waldspurger formula becomes
\begin{multline*}
\frac{P_\Omega(R(h)\phi^*)\cdot P_{\Omega^{-1}}(\overline{R(h)\phi^{**}})}
{\langle\phi^*,\phi^{**}\rangle}
=\vol(\CC^\times/\R^\times)
\prod_{p\nmid N}\frac{\vol(\cO_{E_p}^\times)}{\vol(\Z_p^\times)}
\prod_{p\mid N} \vol(E_p^\times/\Q_p^\times)\\
\cdot \frac{L(\frac 12,\pi_E\otimes \Omega)}{L(1,\pi,\Ad)}
\frac{\varphi(N)}{24N}
\prod_{p\in\Ram(D)\cap\Ram(E)}
\frac{1-\varepsilon_p(\frac 12,\pi)}{2}.
\end{multline*}
\end{Theorem}

We prove this explicit formula 
by calculating place-by-place the local constants
\[\alpha_v=\alpha_v(\pi_v'(h_v)\phi_v^*,
\overline{\pi_v'(h_v)\phi_v^{**}};\Omega_v).\]

\subsubsection{}\label{section.Waldspurger.split}
For $v=p<\infty$ and $p\nmid \disc(D)$,
we can fix an embedding $E_p\hookrightarrow D_p\cong M(2,\Q_p)$
so that the maximal order $\cO_p\cong M(2,\Z_p)$ of $D_p$
satisfies $E_p\cap \cO_p=\cO_{E_p}$.
The embedding can be chosen in the following way (cf. \cite{martin2009central}):
\begin{itemize}
\item If $E_p$ is split, fix an $\Q_p$-algebra isomorphism 
$E_p=\Q_p\oplus\Q_p$ and embed $E_p$ as the diagonal matrices in $M(2,\Q_p)$.
\item If $E_p$ is a field, let $\tau_p$
be such that $\cO_{E_p}=\Z_p[\tau_p]$
and that $\tau_p$ is a uniformizer in $E_p$ if $E_p/\Q_p$ is ramified.
Embed $E_p$ into $M(2,\Q_p)$ by
\[a+b\tau_p\longmapsto
\left(\begin{array}{cc}a+b\Tr_{E_p/\Q_p}(\tau_p)&bN_{E_p/\Q_p}(\tau_p)\\
-b&a\end{array}\right).\]
\end{itemize}
In this case,
$\pi_p'\cong \pi_p$ and they are spherical, i.e., 
$\pi_p$ is a unitarizable unramified principal series representation 
$\pi(\mu_p,\mu_p^{-1})$ 
for some unramified character $\delta_p$.
Recall that, in the induced model, 
the $\GL(2,\Q_p)$-invariant bilinear form
on $\pi_p\otimes\tilde\pi_p$ can be defined by
\[B_p(\phi,\phi')=
\int_{\GL(2,\Z_p)}\phi(k)\phi'(k)\ dk,\quad
\phi\in\pi_p,\ \phi'\in\tilde\pi_p.\]

\begin{Lemma}\label{nonarchiunram2}
Let $\phi'_p=\phi''_p=\pi_p(h_p)\phi_p$
where $\phi_p$ is the normalized spherical vector in $\pi_p$.
For $h_p\in E_p^\times\GL(2,\Z_p)$,
we have
\[\alpha_p(\phi'_p,\overline{\phi''_p};\Omega_p)=\begin{cases}
\dfrac{\vol(\cO_{E_p}^\times)}{\vol(\Z_p^\times)},
&\Omega_p\text{ unramified;}\\0,&\Omega_p\text{ ramified.}\end{cases}\]
\end{Lemma}

\begin{proof}
Write $h=t_hu_h$
with $t_h\in E_p^\times$, $u_h\in\cO_p^\times$.
Then $h^{-1}th=u_h^{-1}t_h^{-1}t t_hu_h=u_h^{-1}tu_h$ 
for any $t\in E_p^\times$.
By the definition of $B_p$,
\[\begin{split}
&\ B_p(\pi_p(t)\phi'_p,\overline{\phi''_p})=
B_p(\pi_p(h^{-1}th)\phi_p,\overline{\phi_p})
=\int_{\cO_p^\times}\pi_p(h^{-1}th)\phi_p(k)\overline{\phi_p(k)}\ dk\\
=&\ \int_{\cO_p^\times}
\phi_p(kh^{-1}th)\overline{\phi_p(k)}\ dk
=\int_{\cO_p^\times}
\phi_p(ku_h^{-1}tu_h)\overline{\phi_p(k)}\ dk\\
=&\ \int_{\cO_p^\times}
\phi_p(ku_h^{-1}t)\overline{\phi_p(ku_h^{-1})}\ d(ku_h^{-1})
=\int_{\cO_p^\times}
\phi_p(kt)\overline{\phi_p(k)}\ dk
=B_p(\pi_p(t)\phi_p,\overline{\phi_p}).
\end{split}\]
So 
\begin{equation}\label{independence.of.hp}
\int_{E_p^\times/\Q_p^\times}
B_p(\pi_p(t_p)\phi'_p,\overline{\phi''_p})
\Omega_p(t_p)\ dt_p
=\int_{E_p^\times/\Q_p^\times}
B_p(\pi_p(t_p)\phi_p,\overline{\phi_p}) \Omega_p(t_p)\ dt_p,
\end{equation}
i.e., the integral is independent of $h_p$;
and therefore, by Theorem \ref{Waldspurger},
\[\alpha_p(\phi'_p,\overline{\phi''_p};\Omega_p)
=\alpha_p(\phi_p,\overline{\phi_p};\Omega_p)
=\vol(\cO_{E_p}^\times)/\vol(\Z_p^\times)\]
\emph{when $\Omega_p$ is unramified and $E_p/\Q_p$ is unramified.}

For the rest cases
(when $\Omega_p$ is ramified or $E_p/\Q_p$ is ramified),
\eqref{independence.of.hp} still holds and hence we only need 
to calculate $\alpha_p(\phi_p,\overline{\phi_p};\Omega_p)$.

(i) \emph{When $E_p/\Q_p$ is split and $\Omega_p$ is ramified,}
we have $E_p^\times=\Q_p^\times\times \Q_p^\times$.
Notice that
$E_p^\times/\Q_p^\times\cong\{(x,1):x\in\Q_p^\times\}$,
We can write
\[\int_{E_p^\times/\Q_p^\times}
B_p(\pi_p(t_p)\phi_p,\overline{\phi_p}) \Omega_p(t_p)\ dt_p
=\int_{\Q_p^\times}
B_p(\pi_p\left(\begin{array}{cc}x&\\&1\end{array}\right)
\phi_p,\overline{\phi_p}) \Omega_p(x,1)\ d^\times x.\]
For $\Omega_p(t_1,t_2)=\nu_p(t_1)\nu_p^{-1}(t_2)$
we have
\[\begin{split}
&\ \int_{E_p^\times/\Q_p^\times}
B_p(\pi_p(t_p)\phi_p,\overline{\phi_p}) \Omega_p(t_p)\ dt_p\\
=&\ \sum_{r\in\Z}\int_{v_p(x)=r}
B_p(\pi_p\left(\begin{array}{cc}x&\\&1\end{array}\right)
\phi_p,\overline{\phi_p}) \Omega_p(x,1)\ d^\times x\\
=&\ \sum_{r\in\Z}\int_{u\in\Z_p^\times}
B_p(\pi_p\left(\begin{array}{cc}p^r&\\&1\end{array}\right)
\left(\begin{array}{cc}u&\\&1\end{array}\right)
\phi_p,\overline{\phi_p}) \nu_p(p^ru)\ du\\
=&\ \sum_{r\in\Z}
B_p(\pi_p\left(\begin{array}{cc}p^r&\\&1\end{array}\right)
\phi_p,\overline{\phi_p})\nu_p(p)^r
\int_{\Z_p^\times} \nu_p(u)\ du=0\quad
\text{when $\nu_p$ is ramified}.
\end{split}\]

(ii) \emph{When $E_p/\Q_p$ is unramified non-split 
and $\Omega_p$ is ramified,}
we have
\[E_p^\times=\sqcup_{r\in\Z}\ p^r\cO_{E_p}^\times=\Q_p^\times\cO_{E_p}^\times.\]
Lemma \ref{productmeasure} implies,
for any integrable function $f$ defined on $E_p^\times/\Q_p^\times$,
\[\int_{E_p^\times/\Q_p^\times}f(t)\ dt=
\int_{\cO_{E_p}^\times/\Z_p^\times}f(t)\ dt.\]
Recall that $\cO_{E_p}=E_p\cap \cO_p$ 
and $\phi_p$ is spherical.
So we have
\begin{multline*}
\int_{E_p^\times/\Q_p^\times}
B_p(\pi_p(t_p)\phi_p,\overline{\phi_p}) \Omega_p(t_p)\ dt_p
=\int_{\cO_{E_p}^\times/\Z_p^\times}
B_p(\pi_p(t_p)\phi_p,\overline{\phi_p})\Omega_p(t_p)\ dt_p\\
=B_p(\phi_p,\overline{\phi_p})
\int_{\cO_{E_p}^\times/\Z_p^\times} \Omega_p(t_p)\ dt_p
=0 \quad \text{when $\Omega_p$ is ramified}.
\end{multline*}

(iii) \emph{When $E_p/\Q_p$ is ramified,}
we know $E_p^\times=\Q_p^\times\cO_{E_p}^\times
\sqcup\varpi_{E_p}\Q_p^\times\cO_{E_p}^\times$.
Then for any integrable function $f$ on $E_p^\times/\Q_p^\times$,
\[\int_{E_p^\times/\Q_p^\times}f(t)\ dt
=\int_{\Q_p^\times\cO_{E_p}^\times/\Q_p^\times}f(t)\ dt+
\int_{\Q_p^\times\cO_{E_p}^\times/\Q_p^\times}f(\varpi_{E_p}t)\ dt.\]
By Lemma \ref{productmeasure} we can write
\[\int_{E_p^\times/\Q_p^\times}f(t)\ dt
=\int_{\cO_{E_p}^\times/\Z_p^\times}(f(t)+f(\varpi_{E_p}t))\ dt.\]
Notice that $\cO_{E_p}^\times\subset\cO_p^\times$. Then
\[\begin{split}
&\ \int_{E_p^\times/\Q_p^\times}
B_p(\pi_p(t_p)\phi_p,\overline{\phi_p}) \Omega_p(t_p)\ dt_p\\
=&\ \int_{\cO_{E_p}^\times/\Z_p^\times}
B_p(\pi_p(t_p)\phi_p,\overline{\phi_p})\Omega_p(t_p)\ dt_p
+\int_{\cO_{E_p}^\times/\Z_p^\times}
B_p(\pi_p(\varpi_{E_p}t_p)\phi_p,\overline{\phi_p}) \Omega_p(\varpi_{E_p}t_p)\ dt_p\\
=&\ B_p(\phi_p,\overline{\phi_p})
\int_{\cO_{E_p}^\times/\Z_p^\times} \Omega_p(t_p)\ dt_p
+B_p(\pi_p(\varpi_{E_p})\phi_p,\overline{\phi_p})\Omega_p(\varpi_{E_p})
\int_{\cO_{E_p}^\times/\Z_p^\times} \Omega_p(t_p)\ dt_p\\
=&\ \begin{cases}\vol(\cO_{E_p}^\times/\Z_p^\times)
\Big(B_p(\phi_p,\overline{\phi_p})+
B_p(\pi_p(\varpi_{E_p})\phi_p,\overline{\phi_p})\Omega_p(\varpi_{E_p})\Big) ,
&\Omega_p\text{ unramified;}\\
0,&\Omega_p\text{ ramified.}\end{cases}
\end{split}\]

By our choice of embedding $E_p\hookrightarrow D_p\cong M(2,\Q_p)$,
the image of $\varpi_{E_p}$ is
\[\left(\begin{array}{cc}\Tr_{E_p/\Q_p}(\varpi_{E_p})&p\\
-1&0\end{array}\right)
=\left(\begin{array}{cc}p&-\Tr_{E_p/\Q_p}(\varpi_{E_p})\\
0&1\end{array}\right)
\left(\begin{array}{cc}0&1\\
-1&0\end{array}\right)
\in\cO_p^\times 
\left(\begin{array}{cc}p&\\&1\end{array}\right)\cO_p^\times.\]
For $\pi_p=\pi(\mu_p,\mu_p^{-1})$,
the Macdonald formula \cite[Theorem 4.6.6]{bump1998automorphic}
implies that
\[
\frac{B_p(\pi_p\left(\begin{array}{cc}p&\\&1\end{array}\right)
\phi_p,\overline{\phi_p})}{B_p(\phi_p,\overline{\phi_p})}
=\frac{p^{-1/2}}{1+p^{-1}}(\mu_p(p)+\mu_p(p)^{-1}).
\]
Hence by the definition of $\alpha_p$, 
for unramified $\Omega_p$,
\begin{multline*}
\alpha_p(\phi'_p,\overline{\phi''_p};\Omega_p)=
\frac{L_p(1,\eta_{E/\Q})L_p(1,\pi,\Ad)}
{\zeta_{\Q_p}(2)L_p(\frac 12,\pi_E\otimes \Omega)}
\int_{E_p^\times/\Q_p^\times}
\frac{B_p(\pi_p'(t_p)\phi'_p,\overline{\phi''_p})}
{B_p(\phi'_p,\overline{\phi''_p})}\Omega_p(t_p)\ dt_p\\
=\frac{L_p(1,\eta_{E/\Q})L_p(1,\pi,\Ad)}
{\zeta_{\Q_p}(2)L_p(\frac 12,\pi_E\otimes \Omega)}
\vol(\cO_{E_p}^\times/\Z_p^\times)
\left(1+\frac{p^{-1/2}}{1+p^{-1}}(\mu_p(p)+\mu_p(p)^{-1})
\Omega_p(\varpi_{E_p})\right).
\end{multline*}
Here $L_p(1,\eta_{E/\Q})=1$, $\zeta_{\Q_p}(2)=(1-p^{-2})^{-1}$,
\[L_p(1,\pi,\Ad)
=\frac{(1-p^{-1})^{-1}}{(1-\mu_p^2(p)p^{-1})(1-\mu_p^{-2}(p)p^{-1})};\]
the base change $(\pi_p)_{E_p}=
\pi(\mu_p\circ N_{E_p/F_p},\mu_p^{-1}\circ N_{E_p/F_p})$ and
\[L_p(\tfrac 12,\pi_E\otimes \Omega)=
(1-\mu_p(p)\Omega_p(\varpi_{E_p})p^{-1/2})^{-1}
(1-\mu_p^{-1}(p)\Omega_p(\varpi_{E_p})p^{-1/2})^{-1}.\]
Notice that $\Omega_p(\varpi_{E_p})^2=1$ since 
$\Omega_p$ is unramified and its restriction to $\Q_p^\times$ is trivial.
One can simplify that
$\alpha_p(\phi'_p,\overline{\phi''_p};\Omega_p)=
\vol(\cO_{E_p}^\times/\Z_p^\times)$.
\end{proof}

\subsubsection{}
For $v=p<\infty$ and $p\mid \disc(D)$,
with Proposition \ref{F=Q}
we have that $E_p/\Q_p$ is non-split.
When $E_p/\Q_p$ is unramified and non-split, 
the ramification index is $e=1$, and
the inertia degree is $f=2$;
when $E_p/\Q_p$ is ramified, 
$e=2$, $f=1$.

In this case, $\pi_p=\sigma_{\delta_p}$
is the special representation of $\GL(2,\Q_p)$
and $\pi_p'=\delta_p\circ N_{D_p}$ is a character of 
$G_p'=Z(\Q_p)\backslash D_p^\times$,
where $\delta_p$ is an unramified character
of $\Q_p^\times$ of order at most 2.
We can take
\[\phi'_p=\phi''_p=\pi_p'(h_p)\phi_p,\quad 
\phi_p=\delta_p\circ N_{D_p}\]
for any $h_p\in D^\times_p$, with 
$B_p(\phi_p,\overline{\phi_p})$ defined to be 1.

\begin{Lemma}\label{nonarchramify1}
When $p\mid \disc(D)$, $\pi_p=\sigma_{\delta_p}$,
we have $\alpha_p(\phi'_p,\overline{\phi''_p};\Omega_p)=0$ unless
$\Omega_p=\delta_p^{-1}\circ N_{D_p}$, in which case
\[\alpha_p(\phi'_p,\overline{\phi''_p};\delta_p^{-1}\circ N_{D_p})
=(1-p^{-1})\vol(E_p^\times/\Q_p^\times).\]
In particular, \[\alpha_p(\phi'_p,\overline{\phi''_p};\1)
=(1-p^{-1})\vol(E_p^\times/\Q_p^\times)
\cdot\begin{cases}1,
&p\notin\Ram(E);\\
\frac 12 (1-\varepsilon_p(\frac 12,\pi)) ,
&p\in\Ram(E).\end{cases}\]
\end{Lemma}

\begin{proof}
We calculate the local integral in $\alpha_p$:
\[\begin{split}
&\ \int_{E_p^\times/\Q_p^\times}
\frac{B_p(\pi_p'(t_p)\pi_p'(h_p)\phi_p,
\overline{\pi_p'(h_p)\phi_p})}
{B_p(\pi_p'(h_p)\phi_p,\overline{\pi_p'(h_p)\phi_p})}
\Omega_p(t_p)\ dt_p\\
=&\ \int_{E_p^\times/\Q_p^\times}
\frac{B_p(\delta_p( N_{D_p}(t_p))\pi_p'(h_p)\phi_p,
\overline{\pi_p'(h_p)\phi_p})}
{B_p(\pi_p'(h_p)\phi_p,\overline{\pi_p'(h_p)\phi_p})}
\Omega_p(t_p)\ dt_p\\
=&\ \int_{E_p^\times/\Q_p^\times}\delta_p( N_{D_p}(t_p))
\Omega_p(t_p)\ dt_p
=\begin{cases}\vol(E_p^\times/\Q_p^\times),&
\text{if }\Omega_p=\delta_p^{-1}\circ N_{D_p};\\
0,&\text{otherwise.}\end{cases}
\end{split}\]
So $\alpha_p(\phi'_p,\overline{\phi''_p};\Omega_p)=0$ unless
$\Omega_p=\delta_p^{-1}\circ N_{D_p}$, in which case
it equals
\[\begin{split}
&\begin{cases}\dfrac{(1-(-1)p^{-1})^{-1}}{(1-p^{-2})^{-1}}
\dfrac{L_p(1,\pi,\Ad)}{L_p(\frac 12,\pi_E\otimes \Omega)}\vol(E_p^\times/\Q_p^\times),&
\text{if $E_p/\Q_p$ is unramified non-split}\\
\dfrac{1}{(1-p^{-2})^{-1}}
\dfrac{L_p(1,\pi,\Ad)}{L_p(\frac 12,\pi_E\otimes \Omega)}\vol(E_p^\times/\Q_p^\times),&
\text{if $E_p/\Q_p$ is ramified}
\end{cases}\\
=\ &\vol(E_p^\times/\Q_p^\times)(1-p^{-e})
\frac{L_p(1,\pi,\Ad)}{L_p(\frac 12, \pi_E\otimes\Omega)}
=\vol(E_p^\times/\Q_p^\times)(1-p^{-1}).
\end{split}\]
Recall that for $\pi_p=\sigma_{\delta_p}$, one has
$(\pi_E)_p=\sigma_{\delta_p\circ N_{E_p/\Q_p}}$ and therefore
\[L_p(\tfrac 12, \pi_E\otimes\Omega)=(1-p^{-f})^{-1},\quad
L_p(1,\pi,\Ad)=(1-p^{-2})^{-1}.\]

For the rest of the lemma, notice that
$\delta_p$ is unramified and of order at most 2.
So when $E_p/\Q_p$ is unramified and non-split,
$N_{D_p}(\varpi_{E_p})
=N_{E_p/\Q_p}(p)=p^2$ and therefore
$\delta_p\circ N_{D_p}=\1_{E_p^\times}$ always holds.
But if $E_p/\Q_p$ is ramified we know $N_{D_p}(\varpi_{E_p})
=N_{E_p/\Q_p}(\varpi_{E_p})=p$;
so $\delta_p^{-1}\circ N_{D_p}(\varpi_{E_p})=\delta_p(p)$,
and then $\alpha_p(\phi'_p,\overline{\phi''_p};\1)=0$
unless $\delta_p(p)=1$.
Recall that we have $\varepsilon_p(\frac 12,\pi)=-\delta_p(p)$ in this case.
\end{proof}

\subsubsection{}

For $v=\infty$ Archimedean, 
under our assumption we have $E_\infty=\CC$.
Recall that every character of $\CC^\times/\R^\times$
is of the form $z\mapsto (z/\bar z)^m$.
Denote it by $\Omega_\infty=\sgn^{2m}$ where
$\sgn(z)=z/|z|$ is the ``sign'' of $z\in\CC^\times$
on the complex unit circle
(and hence $\sgn(z)^2=z/\bar z$).

We will give a condition when $\alpha_\infty$ vanishes.
Recall that
$\pi_\infty'\cong \pi'_{2k}$
can be realized on the space
of homogeneous polynomials in $X,Y$ of degree $2k-2$ 
with
\[\pi'_{2k} (g)P(X,Y)=P((X,Y)g)\det(g)^{1-k}\quad
\text{for} \quad g\in
\left\{\left(\begin{matrix}\alpha&-\beta\\
\bar\beta&\bar\alpha\end{matrix}\right)\in\GL(2,\CC)\right\}
\cong D^\times(\R).\]
(See Section \ref{RepresentationSU2}.) 
This representation is determined by the last isomorphism.
Fix an isomorphism so that $t\in E^\times_\infty=\CC^\times$ 
is embedded as the diagonal matrices
$(\begin{smallmatrix}t&\\&\bar t\end{smallmatrix})$ 
(see Section \ref{section.geo.archi} for an explicit construction). Then
\[\pi_{2k}'(t)X^{2k-2-n}Y^n=\left(t/\overline{t}\right)^{k-1-n}X^{2k-2-n}Y^n.\]
In general
\begin{equation}\label{weightvector}
\pi_{2k}'(t)\left(\sum_{n=0}^{2k-2}c_nX^{2k-2-n}Y^n\right)
=\sum_{n=0}^{2k-2}\sgn^{2(k-1-n)}(t)c_nX^{2k-2-n}Y^n.
\end{equation}
We fix the bilinear form $B_\infty$
such that $B_\infty(\cdot,\overline{\cdot})
=\langle\cdot,\cdot\rangle_{2k}$
is the inner product on $\pi_{2k}'$ defined as \eqref{innerproduct}.

\begin{Lemma}\label{geom.archi.vanish}
For any $\phi'_\infty,\phi''_\infty\in V_{\pi_{2k}'}$,
\[\int_{\CC^\times/\R^\times}
B_\infty(\pi_{2k}'(t)\phi'_\infty,\overline{\phi''_\infty})
\sgn^{2m}(t)\ dt=0\quad\text{unless }
m\in\Z,\ -(k-1)\leq m\leq k-1.\]
Therefore 
$\alpha_\infty(\phi'_\infty,\overline{\phi_\infty''};\Omega_\infty)=0$
for any $\phi'_\infty,\phi''_\infty\in V_{\pi_{2k}'}$,
unless $|m|\leq k-1$.
\end{Lemma}
\begin{proof}
Write
\[\phi'_\infty=\sum_{r=0}^{2k-2}c_r'X^{2k-2-r}Y^r,\quad
\phi''_\infty=\sum_{r=0}^{2k-2}c_r''X^{2k-2-r}Y^r.\]
Recall that $\{X^{2k-2-r}Y^r\}$
forms an orthogonal basis of $V_{\pi_{2k}'}$
which are eigenvectors under the action of $\CC^\times$. Then
\[\begin{split}
&\ \int_{\CC^\times/\R^\times}
B_\infty(\pi_{2k}'(t)\phi'_\infty,\overline{\phi''_\infty})
\sgn^{2m}(t)\ dt\\
=&\ \int_{\CC^\times/\R^\times}
B_\infty\left(\sum_{r=0}^{2k-2}\sgn^{2(k-1-r)}(t)c_r'X^{2k-2-r}Y^r,
\overline{\phi''_v}\right)
\sgn^{2m}(t)\ dt\\
=&\ \sum_{r=0}^{2k-2}c_r'\overline{c_r''}
B_\infty( X^{2k-2-r}Y^r,\overline{X^{2k-2-r}Y^r})
\int_{\CC^\times/\R^\times}\sgn^{2(k-1-r)}(t)\sgn^{2m}(t)\ dt\\
=&\ \begin{cases}c_{k-1+m}'\overline{c_{k-1+m}''}
\| X^{k-1-m}Y^{k-1+m}\|^2
\vol(\CC^\times/\R^\times),
&\text{if }m\in\Z,\ -(k-1)\leq m\leq k-1;\\
0,&\text{otherwise.}\end{cases}
\end{split}\]
Finally, with 
the definitions of $L$-factors in Section \ref{section.L-functions}, 
we notice that
\begin{equation}\label{architerm}
B_\infty(\phi'_\infty,\overline{\phi''_\infty})
\alpha_\infty
=\frac{2^2(2\pi)^{-1-2k}\Gamma(2k)}
{L_\infty(\frac 12,(\pi_3)_E\otimes \Omega)}
\int_{\CC^\times/\R^\times}
B_\infty(\pi_{2k}'(t)\phi'_\infty,\overline{\phi''_\infty})
\sgn^{2m}(t)\ dt.
\end{equation}
\end{proof}

In conclusion,
$\alpha_p=0$ unless $\Omega_p$ is unramified for any finite $p$
by Lemma \ref{nonarchramify1} and \ref{nonarchiunram2}.
By Lemma \ref{geom.archi.vanish} $\alpha_\infty=0$ unless 
$\Omega_\infty$ is of the form $\sgn^{2m}$ with $|m|\leq k-1$.
Moreover for $p\mid \disc(D)$,
$\alpha_p(\phi'_p,\overline{\phi''_p};\Omega_p)=0$ unless
$\Omega_p=\delta_p^{-1}\circ N_{D_p}$,
where $\delta_p$ is the character of $\Q_p^\times$ 
such that $\pi_p=\sigma_{\delta_p}$.
This is to say, besides the unramification of $\Omega_p$,
we also need that
\[\Omega_p(\varpi_{E_p})=\delta_p^{-1}(N_{D_p}(\varpi_{E_p}))
=\delta_p^{-1}(N_{E_p/\Q_p}(\varpi_{E_p})).\]
This gives no extra information when $E_p/\Q_p$ is unramified.
But when $E_p/\Q_p$ is ramified,
this leads to more restriction for $\Omega_\infty$,
noticing that $\Omega$ is trivial on $E^\times$:
if $\tau\in E^\times$ is a uniformizer in $E_p$,
then
\[\Omega_\infty(\tau)\Omega_p(\tau)=1\Rightarrow
\Omega_\infty(\tau)=\Omega_p(\tau)^{-1}=\delta_p(N_{E_p/\Q_p}(\tau))=\delta_p(p).\]
For $\Omega_\infty=\sgn^{2m}$ 
this means $(\tau/\bar\tau)^m=\delta_p(p)=-\varepsilon_p(\frac 12,\pi)$.

This observation, along with
Lemmas \ref{nonarchiunram2}, \ref{nonarchramify1}, 
\ref{geom.archi.vanish} and \eqref{architerm},
completes the proof of Theorem \ref{explicitWaldspurger}.


\section{Spectral Side of the RTF}\label{section.spectral}

Let $D$ be the definite quaternion algebra over $\Q$ with discriminant $N$.
Let $G'$ be the algebraic group defined over $\Q$
with $G'(\Q) =Z(\Q)\backslash D^\times(\Q)$.
We consider the RTF introduced in Section \ref{RTF} for the case
$G'\backslash(G'\times G')/G'$.

Let $f\in C^\infty_c(G'(\A)\times G'(\A))$.
Integrating $f$ against the action of $G'(\A)\times G'(\A)$ gives a linear map
\[R(f) : L^2(G'(\Q)\times G'(\Q)\backslash G'(\A)\times G'(\A)) \to
 L^2(G'(\Q)\times G'(\Q)\backslash G'(\A)\times G'(\A))\]
defined by
\[(R(f)\Phi)(x_1,x_2) = \int_{G'(\A)}\int_{G'(\A)}
f(g_1,g_2) \Phi(x_1g_1,x_2g_2) \ dg_1\ dg_2.\]
From the spectral decomposition of
$L^2(G'(\Q)\times G'(\Q)\backslash G'(\A)\times G'(\A))$
one sees that $R(f)$ is an integral operator with kernel
\[K_f(x_1,x_2;y_1, y_2)
=\sum_{\pi_1'\otimes\pi_2'\in\cA(G'\times G')}
\sum_{\Phi\in\ONB(\pi_1'\otimes\pi_2')}
((\pi_1'\otimes\pi_2')(f)\Phi)(x_1,x_2)
\overline{\Phi(y_1,y_2)},\]
where $\ONB(\pi)$ denotes an orthonormal basis of the space $V_{\pi}$ of $\pi$.
Since $G'=Z\backslash D^\times$ is anisotropic,
for any test function $f\in C^\infty_c(G'(\A)\times G'(\A))$,
the operator $R(f)$ is of trace class,
that means the sums in the kernel $K_f$ are absolutely convergent,
independent of the choice of basis.

Having fixed the diagonal embedding $G' \hookrightarrow G'\times G'$
we get an injection $G'(\A)\hookrightarrow G'(\A)\times G'(\A)$.
Let $\pi_3'\in\cF'(N,2k)$ be an automorphic representation of $G'$.
Fixing an automorphic form
$\phi_3\in \CC X_{3}^{2k-2} \otimes(\pi_{3,\fin}')^{K_{\fin}}$ on $G'$
(a new-line vector defined in Lemma \ref{autoset}),
we define a distribution
\begin{equation}\label{I(f)}
I(f) = \int_{G'(\Q) \backslash G'(\A)} \int_{G'(\Q) \backslash G'(\A)}
K_f(h_1,h_1;h_2, h_2) \phi_3(h_1)\overline{\phi_3(h_2)} \ dh_1 \ dh_2.
\end{equation}
The spectral expansion for the kernel $K_f(x_1,x_2;y_1, y_2)$ gives
\[I(f) = \sum_{\pi_1'\otimes\pi_2'}I_{\pi_1',\pi_2'}(f) \]
where for each $\pi_1',\pi_2'\in\cA(G')$ we define
\begin{multline*}I_{\pi_1',\pi_2'}(f) =
\sum_{\Phi\in\ONB(\pi_1'\otimes\pi_2')}
\int_{G'(\Q) \backslash G'(\A)}
((\pi_1'\otimes\pi_2')(f)\Phi)(h_1,h_1)\phi_3(h_1) \ dh_1 \\
\cdot\overline{\int_{G'(\Q) \backslash G'(\A)}
\Phi(h_2,h_2)\phi_3(h_2)\ dh_2}.
\end{multline*}

Recall that $\cA(G') = \cA_{cusp}(G') \sqcup \cA_{res}(G')$
where $\cA_{res}(G')$ is the set of characters of $G'(\Q)\backslash G'(\A)$.
The automorphic representations of $G'\times G'$
are of the form $\pi_1'\otimes\pi_2'$ where
\begin{itemize}
	\item $\pi_1',\pi_2'\in \cA_{cusp}(G')$;
	\item one in $\cA_{cusp}(G')$ and another in $\cA_{res}(G')$; or
	\item $\pi_1',\pi_2'\in \cA_{res}(G')$.
\end{itemize}
Correspondingly we can decompose $I(f)$ as
\begin{equation}\label{specdecomp}
I(f)=\sum_{\pi_1',\pi_2'\in \cA_{cusp}(G')}I_{\pi_1',\pi_2'}(f)
+2\sum_{\substack{\pi_1'\in \cA_{res}(G')\\\pi_2'\in \cA_{cusp}(G')}}
I_{\pi_1',\pi_2'}(f)
+\sum_{\pi_1',\pi_2'\in \cA_{res}(G')}I_{\pi_1',\pi_2'}(f).
\end{equation}

Our goal is to choose a suitable test function
$f\in C^\infty_c(G'(\A)\times G'(\A))$
such that $R(f)$ kills all
$\pi_1'\otimes\pi_2'\in\cA_{cusp}(G')\otimes \cA_{cusp}(G')$
unless $\pi_1', \pi_2' \in \cF'(N, 2k)$.
Next we will show that, 
for $\pi_3'\in\cF'(N,2k)$,
with the test function defined by \eqref{def.testfunction},
\begin{equation}\label{specterms}
\left\{\begin{aligned}
\frac{\sum\limits_{\pi_1',\pi_2'\in\cA_{cusp}(G')}I_{\pi_1',\pi_2'}(f)}
{\vol(K')^2}
&=\left(\frac 1{2k-1}\right)^2
\sum_{\pi_1',\pi_2'\in\cF'(N,2k)}
\left|\int_{G'(\Q) \backslash G'(\A)}
\Phi_{\pi_1'\otimes \pi'_2}(h)\phi_3(h)\ dh\right|^2.\\
\frac{\sum\limits_{\substack{\pi_1'\in\cA_{res}(G')\\\pi_2'\in\cA_{cusp}(G')}}
I_{\pi_1',\pi_2'}(f)}
{\vol(K')^2}
&
=\begin{cases}
\dfrac 12\langle \phi_3,\phi_3 \rangle,
&\text{if }2k=2;\\
0,&\text{otherwise}.\end{cases}\\
\frac{\sum\limits_{\pi_1',\pi_2'\in\cA_{res}(G')}I_{\pi_1',\pi_2'}(f)}
{\vol(K')^2}&=0.\end{aligned}\right.\end{equation}
Here $\Phi_{\pi_1'\otimes \pi'_2}$
is defined in Lemma \ref{spec.archi.term};
$K'=G'_\infty \prod_{p<\infty} K_p$ is an open subgroup of $G'(\A)$,
and therefore
\[\vol(K')=\vol(G'_\infty;\ dg_\infty)\prod_p\vol(K_p;\ dg_p)
=\frac{24}{\varphi(N)}.\]
Combining \eqref{specdecomp} \eqref{specterms} 
and the explicit Ichino formula (Theorem \ref{explicitIchino}), 
we get that:


\begin{Theorem}[Main Theorem, spectral side]\label{mainspectral}
Let $N$ be a square-free product of an odd number of primes.
For $\pi_3'\in\cF'(N,2k)$, the distribution
$I(f)$ is equal to $\langle \phi_3,\phi_3 \rangle\vol(K')$ times
\[\frac{\Gamma(2k-1)^3}
{\Gamma(k)^{3}\Gamma(3k-1)}
\frac{2^{\omega(N)}}{2 N^2}
\sum_{\substack{\pi_1,\pi_2\in\cF(N,2k)\\
\varepsilon_p=-1,\ \forall p\mid N}}
\frac{L(\frac 12,\pi_1\otimes\pi_2\otimes\pi_3)}
{L(1,\pi_1\otimes\pi_2\otimes\pi_3,\Ad)}+
\begin{cases}
24/\varphi(N),&\text{if }2k=2;\\
0,&\text{otherwise.}
\end{cases}\]
\end{Theorem}

\subsection{Test function}\label{ChooseTest}

For $v=\infty$, recall that $\pi'_{2k}$
(defined in Section \ref{RepresentationSU2})
corresponds to $\pi_{\mathrm{dis}}^{2k}$
via the local Jacquet--Langlands correspondence.
Let $\langle \ , \ \rangle$ be a $G'_\infty$-invariant
inner product of $\pi'_{2k}\otimes\pi'_{2k}$.
By Lemma \ref{autoset} we fix $\phi_{3,\infty}=\|\phi_3\|X_3^{2k-2}$,
and Lemma \ref{Choosew} shows that
there exists a vector $w^\circ_{2k}\in \pi'_{2k}\otimes\pi'_{2k}$
such that
\[\int_{G'_\infty}
\langle\pi'_{2k}\otimes\pi'_{2k}(g,g)w^\circ_{2k},w^\circ_{2k}\rangle\pi'_{2k}(g)X_3^{2k-2}\ dg\]
is a constant multiple of $X_3^{2k-2}$.
We fix such a nonzero vector $w^\circ_{2k} \in \pi'_{2k}\otimes\pi'_{2k}$
as in Lemma \ref{length}
and define $f_\infty\in C^\infty_c(G'_\infty\times G'_\infty)$ by
\[ f_\infty(g_1,g_2) = \overline{
\langle \pi'_{2k}\otimes\pi'_{2k}(g_1,g_2) w^\circ_{2k}, w^\circ_{2k}\rangle}
/\langle w^\circ_{2k}, w^\circ_{2k}\rangle.\]

For $v=p<\infty$, fix a maximal order $\cO_p$ of $D_p=D(\Q_p)$.
Let $K_p$ be the image of $Z(\Q_p) \cO_p^\times $ in $G_p'$.
Clearly $f_0=\prod_{p<\infty}f_p$
is a function in $C_c^\infty(G'(\A_{\fin}))$,
where $f_p=\1_{K_p}$.
We define the test function on $G'(\A)\times G'(\A)$
by $f=f_\infty\times (f_0\otimes f_0)$,
i.e. \begin{equation}\label{def.testfunction}
f(g_1,g_2)=f_\infty(g_{1,\infty},g_{2,\infty})
\prod_{p<\infty}\1_{K_p}(g_{1,p})\1_{K_p}(g_{2,p}).\end{equation}
In particular when $2k=2$, $\pi'_{2k}$ is trivial and so is $f_\infty$.
In this case the test function is simply
\[f(g_1,g_2)= f_0\otimes f_0=
\prod_{p<\infty}\1_{K_p}(g_{1,p})\1_{K_p}(g_{2,p}).\]

When choosing $\Phi\in\ONB(\pi_1'\otimes\pi_2')$
we take $\Phi=\Phi_\infty\cdot \phi_{1,\fin}\otimes\phi_{2,\fin}$,
where $\Phi_\infty$ is a unit vector 
in $\pi_{1,\infty}'\otimes\pi_{2,\infty}'$,
and $\phi_{1,\fin}=\phi_{2,\fin}=\otimes_p\phi_p$
are functions in $\pi_{1,\fin}',\pi_{2,\fin}'$ respectively, 
as the finite parts of the new-line vectors 
in $\pi_{1}',\pi_{2}'$ defined in Lemma \ref{autoset}.
(Here the spherical unit vectors $\phi_p$ are realized in the induced model
when $p\nmid N$.)
With this test function we have
\[\begin{split}
(R(f)\Phi)(x_1,x_2)=&\ 
((\pi_1'\otimes\pi_2')(f_\infty\cdot f_0\otimes f_0)
\Phi)(x_1,x_2)\\
=&\ 
(\pi_{1,\infty}'\otimes\pi_{2,\infty}'(f_\infty)\Phi_\infty)\cdot
(R(f_0\otimes f_0)(\phi_{1,\fin}\otimes\phi_{2,\fin}))(x_{1,\fin},x_{2,\fin})
\end{split}\]
and
\[\begin{split}
&\ (R(f_0\otimes f_0)(\phi_{1,\fin}\otimes\phi_{2,\fin}))(x_{1,\fin},x_{2,\fin}) \\
=& \ \int_{G'(\A_{\fin})}\int_{G'(\A_{\fin})}
f_0(g_1)f_0(g_2) \phi_{1,\fin}(x_{1,\fin}g_1)\phi_{2,\fin}(x_{2,\fin}g_2) \ dg_1\ dg_2\\
=&\ (R(f_0)\phi_{1,\fin})(x_{1,\fin})\cdot(R(f_0)\phi_{2,\fin})(x_{2,\fin})
=(\pi_{1,\fin}'(f_0)\phi_{1,\fin})(x_{1,\fin})\cdot(\pi_{2,\fin}'(f_0)\phi_{2,\fin})(x_{2,\fin}).
\end{split}\]
Then the distribution $I_{\pi_1',\pi_2'}$ becomes
\begin{multline*}
\begin{aligned}&\ I_{\pi_1',\pi_2'}(f_\infty\cdot f_0\otimes f_0) \\= &
\sum_{\substack{\Phi\in\ONB(\pi_1'\otimes\pi_2')\\
\Phi=\Phi_\infty\cdot \phi_{1,\fin}\otimes\phi_{2,\fin}}}
\int_{G'(\Q) \backslash G'(\A)}
\Big(\pi_{1,\infty}'\otimes\pi_{2,\infty}'(f_\infty)\Phi_\infty\Big)
\Big(\pi_{1,\fin}'(f_0)\phi_{1,\fin}\Big)
\Big(\pi_{2,\fin}'(f_0)\phi_{2,\fin}\Big)
\phi_3(h_1) \ dh_1 \end{aligned}\\
\cdot \overline{\int_{G'(\Q) \backslash G'(\A)}
\Phi(h_2,h_2)\phi_3(h_2)\ dh_2}.\end{multline*} 

Next, we will answer the question that,
with the test function defined by \eqref{def.testfunction},
which representations would contribute to the
spectral decomposition \eqref{specdecomp}.

\subsection{Cusp \texorpdfstring{$\otimes$}{x} Cusp}

For an irreducible admissible representation $\sigma$ of $G'_v$
acting on the space $V_{\sigma}$ and for $f_v\in C_c^\infty(G_v')$,
we define $\sigma(f_v):V_{\sigma}\to V_{\sigma}$ by
\[\sigma(f_v)w=\int_{G_v'}f_v(g_v)\sigma(g_v)w\ dg_v.\]

On the non-Archimedean places we have the following lemma.
\begin{Lemma}[{\cite[Lemma 3.2, 3.3]{feigon2009averages}}]
\label{local.cuspcusp}
For $v=p<\infty$, let $f_p=\1_{K_p}$ as above.
Let $\sigma$ be an irreducible unitary representation of $G'_p$.
Then $\sigma(f_p)$ kills the orthogonal complement of $\sigma^{K_p}$ in $V_\sigma$,
and $\sigma(f_p)w=\vol(K_p)w$ for $w\in\sigma^{K_p}$.
\end{Lemma}

On the Archimedean place we have: 
\begin{Lemma}[cf. {\cite[Lemma 3.4]{feigon2009averages}}]
Let $\sigma=\sigma_1\otimes\sigma_2$ be an irreducible
unitary representation of $G'_\infty\times G'_\infty$.
Then, with the definition of $f_\infty$ as above,
$\sigma(f_\infty)$ kills the space $V_\sigma$ unless
$\sigma\cong \pi'_{2k}\otimes \pi'_{2k}$.
Furthermore for $\sigma=\pi'_{2k}\otimes\pi'_{2k}$,
$\sigma(f)$ kills the orthogonal complement of $\CC w^\circ_{2k}$ in $V_\sigma$,
and  \[\pi'_{2k}\otimes \pi'_{2k}(f_\infty)w^\circ_{2k}
=\left(\frac{\vol(G'_\infty)}{2k-1}\right)^2 w^\circ_{2k}. \]
\end{Lemma}
\begin{proof}
Since $f_\infty$ is a matrix coefficient,
Schur Orthogonality Relations (Lemma \ref{SchurOrtho}) show that
$\sigma(f_\infty)$ kills the space $V_\sigma$ unless
$\sigma\cong \pi'_{2k}\otimes  \pi'_{2k}$, and
\[\begin{split}
&\ \langle \pi'_{2k}\otimes \pi'_{2k}(f_\infty)w_1,w_2\rangle=
\int_{G'_\infty\times G'_\infty}f_\infty(g)
\langle\pi'_{2k}\otimes \pi'_{2k}(g)w_1,w_2\rangle\ dg\\
=&\ 
\int_{G'_\infty\times G'_\infty}\langle\pi'_{2k}\otimes \pi'_{2k}(g)w_1,w_2\rangle
\frac{\overline{\langle \pi'_{2k}\otimes \pi'_{2k}(g)w^\circ_{2k},w^\circ_{2k}\rangle}}
{\langle w^\circ_{2k},w^\circ_{2k}\rangle}\ dg\\
=&\ \frac{\vol(G'_\infty\times G'_\infty)}{\dim \pi'_{2k}\otimes \pi'_{2k}}
\frac{\langle w_1,w^\circ_{2k}\rangle\overline{\langle w_2,w^\circ_{2k}\rangle}}
{\langle w^\circ_{2k},w^\circ_{2k}\rangle}
=\begin{cases}\left(\dfrac{\vol(G'_\infty)}{2k-1}\right)^2
\|w^\circ_{2k}\|^2&\text{when }w_1=w_2=w^\circ_{2k};\\
0&\text{when }w_1\text{ or }w_2\in (\CC w^\circ_{2k})^\perp.
\end{cases}
\end{split}\]
\end{proof}

Now we apply the lemmas for $\pi_{1,v}'\otimes\pi_{2,v}'$ and
work on $\pi_1'\otimes\pi_2'(f)\Phi$
for $\Phi\in\ONB(\pi_1'\otimes\pi_2')$.

\begin{Lemma}\label{spec.archi.term}
For cuspidal representations $\pi_1',\pi_2'$ of $G'$,
\[\sum_{\Phi\in\ONB(\pi_1'\otimes\pi_2')}
\Big(\pi_1'\otimes\pi_2'(f)\Phi\Big)(x)
=\begin{cases}
\Phi_{\pi_1'\otimes\pi_2'}(x)\left(\dfrac{\vol(K')}{2k-1}\right)^2,&
\pi_1',\pi_2'\in\cF'(N,2k);\\
0,&\text{otherwise}.
\end{cases}\]
Here $K'=G'_\infty \prod_{p<\infty} K_p$, and
$\Phi_{\pi_1'\otimes\pi_2'}$ is an orthonormal basis of
the $1$-dimensional subspace
\[W_{\pi_1'\otimes \pi_2'}=\CC w^\circ_{2k}
\otimes (\pi'_{1,\fin})^{ K_{\fin}}(\pi'_{2,\fin})^{ K_{\fin}}
\subseteq V_{\pi_1'\otimes\pi_2'}.\]
\end{Lemma}
\begin{proof}
It follows from the previous two lemmas
that $R(f)$ kills the orthogonal complement of $\Phi_{\pi_1'\otimes\pi_2'}$
in $V_{\pi_1'\otimes\pi_2'}$ and
\[\Big(\pi_1'\otimes\pi_2'(f)\Phi_{\pi_1'\otimes\pi_2'}\Big)(x_1,x_2)
=\Phi_{\pi_1'\otimes\pi_2'}(x_1,x_2)\left(\frac{\vol(G'_\infty)}{2k-1}\prod_{p<\infty}\vol(K_p)\right)^2.\]
\end{proof}

This lemma implies that, for $\pi_3'\in\cF'(N,2k)$,
\[
\sum_{\pi_1',\pi_2'\in\cA_{cusp}(G')}I_{\pi_1',\pi_2'}(f)
=\left(\frac {\vol(K')}{2k-1}\right)^2
\sum_{\pi_1',\pi_2'\in\cF'(N,2k)}
\left|\int_{G'(\Q) \backslash G'(\A)}
\Phi_{\pi_1'\otimes \pi'_2}(h)\phi_3(h)\ dh\right|^2.\]
This is the first case in \eqref{specterms}.

\subsection{Res \texorpdfstring{$\otimes$}{x} Res}
Every $\pi'\in\cA_{res}(G')$
is a character $\delta\circ N_D$
for some $\delta:\Q^\times\backslash\A^\times\to\{\pm 1\}$.
We can take $\phi\in\ONB(\pi')$ to be the normalization of $\delta\circ N_D$.

\begin{Lemma}\label{resres}
$R(f)(\delta_1\circ N_D\otimes \delta_2\circ N_D)\equiv 0$
unless $2k=2$
and $\delta_1,\delta_2$ are unramified everywhere,
in which case (with $K'=G'_\infty \prod_{p<\infty} K_p$)
\[R(f)(\delta_1\circ N_D\otimes \delta_2\circ N_D)(x_1,x_2)
=\delta_1(N_D(x_1))\delta_2(N_D(x_2))\vol(K')^2.\]
\end{Lemma}

\begin{proof}
By definition we have
\begin{multline*}\begin{aligned}
&\ R(f)(\delta_1\circ N_D\otimes \delta_2\circ N_D)(x_1,x_2)\\
=&\ \int_{G'(\A)\times G'(\A)}f(g_1,g_2)
\delta_1(N_D(x_{1}g_1))\delta_2(N_D(x_{2}g_2))~dg_1~dg_2\\
=&\ \int_{G'_\infty\times G'_\infty}f_\infty(g_{1},g_{2})
\delta_{1,\infty}(N_{D_\infty}(x_{1,\infty}g_{1}))
\delta_{2,\infty}(N_{D_\infty}(x_{2,\infty}g_{2}))~dg_{1}~dg_{2}
\end{aligned}\\
\cdot \prod_{p<\infty}
\int_{G'_p}\1_{K_p}(g)\delta_v(N_{D_p}(x_{1,p}g))~dg
\int_{G'_p}\1_{K_p}(g)\delta_p(N_{D_p}(x_{2,p}g))~dg.
\end{multline*}

When $p<\infty$,
since the norm map $N_{D_p}:\cO_p^\times\to\Z_p^\times$ is surjective
(whether $D_p$ is split or not, 
cf. \cite[Lemma 13.4.9]{voight2017quaternion}), we have
\[ \begin{split}
\int_{G'_p}\1_{K_p}(g)\delta_p(N_{D_p}(xg))~dg  &
=\delta_p(N_{D_p}(x))\int_{K_p}\delta_p(N_{D_p}(g))~dg \\&
=\begin{cases}
0,&\text{$\delta_p$ ramified;}\\
\delta_p(N_{D_p}(x))\vol(K_p),&\text{$\delta_p$ unramified.}
\end{cases}\end{split}
\]

When $v=\infty$, $N_{D_\infty}(D_\infty^\times)=\R^\times_{>0}$
by \cite[Lemma 13.4.9]{voight2017quaternion}.
Then $\delta_\infty(N_{D_\infty}(g))=1$
for all $g\in D_\infty^\times$
since $\delta_\infty$ is quadratic. Thus
\[\begin{split}
&\ \int_{G_\infty'\times G_\infty'}f_\infty(g_{1},g_{2})
\delta_{1,\infty}(N_{D_\infty}(x_{1}g_{1}))
\delta_{2,\infty}(N_{D_\infty}(x_{2}g_{2}))~dg_{1}~dg_{2}\\
=&\ \delta_{1,\infty}(N_{D_\infty}(x_{1}))\delta_{2,\infty}(N_{D_\infty}(x_{2}))
\frac{\overline{\int_{G_\infty'\times G_\infty'}
\langle \pi'_{2k}\otimes \pi'_{2k}(g_{1},g_{2}) w^\circ_{2k}, w^\circ_{2k}\rangle
~d(g_1,g_2)}}
{\langle w^\circ_{2k}, w^\circ_{2k}\rangle}\\
=&\ \begin{cases}
0,&\text{if $2k>2$;}\\
\delta_{1,\infty}(N_D(x_{1}))\delta_{2,\infty}(N_D(x_{2}))
\vol(G'_\infty)^2,&\text{if }\pi'_{2k}\otimes \pi'_{2k}\cong \id.\text{, i.e. }2k=2.\end{cases}
\end{split}\]

Putting these local calculations together shows the statement.
\end{proof}

For any global field $F$,
we define $X^{\text{un}}(F)$ to be the set of Hecke characters
$\delta:F^\times\backslash\A_F^\times\to\{\pm 1\}$
that are unramified everywhere.
Notice that 
\[\|\delta\circ N_D\|^2
=\int_{G'(\Q)\backslash G'(\A)}|\delta( N_D(h))|^2\ dh
=\vol(G'(\Q)\backslash G'(\A)).\]
Then
\[\frac{\sum\limits_{\pi_1',\pi_2'\in\cA_{res}(G')}I_{\pi_1',\pi_2'}(f)}
{\vol(K')^2}
=\begin{cases}
\sum\limits_{\delta_1,\delta_2\in X^\text{un}(\Q)}
\left|\int_{G'(\Q) \backslash G'(\A)}
\dfrac{\delta_1\delta_2(N_D(h))\phi_3(h)}{\vol(G'(\Q)\backslash G'(\A))}\ dh\right|^2,
&\text{if }2k=2;\\
0,&\text{otherwise}.\end{cases}\]
Recall that $\vol(G'(\Q)\backslash G'(\A))=\vol([D^\times])=2$.
Moreover, Lemma \ref{char.unram} shows that any character
$\delta:\Q^\times\backslash\A^\times\to\{\pm 1\}$
that is unramified everywhere can only be trivial,
i.e. $X^{\text{un}}(\Q)=\{\1\}$.
Then, if $2k=2$, \[\frac{\sum\limits_{\pi_1',\pi_2'\in\cA_{res}(G')}
I_{\pi_1',\pi_2'}(f)}
{\vol(K')^2}
=\frac 14\left|\int_{G'(\Q) \backslash G'(\A)}\phi_3(h)\ dh\right|^2
=\frac 14|\langle \phi_3,\1 \rangle|^2.\]
By orthogonality,
$\langle \phi_3,\1 \rangle=0$.
So, for the last case in \eqref{specterms},
\[\sum_{\pi_1',\pi_2'\in\cA_{res}(G')}I_{\pi_1',\pi_2'}(f)=0.\]

\begin{Lemma}\label{char.unram}
Let $F$ be $\Q$ or
an imaginary quadratic extension of $\Q$ with class number 1.
Then the set $X^{\text{un}}(F)$ of all characters
$\delta:F^\times\backslash\A_F^\times\to\CC^\times$
which are unramified everywhere
is parameterized by
$\widehat{F_\infty^\times/\cO_F^\times}$,
the set of characters of
$F_\infty^\times$ which are invariant under $\cO_F^\times$.
Here $\cO_F^\times$ is the group of units in
the ring $\cO_F$ of algebraic integers,
which is actually the set of roots of unity in $F$.
\end{Lemma}
\begin{proof}
This is a direct corollary of the strong approximation theorem:
\[\cO_F^\times\backslash F_\infty^\times\times
\prod_{v<\infty}\cO_{F_v}^\times\cong F^\times\backslash\A_F^\times;\]
and the Dirichlet's unit theorem.
\end{proof}

\subsection{Res \texorpdfstring{$\otimes$}{x} Cusp}

In the same way as in the previous section we can show that
\begin{Lemma}
For $\pi_2'\in\cA_{cusp}(G')$,
$R(f)(\delta\circ N_D\otimes \phi_2)\equiv 0$
unless $\delta$ is unramified everywhere,
$2k=2$, $\pi_2'\in\cF'(N,2)$ and $\phi_2\in(\pi_2')^{ K_{\fin}}$,
in which case
\[R(f)(\delta\circ N_D\otimes \phi_{2})(x_1,x_2)
=\delta(N_D(x_1))\phi_{2}(x_2)\vol(K')^2,\]
where $K'=G'_\infty \prod_{p<\infty} K_p$.
\end{Lemma}
\begin{proof}
The proof on the non-Archimedean places
has been done in the proof of Lemmas \ref{local.cuspcusp}
and \ref{resres}.
On the Archimedean place, one can apply a similar proof
as in Lemma \ref{resres} to show that,
for $\pi_{1,\infty}'=\delta_\infty\circ N_{D_\infty}$,
$\pi_{2,\infty}'=\pi'_{2k'}$,
\[\pi_{1,\infty}'\otimes\pi_{2,\infty}'(f_\infty)
(\delta_\infty\circ N_{D_\infty}\otimes \phi_{2,\infty})= 0\]
unless $2k=2k'=2$,
in which case $\phi_{2,\infty}=\1$ and
\[\pi_{1,\infty}'\otimes\pi_{2,\infty}'(f_\infty)
(\delta_\infty\circ N_{D_\infty}\otimes \phi_{2,\infty})
=(\delta_\infty\circ N_{D_\infty}\otimes \phi_{2,\infty})\vol(G'_\infty)^2.\]
\end{proof}

For $\pi'\in\cA_{cusp}(G')$, let
$\phi_{\pi'}\in\CC X^{2k-2}\otimes(\pi'_{\fin})^{ K_{\fin}}$
be the normalized new-line vector as defined in Lemma \ref{autoset}.
Then we can write
\begin{multline*}
\frac{\sum\limits_{\pi_1'\in\cA_{res}(G')}
\sum\limits_{\pi_2'\in\cA_{cusp}(G')}
I_{\pi_1',\pi_2'}(f)}
{\vol(K')^2}\\
=\begin{cases}
\sum\limits_{\delta_1\in X^\text{un}(\Q)}
\sum\limits_{\pi_2'\in\cF'(N,2)}
\left|\int_{G'(\Q) \backslash G'(\A)}
\dfrac{\delta_1(N_D(h))}{\vol(G'(\Q)\backslash G'(\A))^{1/2}}
\phi_{\pi'_2}(h)\phi_3(h)\ dh\right|^2,
&\text{if }2k=2;\\
0,&\text{otherwise}.\end{cases}
\end{multline*}
Here $X^{\text{un}}(\Q)=\{\1\}$. So when $2k=2$,
\[\begin{split}
\frac{\sum\limits_{\pi_1'\in\cA_{res}(G')}\sum\limits_{\pi_2'\in\cA_{cusp}(G')}
I_{\pi_1',\pi_2'}(f)}
{\vol(K')^2}
=\ &\frac 12\sum_{\pi_2'\in\cF'(N,2)}
\left|\int_{G'(\Q) \backslash G'(\A)}
\phi_{\pi'_2}(h)\phi_3(h)\ dh\right|^2\\
=\ &\frac 12\sum_{\pi_2'\in\cF'(N,2)}
\left|\langle \phi_3,\overline{\phi_{\pi'_2}} \rangle \right|^2.
\end{split}\]
By orthogonality,
$\langle \phi_3,\overline{\phi_{\pi'_2}} \rangle\neq 0$
only when $\phi_{\pi'_2}=\overline{\phi_3}/\|\phi_3\|$.
So
\[\frac{\sum\limits_{\pi_1'\in\cA_{res}(G')}\sum\limits_{\pi_2'\in\cA_{cusp}(G')}
I_{\pi_1',\pi_2'}(f)}
{\vol(K')^2}
=\frac 12\left|\langle
\phi_3,\frac{\phi_3}{\|\phi_3\|}\rangle\right|^2
=\frac 12\langle \phi_3,\phi_3 \rangle.\]
This completes the proof of \eqref{specterms}.


\section{Geometric Side of the RTF}\label{section.geometric}

Recall that $G'=Z\backslash D^\times$,
$K_p$ is the image of $Z_p\cO_p^\times$ in $G'_p$;
\begin{gather*}
f(g_1,g_2)=f_\infty(g_1,g_2)
\prod_{p<\infty}\1_{K_p}(g_1)\1_{K_p}(g_2),\\
f_\infty(g_1,g_2) = \overline{
\langle \pi'_{2k}\otimes\pi'_{2k}(g_1,g_2) w^\circ_{2k}, w^\circ_{2k}\rangle}/
\langle w^\circ_{2k}, w^\circ_{2k}\rangle,\quad
w^\circ_{2k}=\left(-Y_1Y_2\left|\begin{matrix}X_1&X_2\\
Y_1&Y_2\end{matrix}\right|\right)^{k-1};\\
\mathbb{P}_{2k}=
\left|\begin{matrix}X_1&X_2\\Y_1&Y_2\end{matrix}\right|^{k-1}\otimes
\left|\begin{matrix}X_2&X_3\\Y_2&Y_3\end{matrix}\right|^{k-1}\otimes
\left|\begin{matrix}X_3&X_1\\Y_3&Y_1\end{matrix}\right|^{k-1};
\end{gather*}
$\phi_3$ is a new-line vector in $\pi_3'\in\cF'(N,2k)$
with $\phi_{3,\infty}$ a highest weight vector of $\pi_{2k}'$,
as defined in Lemma \ref{autoset}.
In Section \ref{proofoforbdecomp} we will prove the following theorem
which gives the orbital expansion of the distribution
\[I(f) = \int_{G'(\Q) \backslash G'(\A)} \int_{G'(\Q) \backslash G'(\A)}
K_f(h_1,h_1;h_2, h_2) \phi_3(h_1)\overline{\phi_3(h_2)} \ dh_1 \ dh_2.\]

\begin{Theorem}\label{orbdecomp}
Let $N$ be a square-free integer which has an odd number of prime factors.
Let $D$ be the quaternion algebra $\Q$ which is ramified
precisely at $\infty$ and the primes dividing $N$.
Then
\[I(f)=I_{[1]}+I_{[\gamma_0]}+I_{[\gamma_1]}\]
where $\gamma_0$, $\gamma_1\in D^\times(\Q)$ such that
\[\Tr_D(\gamma_0)=0,\quad N_D(\gamma_0)=1;\quad
\Tr_D(\gamma_1)=N_D(\gamma_1)=1,\]
with \begin{gather*}
I_{[1]}=\langle \phi_3,\phi_3\rangle
\frac{\vol(K')}{2k-1},\\
I_{[\gamma_0]}=\frac 12\vol(K')
\int_{\A_{E_0}^\times \backslash D^\times(\A)}\varphi_{\gamma_0}(h)
\int_{\A^\times E_0^\times \backslash \A_{E_0}^\times}
(R(h)\phi^*)(t)\overline{(R(h)\phi^{**})(t)}\ dt\ dh,\\
I_{[\gamma_1]}=\vol(K')
\int_{\A_{E_1}^\times \backslash D^\times(\A)}\varphi_{\gamma_1}(h)
\int_{\A^\times E_1^\times \backslash \A_{E_1}^\times}
(R(h)\phi^*)(t)\overline{(R(h)\phi^{**})(t)}\ dt\ dh.
\end{gather*}
Here $K'=G'_\infty\prod_{p<\infty} K_p$;
$E=\Q(\gamma)$ is the quadratic extension of $\Q$
which can be embedded in $D$ when $\gamma$ exists
(in particular $E_0=\Q(\gamma_0)=\Q(\sqrt{-1})$,
$E_1=\Q(\gamma_1)=\Q(\sqrt{-3})$);
$\phi^*$, $\phi^{**}\in\pi_3'$ such that
\begin{equation}\label{def.e_gamma}
\begin{gathered}
\phi^*=\phi_3;\quad \phi^{**}=\otimes\phi^{**}_v,\quad
\phi^{**}_\infty=\frac{\|\phi_3\|}{\|\mathbb{P}_{2k}\|^2}e_{\gamma},\quad
\phi^{**}_p=\phi_{3,p};\\
e_\gamma=\sum_{i=0}^{2k-2}
\binom{2k-2}{i}
\langle \pi'_{2k}\otimes\pi'_{2k}\otimes\pi'_{2k}
(h^{-1}\gamma h,1,1) \mathbb{P}_{2k},
w^\circ_{2k}\otimes X_3^{2k-2-i}Y_3^{i}\rangle
X_3^{2k-2-i}Y_3^{i};\end{gathered}
\end{equation}
$\varphi_\gamma$ is the characteristic function of the set
\[\{h\in G'_\gamma(\A) \backslash G'(\A): h_{p}^{-1}\gamma h_{p}\in K_p
\text{ for all primes }p\}.\]
And \begin{itemize}
\item $I_{[\gamma_0]}=0$ if
$N$ has a prime factor of the form $4n+1$
(in this case $\gamma_0$ does not exist);
\item $I_{[\gamma_1]}=0$
if $N$ has a prime factor of the form $3n+1$
(in this case $\gamma_1$ does not exist).
\end{itemize}
(These two primes do not have to be distinct.)
\end{Theorem}

For some particular $N$ only the trivial orbit appears
in the orbital decomposition. In fact we have:
\begin{Corollary}\label{maingeometric2k}
With assumptions and notations as before,
if in addition $N$ has a prime factor $\equiv 1\pmod 4$
and one $\equiv 1\pmod 3$,
we have
\[\frac{I(f)}{\vol(K')}=
\frac{\langle \phi_3,\phi_3\rangle}{2k-1}.\]
With Theorem \ref{mainspectral} we have that
\[\frac{2^{\omega(N)}}{2 N^2}
\sum_{\substack{\pi_1,\pi_2\in\cF(N,2k)\\
\varepsilon_p=-1,\ \forall p\mid N}}
\frac{L(\frac 12,\pi_1\otimes\pi_2\otimes\pi_3)}
{L(1,\pi_1\otimes\pi_2\otimes\pi_3,\Ad)}
=\begin{cases}
1-\dfrac{24}{\varphi(N)},&\text{if }2k=2;\\
\dfrac{\Gamma(k)^{3}\Gamma(3k-1)}
{\Gamma(2k-1)^2\Gamma(2k)}
,&\text{otherwise.}
\end{cases}\]
\end{Corollary}

This corollary and \eqref{adelizedL} \eqref{adelizedL2}
together give a proof of \eqref{maineq2}.

After proving Theorem \ref{orbdecomp} we will use
Waldspurger's formula (Theorem \ref{Waldspurger})
to compute $I_{[\gamma_0]}$ and $I_{[\gamma_1]}$
in Section \ref{section.harmonic}.
This combined with the calculations 
in Sections \ref{section.nonarchi.ram}, \ref{section.nonarchi.unram}
and \ref{section.geo.archi}
gives a proof of the following theorem.


\begin{Theorem}[Main Theorem, geometric side]\label{maingeometric}
Let $N$ be a square-free product of an odd number of primes.
For $\pi_3'\in\cF'(N,2k)$,
\[I(f)=I_{[1]}+I_{[\gamma_0]}+I_{[\gamma_1]}\]
with \[\begin{split}
\frac{I_{[1]}}
{\langle \phi_3,\phi_3\rangle\vol(K')}=&\ \frac{1}{2k-1};\\
\frac{I_{[\gamma_0]}}
{\langle \phi_3,\phi_3\rangle\vol(K')}=&\ 4
\frac{\Gamma(2k-1)^4}{(2\pi)^{2k}\Gamma(k)^3\Gamma(3k-1)}
\frac {2^{\omega(N)}}{N\cdot L(1,\pi_3,\Ad)} \\
&\ \cdot 2^{\ord_2(N)}
\prod_{p\mid N}\frac{1-\chi_{-4}(p)}{2}
\sum_{\Omega\in\widehat{[E_0^\times]}}
I_0(\Omega)\cdot L_{\fin}(\frac 12,(\pi_3)_{E_0}\otimes \Omega),\\
\frac{I_{[\gamma_1]}}
{\langle \phi_3,\phi_3\rangle\vol(K')}=&\ 6\sqrt{3}
\frac{\Gamma(2k-1)^4}{(2\pi)^{2k}\Gamma(k)^3\Gamma(3k-1)}
\frac {2^{\omega(N)}}{N\cdot L(1,\pi_3,\Ad)} \\
&\ \cdot 2^{\ord_3(N)}
\prod_{p\mid N}\frac{1-\chi_{-3}(p)}{2}
\sum_{\Omega\in\widehat{[E_1^\times]}}
I_1(\Omega)\cdot L_{\fin}(\frac 12,(\pi_3)_{E_1}\otimes \Omega).
\end{split}\]
Here $E_0=\Q(\sqrt{-1})$, $E_1=\Q(\sqrt{-3})$,
$I_0$ and $I_1$ are constants depending only on $k$ and $\Omega$
(and on $a_2(h)$, $a_3(h)$ when $2$ or $3\mid N$ respectively).
More precisely,
$I_0$ and $I_1$ vanish unless $\Omega$ satisfies 
the restrictions in Lemma \ref{nonvanishOmega} respectively,
in which case we define
\[I_0(\Omega)=\mathbb{I}_{2k}^{(m)}(\gamma_0),\quad 
I_1(\Omega)=\mathbb{I}_{2k}^{(m)}(\gamma_1).\]
Here $\mathbb{I}_{2k}^{(m)}(\gamma)$ is defined by
\[\mathbb{I}_{2k}^{(m)}(\gamma)=\binom{2k-2}{k-1+m}^{-1}
\sum_{\substack{0\leq i,j\leq 2k-2\\i+j=2(k-1)-m}}\gamma^{2(k-1-i)}
\binom{2k-2}i^{-1}\binom{2k-2}j^{-1}|C_{i,j,k-1+m}|^2,\]
with $C_{i,j,r}$ the coefficient 
of $X_1^{2k-2-i}Y_1^i\otimes X_2^{2k-2-j}Y_2^j
\otimes X_3^{2k-2-r}Y_3^{r}$
in \[\mathbb{P}_{2k}=
\left|\begin{matrix}X_1&X_2\\Y_1&Y_2\end{matrix}\right|^{k-1}\otimes
\left|\begin{matrix}X_2&X_3\\Y_2&Y_3\end{matrix}\right|^{k-1}\otimes
\left|\begin{matrix}X_3&X_1\\Y_3&Y_1\end{matrix}\right|^{k-1}.\]
In particular for $\Omega=\1$ the trivial character,
\[\begin{split}
I_0(\1)&=
\begin{cases}
\mathbb{I}_{2k}^{(0)}(\gamma_0)=
\dfrac{\Gamma(k)^2}{\Gamma(2k-1)}
\sum\limits_{i=0}^{2k-2}\gamma_0^{2(k-1-i)}
\binom{2k-2}i^{-2}|C_{i,2k-2-i,k-1}|^2,&\text{if }2\nmid N,\\
\frac 12(1-\varepsilon_2(\frac 12, \pi_3))
\mathbb{I}_{2k}^{(0)}(\gamma_0),&\text{if }2\mid N;\\
\end{cases}\\
I_1(\1)&=
\begin{cases}
\mathbb{I}_{2k}^{(0)}(\gamma_1)=
\dfrac{\Gamma(k)^2}{\Gamma(2k-1)}
\sum\limits_{i=0}^{2k-2}\gamma_1^{2(k-1-i)}
\binom{2k-2}i^{-2}|C_{i,2k-2-i,k-1}|^2,&\text{if }3\nmid N,\\
\frac 12(1-\varepsilon_3(\frac 12, \pi_3))
\mathbb{I}_{2k}^{(0)}(\gamma_1),&\text{if }3\mid N.\\
\end{cases}
\end{split}\]

\end{Theorem}

Theorem \ref{mainspectral} and \ref{maingeometric}
together imply
the Main Theorem \ref{mainadelic}.

\subsection{Orbital decomposition}\label{proofoforbdecomp}

We apply the geometric side of the RTF to the distribution $I(f)$.
When $G=G'\times G'$ and $H_1=H_2=G'$ ($H_1,H_2\hookrightarrow G$ diagonally),
the representatives $[(\gamma_1,\gamma_2)]$
in $G'(\Q)\backslash G'(\Q)\times G'(\Q)/G'(\Q)$
can be chosen such that $\gamma_2=1$
and $[\gamma_1]$ runs through all conjugacy classes of $G'(\Q)$.
For $\theta_1,\theta_2\in G'(\Q)$,
$\theta_1^{-1}(\gamma,1)\theta_2=(\gamma,1)$ if and only if
$\theta_1=\theta_2\in G'_\gamma(\Q)$, the centralizer of $\gamma$ in $G'(\Q)$.
By the orbital expansion of the RTF, we have
\[\begin{split}
I(f) =& \ \int_{G'(\Q) \backslash G'(\A)} \int_{G'(\Q) \backslash G'(\A)}
K_f(h_1,h_1; h_2,h_2) \phi_3(h_1)\overline{\phi_3(h_2)} \ dh_1 \ dh_2\\
=&\ \iint_{(G'(\Q) \backslash G'(\A))^2}
\sum_{\gamma_1,\gamma_2\in G'(\Q)}f(h_1^{-1}\gamma_1 h_2,h_1^{-1}\gamma_2 h_2)
\phi_3(h_1)\overline{\phi_3(h_2)} \ dh_1 \ dh_2\\
=&\ \iint_{(G'(\Q) \backslash G'(\A))^2}
\sum_{\substack{[\gamma_1,\gamma_2]=[\gamma,1]\\\gamma\in[G'(\Q)]}}
\sum_{(\theta_1,\theta_2)}
f(h_1^{-1}\theta_1^{-1}\gamma\theta_2h_2,h_1^{-1}\theta_1^{-1}\theta_2h_2)
\phi_3(h_1)\overline{\phi_3(h_2)} \ dh_1 \ dh_2
\end{split}\]
Here $(\theta_1,\theta_2)$ runs through \[G'(\Q)\times G'(\Q)/
\{(\theta_1,\theta_2):\theta_1=\theta_2\in G'_\gamma(\Q)\}
\cong (G'_\gamma(\Q)\backslash G'(\Q))\times G'(\Q),\]
i.e. for a fixed $\theta_1\in G'_\gamma(\Q)\backslash G'(\Q)$,
$\theta_2$ runs through $G'(\Q)$.
So we have
\[I(f) =\sum_{[\gamma]}\sum_{\theta_1,\theta_2}
\iint_{(G'(\Q) \backslash G'(\A))^2}
f\big((\theta_1h_1)^{-1}\gamma (\theta_2 h_2),(\theta_1h_1)^{-1}(\theta_2 h_2)\big)
\phi_3(\theta_1h_1)\overline{\phi_3(\theta_2h_2)} \ d\theta_1h_1 \ d\theta_2h_2\]
where $[\gamma]$ runs through conjugacy classes of $G'(\Q)$.
Let $\theta_i h_i$ be the new $h_i$ ($i=1,2$). Then
\[I(f) =\sum_{[\gamma]}\int_{G'_\gamma(\Q) \backslash G'(\A)}
\left(\int_{G'(\A)}
f(h_1^{-1}\gamma h_2,h_1^{-1}h_2) \overline{\phi_3(h_2)} \ dh_2
\right)\phi_3(h_1) \ dh_1.\]
We split $h_1=th$ with $h\in G'_\gamma(\A) \backslash G'(\A)$
and $t\in G'_\gamma(\Q) \backslash G'_\gamma(\A)$,
and let $g=t^{-1}h_2$. Then
\begin{equation}\label{def.I_gamma}\begin{split}
I(f)=& \ \sum_{[\gamma]}\int_{G'_\gamma(\Q) \backslash G'_\gamma(\A)}
\int_{G'_\gamma(\A) \backslash G'(\A)}
\left(\int_{G'(\A)}
f((th)^{-1}\gamma (tg),(th)^{-1}tg) \overline{\phi_3(tg)} \ d(tg)
\right)\phi_3(th) \ dh\ dt\\
=&\ \sum_{[\gamma]}\int_{G'_\gamma(\Q) \backslash G'_\gamma(\A)}
\int_{G'_\gamma(\A) \backslash G'(\A)}
\left(\int_{G'(\A)}
f(h^{-1}\gamma g,h^{-1}g) \overline{\phi_3(tg)} \ dg
\right)\phi_3(th) \ dh\ dt.
\end{split}\end{equation}
Denote the summand as $I_{[\gamma]}(f)$.

Recall that the new-line vector
$\phi_3=\otimes\phi_{3,v}\in \CC X_{3}^{2k-2}
\otimes \pi_{3,\fin}'^{ K_{\fin}}$
can be written such that
$\phi_{3,\infty}=\|\phi_3\|X_{3}^{2k-2}$ and $\phi_{3,p}$'s
are $K_p$-invariant unit vectors for $p<\infty$
(see Lemma \ref{autoset}).
Let $f=f_\infty\cdot (f_0\otimes f_0)$
be the test function defined in \eqref{def.testfunction}.
Then

\begin{Lemma}
$\int_{G'(\A)} \overline{
f(h^{-1}\gamma g,h^{-1}g)} (R(g)\phi_3) \ dg$ is a pure tensor in $\pi_3'$.
\end{Lemma}
\begin{proof} For $f=f_\infty\cdot (f_0\otimes f_0)$,
\begin{multline*}
\int_{G'(\A)}
\overline{f(h^{-1}\gamma g,h^{-1}g)}(R(g)\phi_3) \ dg
=\|\phi_3\|\int_{G'_\infty}\overline{f_\infty(h_\infty^{-1}\gamma g_\infty,h_\infty^{-1}g_\infty)}
(\pi'_{2k}(g_\infty)X_3^{2k-2})\ dg_\infty\\
\cdot \prod\limits_{p<\infty}\int_{G'_p}
\overline{\1_{K_p}(h_{p}^{-1}\gamma g_p)\ \1_{K_p}(h_{p}^{-1}g_p) }
R(g_p)\phi_{3,p}\ dg_p.
\end{multline*}

For $v=p< \infty$ the local test function
$\1_{K_p}(h_{p}^{-1}\gamma h_{2,p})\ \1_{K_p}(h_{p}^{-1}h_{2,p})$
is nonzero only if both
$h_{p}^{-1}\gamma g_p, \ h_{p}^{-1} g_p\in K_p$,
and hence
\[\int_{G'_p} \overline{\1_{K_p}(h_{p}^{-1}\gamma g_p)\
\1_{K_p}(h_{p}^{-1}g_p) }
R(g_p)\phi_{3,p} \ dg_p
=\int_{h_{p}K_p\cap \gamma^{-1}h_{p}K_p}
R(g_p)\phi_{3,p} \ dg_p.\]
These two left cosets either coincide or are disjoint,
and $h_{p}K_p= \gamma^{-1}h_{p}K_p$ if and only if
$h_{p}^{-1}\gamma h_{p}\in K_p$.
Let $\varphi_\gamma=\prod\varphi_{\gamma,p}$ and
$\varphi_{\gamma,p}:G'_\gamma(\Q_p) \backslash G'_p\to\CC$ be
the characteristic function of the set $\{h_p: h_{p}^{-1}\gamma h_{p}\in K_p\}$.
Since $\phi_{3,p}$ is $K_p$-invariant, we have
\begin{multline}\label{simplify.I_gamma.nonarch}
\int_{G'_p} \overline{\1_{K_p}(h_{p}^{-1}\gamma g_p)\
\1_{K_p}(h_{p}^{-1}g_p) }
R(g_p)\phi_{3,p} \ dg_p\\
=\varphi_{\gamma,p}(h_p)\int_{h_{p}K_p}
R(g_p)\phi_{3,p} \ dg_p
=\vol(K_p)\varphi_{\gamma,p}(h_p)R(h_p)\phi_{3,p}.
\end{multline}

For $v=\infty$, by the following lemma we have
\begin{equation}\label{simplify.I_gamma.arch}
\int_{G'_\infty}
\overline{f_\infty(h_\infty^{-1}\gamma g_\infty,h_\infty^{-1}g_\infty)}
\pi'_{2k}(g_\infty)X_3^{2k-2}\ dg_\infty
=\frac{\vol(G'_\infty)}{\|\mathbb{P}_{2k}\|^2}\pi'_{2k}(h_{\infty})e_{\gamma},
\end{equation}
where $e_\gamma$ is defined in \eqref{def.e_gamma}.
\end{proof}

\begin{Lemma}\label{def.edelta}
For $\gamma\in G'(\Q)$, $h\in G'_\gamma(\R) \backslash G'_\infty$,
with $e_{\gamma}$ defined in \eqref{def.e_gamma},
\[\int_{G'_\infty}
\langle \pi'_{2k}\otimes\pi'_{2k}
(h^{-1}\gamma g,h^{-1}g) w^\circ_{2k}, w^\circ_{2k}\rangle
\pi'_{2k}(g)X_3^{2k-2}\ dg\\
=\vol(G'_\infty)\frac{\|w^\circ_{2k}\|^2}{\|\mathbb{P}_{2k}\|^2}
\pi'_{2k}(h)e_{\gamma}.\]
\end{Lemma}

\begin{proof}
First we consider the inner product
\[\begin{split}
&\ \left\langle\int_{G'_\infty}
\langle \pi'_{2k}\otimes\pi'_{2k} (h^{-1}\gamma g,h^{-1}g) w^\circ_{2k}, w^\circ_{2k}\rangle
\pi'_{2k}(g)X_3^{2k-2}\ dg,\pi'_{2k}(h)X_3^{2k-2-i}Y_3^{i}
\right\rangle\\
=&\ \int_{G'_\infty}
\langle \pi'_{2k}\otimes\pi'_{2k} (g,g) w^\circ_{2k},
\pi'_{2k}\otimes\pi'_{2k} (\gamma^{-1} h,h)w^\circ_{2k}\rangle
\langle\pi'_{2k}(g)X_3^{2k-2},\pi'_{2k}(h)X_3^{2k-2-i}Y_3^{i}\rangle\ dg\\
=&\ \int_{G'_\infty}
\langle \pi'_{2k}\otimes\pi'_{2k}\otimes\pi'_{2k} (g,g,g) w^\circ_{2k}\otimes X_3^{2k-2},
\pi'_{2k}\otimes\pi'_{2k}\otimes\pi'_{2k}
(\gamma^{-1} h,h,h)w^\circ_{2k}\otimes X_3^{2k-2-i}Y_3^{i}\rangle\ dg.
\end{split}\]
By Lemma \ref{Schur}
the above integral is equal to
\[\vol(G'_\infty)
\langle w^\circ_{2k}\otimes X_3^{2k-2},
\frac{\mathbb{P}_{2k}}{\|\mathbb{P}_{2k}\|}\rangle
\overline{\langle
\pi'_{2k}\otimes\pi'_{2k}\otimes\pi'_{2k}
(\gamma^{-1} h,h,h)w^\circ_{2k}\otimes X_3^{2k-2-i}Y_3^{i},
\frac{\mathbb{P}_{2k}}{\|\mathbb{P}_{2k}\|}\rangle}.\]
Recall that $\mathbb{P}_{2k}$ is $G'_\infty$-invariant, and we have
$\langle w^\circ_{2k}\otimes X_3^{2k-2},\mathbb{P}_{2k}\rangle
=\|w^\circ_{2k}\|^2$
by Lemma \ref{Choosew}.
Now the above integral is equal to 
$\vol(G'_\infty)\frac{\|w^\circ_{2k}\|^2}{\|\mathbb{P}_{2k}\|^2}$ times
\[\begin{split}
&\ \langle \pi'_{2k}\otimes\pi'_{2k}\otimes\pi'_{2k}
(\gamma,1,1) \mathbb{P}_{2k},
\pi'_{2k}\otimes\pi'_{2k}\otimes\pi'_{2k}
(h,h,h)w^\circ_{2k}\otimes X_3^{2k-2-i}Y_3^{i}\rangle\\
=&\ \langle \pi'_{2k}\otimes\pi'_{2k}\otimes\pi'_{2k}
(\gamma h,h,h) \mathbb{P}_{2k},
\pi'_{2k}\otimes\pi'_{2k}\otimes\pi'_{2k}
(h,h,h)w^\circ_{2k}\otimes X_3^{2k-2-i}Y_3^{i}\rangle\\
=&\ \langle \pi'_{2k}\otimes\pi'_{2k}\otimes\pi'_{2k}
(h^{-1}\gamma h,1,1) \mathbb{P}_{2k},
w^\circ_{2k}\otimes X_3^{2k-2-i}Y_3^{i}\rangle.
\end{split}\]

Recall that, for a fixed $h$,
$\{\pi'_{2k}(h)X_3^{2k-2-i}Y_3^{i}\}$ 
forms an orthogonal basis of $V_{\pi'_{2k}}$,
with the inner product defined as \eqref{innerproduct}.
So we have
\begin{small}
\[\begin{split}
&\ \int_{G'_\infty}
\langle \pi'_{2k}\otimes\pi'_{2k} (h^{-1}\gamma g,h^{-1}g) w^\circ_{2k}, w^\circ_{2k}\rangle
\pi'_{2k}(g)X_3^{2k-2}\ dg\\
=&\ \sum_{i=0}^{2k-2}\frac{\left\langle\int_{G'_\infty}
\langle \pi'_{2k}\otimes\pi'_{2k} (h^{-1}\gamma g,h^{-1}g) w^\circ_{2k}, w^\circ_{2k}\rangle
\pi'_{2k}(g)X_3^{2k-2}\ dg,\pi'_{2k}(h)X_3^{2k-2-i}Y_3^{i}
\right\rangle}
{\langle\pi'_{2k}(h)X_3^{2k-2-i}Y_3^{i},\pi'_{2k}(h)X_3^{2k-2-i}Y_3^{i}\rangle}
\pi'_{2k}(h)X_3^{2k-2-i}Y_3^{i}\\
=&\ \sum_{i=0}^{2k-2}\frac{\vol(G'_\infty)
\frac{\|w^\circ_{2k}\|^2}{\|\mathbb{P}_{2k}\|^2}
\langle \pi'_{2k}\otimes\pi'_{2k}\otimes\pi'_{2k}
(h^{-1}\gamma h,1,1) \mathbb{P}_{2k},
w^\circ_{2k}\otimes X_3^{2k-2-i}Y_3^{i}\rangle}
{\langle X_3^{2k-2-i}Y_3^{i}, X_3^{2k-2-i}Y_3^{i}\rangle}
\pi'_{2k}(h)X_3^{2k-2-i}Y_3^{i}\\
=&\ \vol(G'_\infty)\frac{\|w^\circ_{2k}\|^2}{\|\mathbb{P}_{2k}\|^2}
\sum_{i=0}^{2k-2}\binom{2k-2}{i}
\langle \pi'_{2k}\otimes\pi'_{2k}\otimes\pi'_{2k}
(h^{-1}\gamma h,1,1) \mathbb{P}_{2k},
w^\circ_{2k}\otimes X_3^{2k-2-i}Y_3^{i}\rangle
\pi'_{2k}(h)X_3^{2k-2-i}Y_3^{i}.
\end{split}\]\end{small}
\end{proof}

Take $\phi^*,\phi^{**}\in\pi_3'$ as in \eqref{def.e_gamma}.
By \eqref{simplify.I_gamma.nonarch} and \eqref{simplify.I_gamma.arch}
we can write $I_{[\gamma]}(f)$ (defined in \eqref{def.I_gamma}) as
\[\begin{split}
I_{[\gamma]}(f)=&\ \int_{G'_\gamma(\Q) \backslash G'_\gamma(\A)}
\int_{G'_\gamma(\A) \backslash G'(\A)}
\overline{\vol(K')(R(h)\phi^{**})(t)
\prod_{p<\infty}\varphi_{\gamma,p}(h_p)}
(R(h)\phi^*)(t) \ dh\ dt\\
=&\ \vol(K')
\int_{G'_\gamma(\A) \backslash G'(\A)}\varphi_{\gamma}(h)
\int_{G'_\gamma(\Q) \backslash G'_\gamma(\A)}
(R(h)\phi^*)(t)\overline{(R(h)\phi^{**})(t)}\ dt\ dh,
\end{split}\]
with $K'$ and $\varphi_{\gamma}$ defined in Theorem \ref{orbdecomp}.

When $\gamma=1$, the centralizer $G_\gamma'=G'$.
According to Lemma \ref{Choosew},
\[e_{\gamma}=\|w^\circ_{2k}\|^2 X_3^{2k-2}.\]
Then by Lemma \ref{length},
\[\begin{split}
\frac{I_{[1]}(f)}{\vol(K')}
=&\ \int_{G'(\Q) \backslash G'(\A)}
\phi^*(t)\overline{\phi^{**}(t)}\ dt
=\langle\phi^{*},\phi^{**}\rangle
=\prod_{v}\langle\phi^{*}_v,\phi^{**}_v\rangle_v\\
=&\ \|\phi_3\|^2\langle X_3^{2k-2},
\frac{1}{\|\mathbb{P}_{2k}\|^2}e_{\gamma}\rangle_\infty
=\|\phi_3\|^2
\frac{\|w^\circ_{2k}\|^2}{\|\mathbb{P}_{2k}\|^2}
=\frac{\|\phi_3\|^2}{2k-1}.
\end{split}\]

Now we study the property for the other $[\gamma]$'s
such that $\varphi_\gamma$ is not identically 0.
Instead of $\Q$, we consider it over an arbitrary
number field $F$.

\begin{Lemma}\label{nonzeroterms}
$\varphi_\gamma=0$
unless $\Tr_D(\gamma)\in\{\pm 1\}\backslash\cO_F$
and $N_D(\gamma)\in \cO_F^\times/(\cO_F^\times)^2$.
In particular, when $F=\Q$, $\varphi_\gamma=0$
unless $\Tr_D(\gamma')\in  \Z_{\geq 0}$
and $N_D(\gamma')=\pm 1$ for 
some representative $\gamma'\in D^\times(F)$ of $[\gamma]$.
\end{Lemma}
\begin{proof}
Fix a set $\Sigma\subset\cO_F-\{0\}$ of representatives in $F^\times/(F^\times)^2$.
We can choose $\Sigma$ to be the set of ``square-free'' integers. More precisely,
the factorization of the principal ideal of $\cO_F$ generated by any number in $\Sigma$
has exactly one factor for each prime ideal that appears in it.
Then as in Lemma \ref{parametrization}
we can fix a representative of $[\gamma]$ in $D^\times(F)$ (also denoted $\gamma$)
so that $N_D(\gamma)\in\Sigma$.
Under this assumption we see that $\ord_v(N_D(\gamma))$ is either 0 or 1.

Suppose that $\varphi_\gamma(h)\neq 0$.
This means $h_{v}^{-1}\gamma h_{v}\in K_v$
for all $v<\infty$.

When $v\notin\Ram(D)$, $K_v=\PGL(2,\cO_{F_v})$.
We can say that,
fixing a representative of $\gamma$ in $\GL(2,F_v)$
there is an $h_v\in \GL(2,F_v)$
such that $\lambda_v h_v^{-1}\gamma h_v\in \GL(2,\cO_{F_v})$
for some $\lambda_v\in F_v^\times$.
As a matrix in $\GL(2,\cO_{F_v})$,
we have that \[\Tr_{D_v}(\lambda_v h_v^{-1}\gamma h_v)\in \cO_{F_v},\quad
N_{D_v}(\lambda_v h_v^{-1}\gamma h_v)\in \cO_{F_v}^\times.\]
Conjugate matrices have the same norm and trace, so
\[\ord_v(\Tr_{D_v}(\lambda_v\gamma))=\ord_v(\lambda_v\Tr_{D}(\gamma))\geq 0,\quad
\ord_v(N_{D_v}(\lambda_v\gamma))=\ord_v(\lambda_v^2 N_{D}(\gamma))=0.\]
With the assumption of $\Sigma$, one can imply that $\ord_v(N_{D}(\gamma))=0$,
$\lambda_v\in\cO_{F_v}^\times$ is a unit,
and then $\ord_v(\Tr_{D}(\gamma))\geq 0$.

When $v\in\Ram(D)$ and $v<\infty$,
$D_v$ is a division algebra and it has only one maximal order
\[\cO_v=\{x\in D_v: N_{D_v}(x)\in\cO_{F_v}\}.\]
We still know that,
there is an $h_v\in D_v^\times$
such that $\lambda_v h_v^{-1}\gamma h_v\in \cO_v^\times$
for some $\lambda_v\in F_v^\times$.
With the assumption of $\Sigma$,
as in the previous case,
one can imply that $\ord_v(N_{D}(\gamma))=0$,
i.e. $\gamma\in\cO_v^\times$.
Moreover, Proposition \ref{maxorder}
shows that $\gamma$ is $\cO_{F_v}$-integral and
then $\ord_v(\Tr_{D}(\gamma))\geq 0$.

Globally we see that, fixing $\gamma$ such that $N_D(\gamma)\in\Sigma$,
for any place $v$ of $F$, $N_D(\gamma)$ cannot be divisible by $v$,
and the order of $v$ in the ideal decomposition of $\Tr_D(\gamma)\in F$
has to be non-negative.
In other words $N_D(\gamma)$ has to be a unit and $\Tr_D(\gamma)$ in $\cO_F$.
With Lemma \ref{parametrization} we get the conclusion.
\end{proof}

With this lemma, we only need to
consider the summand $I_{[\gamma]}(f)$
with $N_D(\gamma)$ a ``square-free'' unit in $\cO_F^\times$
and $\Tr_D(\gamma)\in\{\pm 1\}\backslash\cO_F$.
Going back to the case when $F=\Q$,
we only need to consider the summand $I_{[\gamma]}(f)$
with $N_D(\gamma)=1$ and $\Tr_D(\gamma)\in\Z_{\geq 0}$
(obviously $N_D(\gamma)$ cannot be $-1$ when $D$ is definite).
Now we consider the infinite place $\Q_\infty=\R$.
When $D$ is definite, $D(\R)\cong\mathbb{H}$ is non-split.
That means the characteristic polynomial \[X^2-\Tr_D(x)X+N_D(x)\]
of any $x\in D^\times(\Q)$ is irreducible over $\R$
if and only if $x\notin \R\cap D^\times(\Q)=\Q^\times$.
For $\gamma\in G'(\Q)=Z(\Q)\backslash D^\times(\Q)$,
if $\gamma\neq 1$,
$X^2-\Tr_D(\gamma)X+N_D(\gamma)$ is irreducible over $\R$,
which implies that $\Tr_D(\gamma)^2<4N_D(\gamma)$.
Now $N_D(\gamma)=1$, we only need $\Tr_D(\gamma)=0$ or $1$.

\begin{Proposition}
[{\cite[Proposition 14.6.7]{voight2017quaternion}}]\label{F=Q}
Let $D$ be a quaternion algebra over $\Q$
and $E$ be a quadratic field extension of $\Q$.
Then $E$ can be embedded in $D$ if and only if 
every $v\in\Ram(D)$ does not split in $E$,
i.e. $E_v$ is a field for all $v\in\Ram(D)$.
In particular,
\begin{enumerate}[(1)]
\item there is an $x\in D$ with $\Tr_D(x)=0$, $N_D(x)=1$ 
(i.e. $\sqrt{-1}$ can be embedded in $D$) 
if and only if
$\disc(D)$ has no prime factor of the form $4n+1$.
\item there is an $x\in D$ with $\Tr_D(x)=1$, $N_D(x)=1$ 
(i.e. $\frac{1+\sqrt{-3}}2$ can be embedded in $D$) 
if and only if
$\disc(D)$ has no prime factor of the form $3n+1$.
\end{enumerate}
\end{Proposition}

Lemma \ref{nonzeroterms}, Proposition \ref{F=Q},
and the following lemma together imply Theorem \ref{orbdecomp}.

\begin{Lemma} \label{nontrivialorbits}
With notations in Theorem \ref{orbdecomp},
\[\frac{I_{[\gamma_0]}(f)}{\vol(K') }=
\frac 12\int_{\A_{E_0}^\times \backslash D^\times(\A)}\varphi_{\gamma_0}(h)\left(
\int_{\A^\times E_0^\times \backslash \A_{E_0}^\times}
(R(h)\phi^*)(t)\overline{(R(h)\phi^{**})(t)}\ dt
\right)\ dh,\]
\[\frac{I_{[\gamma_1]}(f)}{\vol(K') }=
\int_{\A_{E_1}^\times \backslash D^\times(\A)}\varphi_{\gamma_1}(h)\left(
\int_{\A^\times E_1^\times \backslash \A_{E_1}^\times}
(R(h)\phi^*)(t)\overline{(R(h)\phi^{**})(t)}\ dt
 \right)\ dh.\]
Here $E_0=\Q(\sqrt{-1})$,
$E_1=\Q(\sqrt{-3})$
are quadratic extensions of $\Q$
which can be embedded in $D$ when $\gamma_0,\gamma_1$ exist respectively.
\end{Lemma}

\begin{proof}
(1) When $N$ has no prime factor of the form $4n+1$,
we can write $D = (\frac{-1,-N}{\Q})$ by Lemma \ref{AnotherPresentation}
and take $\gamma_0=i_D$ (the $i$ in the $\Q$-basis $\{1,i,j,k\}$ of $D$).
Then
\[I_{[\gamma_0]}=I_{[i_D]}
=\vol(K')
\int_{G'_{i_D}(\A) \backslash G'(\A)}\varphi_{i_D}(h)\left(
\int_{G'_{i_D}(\Q \backslash \A)}(R(h)\phi^*)(t)\overline{(R(h)\phi^{**})(t)}\ dt
\right)\ dh,\]
where $\varphi_{i_D}$ is the characteristic function of the set
\[\{h\in G'_{i_D}(\A) \backslash G'(\A): h_{p}^{-1}i_D h_{p}\in K_p
\text{ for all primes }p\}.\]
Lemma \ref{centralizer} shows that $G'_{i_D}(\Q)$ is the image in $G'(\Q)$ of
$\{1,j_D\}\cdot \Q(i_D)^\times$.
Instead of using $G'_{i_D}$
we can simply consider a subgroup $T$ in $G'$ such that
$T(\Q)$ is the image in $G'(\Q)$ of $\Q(i_D)^\times$.
By taking $H=T(\A)$ and $K=G'_{i_D}(\Q)$ in Lemma \ref{productmeasure},
for $f\in C_c(G'_{i_D}(\Q) \backslash G_{i_D}'(\A))$,
\[
\int_{G'_{i_D}(\Q) \backslash G_{i_D}'(\A)}f(g)\ dg
=\int_{G'_{i_D}(\Q) \backslash G'_{i_D}(\Q)T(\A)}f(g)\ dg 
=\int_{(G'_{i_D}(\Q)\cap T(\A)) \backslash T(\A)}f(t)\ dt
=\int_{T(\Q) \backslash T(\A)}f(t)\ dt 
\]
and hence
\[\int_{G'_{i_D}(\Q \backslash \A)}(R(h)\phi^*)(t)\overline{(R(h)\phi^{**})(t)}\ dt
=\int_{T(\Q \backslash \A)}(R(h)\phi^*)(t)\overline{(R(h)\phi^{**})(t)}\ dt.\]
Therefore
\[\begin{split}
\frac{I_{[{i_D}]}}{\vol(K')}
=&\ \int_{G'_{i_D}(\A) \backslash G'(\A)}\varphi_{i_D}(h)\left(
\int_{T(\Q \backslash \A)}(R(h)\phi^*)(t)\overline{(R(h)\phi^{**})(t)}\ dt
 \right)\ dh\\
=&\ \frac 12\int_{T(\A) \backslash G'(\A)}\varphi_{i_D}(h)\left(
\int_{T(\Q \backslash \A)}(R(h)\phi^*)(t)\overline{(R(h)\phi^{**})(t)}\ dt
\right)\ dh.
\end{split}\]

(2) When $N$ has no prime factor of the form $3n+1$,
we can write $D = (\frac{-3,-N}{\Q})$ by Lemma \ref{AnotherPresentation}
and take $\gamma_1=\frac 12(1+i_D)$.
Then
\[I_{[\gamma_1]}=I_{[{1+i_D}]}
=\vol(K')
\int_{G'_{1+i_D}(\A) \backslash G'(\A)}\varphi_{1+i}(h)\left(
\int_{G'_{1+i_D}(\Q \backslash \A)}(R(h)\phi^*)(t)\overline{(R(h)\phi^{**})(t)}\ dt
\right)\ dh,\]
where $\varphi_{1+i_D}$ is the characteristic function of the set
\[\{h\in G'_{1+i_D}(\A) \backslash G'(\A): h_{p}^{-1}(1+i_D) h_{p}\in K_p
\text{ for all primes }p\}.\]
Lemma \ref{centralizer} shows that $G'_{1+i_D}(\Q)$ is the image in $G'(\Q)$ of
$\Q(i_D)^\times$.
Again we denote by $T$ the subgroup in $G'$ such that
$T(\Q)$ is the image in $G'(\Q)$ of $\Q(i_D)^\times$
(notice that this $i_D=\sqrt{-3}$ is different 
with the $i_D=\sqrt{-1}$ in the previous case). So
\[\frac{I_{[1+i_D]}}{\vol(K') }
=\int_{T(\A) \backslash G'(\A)}\varphi_{1+i_D}(h)\left(
\int_{T(\Q \backslash \A)}(R(h)\phi^*)(t)\overline{(R(h)\phi^{**})(t)}\ dt
\right)\ dh.\]
\end{proof}

\subsection{Nontrivial orbits and Waldspurger's formula}\label{section.harmonic}

For any two forms $\phi',\phi''\in\pi_3'$,
as functions in $L^2([D^\times])$,
we already have a Petersson inner product $\langle\cdot,\cdot\rangle$ defined as
\begin{equation}\label{innerproduct.D}
\langle\phi',\phi''\rangle=\int_{[D^\times]}
\phi'(g)\overline{\phi''(g)}\ dg.
\end{equation}
We can also consider them as functions in $L^2([E^\times])$,
in which we have an inner product
\[\langle\phi',\phi''\rangle_E:
=\int_{[E^\times]}\phi'(t)\overline{\phi''(t)}\ dt.\]
Here $E/\Q$ is a quadratic field extension as in the above lemma,
embedded in $D$ by $E=\Q(i_D)$.
We can regard the set of all characters
on $E^\times\backslash\A_E^\times$
(whose restrictions on $\A^\times$ are trivial)
as an orthogonal basis of $L^2([E^\times])$ and
decompose $\phi|_{[E^\times]}$ by
\[\phi|_{[E^\times]}=\sum_{\Omega\in\widehat{[E^\times]}}
\frac{\langle \phi|_{[E^\times]},\Omega\rangle_E}
{\langle \Omega,\Omega\rangle_E}\Omega
=\sum_{\Omega\in\widehat{[E^\times]}}
\frac{P_{\Omega^{-1}}(\phi)}
{\langle \Omega,\Omega\rangle_E}\Omega.\]
Here the period integral $P_\Omega:\pi'\to\CC$ 
is defined by \eqref{def.periodintegral}.
Then
\[\begin{split}
&\int_{[E^\times]}\phi'(t)\overline{\phi''(t)}\ dt
=\langle \phi'|_{[E^\times]},\phi''|_{[E^\times]}\rangle_E\\
=\ &\left\langle \sum_{\Omega\in\widehat{[E^\times]}}
\frac{P_{\Omega^{-1}}(\phi')}{\langle \Omega,\Omega\rangle_E}\Omega,
\sum_{\Omega\in\widehat{[E^\times]}}
\frac{P_{\Omega^{-1}}(\phi'')}{\langle \Omega,\Omega\rangle_E}\Omega
\right\rangle_E
=\sum_{\Omega\in\widehat{[E^\times]}}
\frac{P_{\Omega^{-1}}(\phi')}{\langle \Omega,\Omega\rangle_E}
\overline{\left(\frac{P_{\Omega^{-1}}(\phi'')}
{\langle \Omega,\Omega\rangle_E}\right)}
\langle \Omega,\Omega\rangle_E.
\end{split}\]
Noticing $\langle \Omega,\Omega\rangle_E=\vol([E^\times])$,
we have
\begin{equation}\label{harmonicanalysis}
\int_{[E^\times]}\phi'(t)\overline{\phi''(t)}\ dt
=\frac 1{\vol([E^\times])}\sum_{\Omega\in\widehat{[E^\times]}}
P_{\Omega^{-1}}(\phi')P_{\Omega}(\overline{\phi''}).
\end{equation}
The product of two period integrals is related via
a period formula of Waldspurger (Theorem \ref{Waldspurger})
to the central value of a base change $L$-function.

Now we go back to the nontrivial orbits in our RTF.
Lemma \ref{nontrivialorbits} shows that
\[\frac{I_{[\gamma]}(f)}{\vol(K') }=c_\gamma
\int_{\A_{E}^\times \backslash D^\times(\A)}\varphi_\gamma(h)\left(
\int_{\A^\times E^\times \backslash \A_{E}^\times}
(R(h)\phi^*)(t)\overline{(R(h)\phi^{**})(t)}\ dt
\right)\ dh.\]
Here $c_\gamma=\frac 12$ when $\gamma=\gamma_0$, and
$c_\gamma=1$ when $\gamma=\gamma_1$.
Apply \eqref{harmonicanalysis} and Theorem \ref{Waldspurger}
with $\phi'=R(h)\phi^*$, $\phi''=R(h)\phi^{**}$, and we have
\begin{multline*}
\frac{\int_{[E^\times]}
(R(h)\phi^*)(t)\overline{(R(h)\phi^{**})(t)}\ dt}
{\langle R(h)\phi^{*},R(h)\phi^{**}\rangle}
=\frac 1{\vol([E^\times])}
\frac{\zeta^*_{\Q}(2)}{2L(1,\pi_3,\Ad)}\\
\cdot\sum_{\Omega\in\widehat{[E^\times]}}
L(\frac 12,(\pi_3)_E\otimes \Omega)
\prod_v\alpha_v(\pi_{3,v}'(h_v)\phi^*_{v},
\overline{\pi_{3,v}'(h_v)\phi^{**}_{v}};\Omega_v).
\end{multline*}
Theorem \ref{explicitWaldspurger}
gives some conditions for $\Omega$ such that 
$\prod_{v}\alpha_v$ does not vanish.
But for $E=E_0=\Q(\sqrt{-1})$
or $E=E_1=\Q(\sqrt{-3})$ (which has class number 1),
the conditions can be more specific.

Lemma \ref{char.unram} shows that, if
$\Omega:E^\times\backslash\A_E^\times\to\CC^\times$
is unramified everywhere,
it is determined by its Archimedean factor, i.e.
a character of $\CC^\times$ which is invariant under $\cO_E^\times$.
Here \[\begin{array}{lll}
\cO_{E_0}^\times=\{\pm 1,\pm i\}\quad
&\text{for }E_0=\Q(\gamma_0)=\Q(\sqrt{-1}),\
&i=\sqrt{-1};\\
\cO_{E_1}^\times=\{\zeta_6^r\}_{r=0}^5,\quad
&\text{for }E_1=\Q(\gamma_1)=\Q(\sqrt{-3}),\
&\zeta_6=\frac{1}{2}(1+\sqrt{-3}).
\end{array}\]
Then the sum over $\Omega\in\widehat{[E^\times]}$
becomes that over $\Omega_\infty\in
\widehat{\cO_E^\times\backslash\CC^\times/\R^\times}$.
Recall that every character of $\CC^\times/\R^\times$
is of the form $z\mapsto \sgn^{2m}(z)=(z/\bar z)^m$.
Obviously 
\[
\begin{array}{c}
\sgn^{2m}\text{ is $\cO_{E_0}^\times$-invariant}
\Leftrightarrow m\in\Z,\ 2\mid m;\\
\sgn^{2m}\text{ is $\cO_{E_1}^\times$-invariant}
\Leftrightarrow m\in\Z,\ 3\mid m.\\
\end{array} 
\]
These are the $\Omega_\infty$ that may appear in the sum.

Moreover we can rewrite the last condition in Theorem \ref{explicitWaldspurger}:
\begin{itemize}
\item For $E=\Q(\sqrt{-1})$ and $p=2\mid N$, 
$\tau=1+\sqrt{-1}$ is a uniformizer of $E_p$.
For $\Omega_\infty=\sgn^{2m}$ with $2\mid m$,
the last condition in Theorem \ref{explicitWaldspurger} implies
\[\sgn^{2m}(1+i)=-\varepsilon_2(\frac 12,\pi_3)=\begin{cases}
1&\Rightarrow m\equiv 0\pmod 4,\\
-1&\Rightarrow m\equiv 2\pmod 4.
\end{cases}\]
\item For $E=\Q(\sqrt{-3})$ and $p=3\mid N$, 
$\tau=\sqrt{-3}$ is a uniformizer of $E_p$.
And similarly for $\Omega_\infty=\sgn^{2m}$ with $3\mid m$,
we have
\[\sgn^{2m}(\sqrt{-3})=-\varepsilon_3(\frac 12,\pi_3)=\begin{cases}
1&\Rightarrow m\equiv 0\pmod 6,\\
-1&\Rightarrow m\equiv 3\pmod 6.
\end{cases}\]
\end{itemize}

In conclusion we have the following restrictions on $\Omega$.
\begin{Lemma}\label{nonvanishOmega}
Let $\ell=2$ or $3$ be the ramified place of $E=E_0$ or $E_1$.
With notations in Theorem \ref{orbdecomp} and Theorem \ref{Waldspurger},
\[\prod_v\alpha_v(\pi_v'(h_v)\phi_{v}^*,
\overline{\pi_v'(h_v)\phi_{v}^{**}};\Omega_v)=0\] unless
$\Omega_v$ is unramified for all finite $v$ and
$\Omega_\infty(z)=\sgn^{2m}(z)=(z/\bar z)^m$ where
\[m\in \ell\Z,\quad |m|\leq k-1,\quad\text{and} \quad
(-1)^{\ord_2(m/\ell)}=-\varepsilon_\ell(\frac 12,\pi_3)\quad
\text{if }\ell\mid N.\]
In particular, when the weight is $2k=2$ or $4$,
$m$ can only be $0$, i.e. $\Omega$ can only be a trivial character. 
\end{Lemma}

Now we can write $I_{[\gamma]}(f)$ as
\begin{multline*}
\frac{I_{[\gamma]}(f)}{\vol(K') }=
\langle \phi^{*},\phi^{**}\rangle\frac {c_\gamma}{\vol([E^\times])}
\frac{\zeta^*_{\Q}(2)}{2L(1,\pi_3,\Ad)}\\
\cdot \sum_{\Omega}
L(\frac 12,(\pi_3)_E\otimes \Omega)
\prod_v\int_{E_v^\times \backslash D_v^\times}
\alpha_v(\pi_{3,v}'(h_v)\phi^*_{v},
\overline{\pi_{3,v}'(h_v)\phi^{**}_{v}};\Omega_v)
\varphi_{\gamma,v}(h_v)\ dh_v,
\end{multline*}
where the sum is over all $\Omega$ satisfying 
the conditions in Lemma \ref{nonvanishOmega}.
Denote $\phi'=R(h)\phi^*$, $\phi''=R(h)\phi^{**}$.
Notice that $B_\infty(\phi'_\infty,
\overline{\phi''_\infty})
=\langle\phi^*,\phi^{**}\rangle$.
We can write 
\begin{multline}\label{nontriv.orbit}
\frac{I_{[\gamma]}(f)}{\vol(K') }=
\frac {c_\gamma\zeta^*_{\Q}(2)}{2L(1,\pi_3,\Ad)\vol([E^\times])}\\
\cdot \sum_{\Omega}
L(\frac 12,(\pi_3)_E\otimes \Omega)
\int_{\CC^\times \backslash D_\infty^\times}
B_\infty(\phi'_\infty,\overline{\phi''_\infty})
\alpha_\infty\ dh_\infty
\prod_p
\int_{E_p^\times \backslash D_p^\times}\varphi_{\gamma,p}(h_p)
\alpha_p\ dh_p.
\end{multline}

We will show in the next sections 
the calculations of the above local integrals.
Lemma \ref{geom.archi.coeff} together with \eqref{architerm} shows
\[\int_{\CC^\times \backslash D_\infty^\times}
B_\infty(\phi'_\infty,\overline{\phi''_\infty})
\alpha_\infty\ dh_\infty=
\frac{\vol(G'_\infty)\|\phi_3\|^2}
{L_\infty(\frac 12,(\pi_3)_E\otimes \Omega)}
\frac{2^2\Gamma(2k-1)^4}{(2\pi)^{2k+1}\Gamma(k)^3\Gamma(3k-1)}
\mathbb{I}_{2k}^{(m)}(\gamma)\]
for some constant $\mathbb{I}_{2k}^{(m)}(\gamma)$
depending only on $\gamma$, $k$ and $m$.
For the integrals at finite places,
Lemmas \ref{nonarchiunram2}, \ref{nonarchiunram1}
and \eqref{nonarchramify3} imply that
\[\int_{E_p^\times \backslash D_p^\times}\varphi_{\gamma,p}
\alpha_p\ dh_p=\vol(K_p)\cdot\begin{cases}
1,& p\nmid\disc(D);\\
2(1-p^{-1}),& p\mid\disc(D)
\end{cases}\]
when $\alpha_p$ does not vanish.
Recall that $\zeta^*_{\Q}(2)=\pi/6$, 
$\vol([E^\times])=2L(1,\eta)$.
With the above results, we can rewrite \eqref{nontriv.orbit} as
\[
\frac{I_{[\gamma]}(f)}{\vol(K') }=
\frac {c_\gamma\|\phi_3\|^2\vol(K')\pi/6}{L(1,\pi_3,\Ad)L(1,\eta)} 
2^{\omega(N)}\frac{\varphi(N)}N
\frac{\Gamma(2k-1)^4}{(2\pi)^{2k+1}\Gamma(k)^3\Gamma(3k-1)}
\cdot \sum_{\Omega}
L_{\fin}(\frac 12,(\pi_3)_E\otimes \Omega)
\mathbb{I}_{2k}^{(m)}(\gamma).
\]
Here we have 
\[\vol(K')=\frac {24}{\varphi(N)},\quad
c_\gamma=\begin{cases}\frac 12&\\1&\end{cases},\quad
\eta\text{ corresponds to }
\begin{cases}\chi_{-4},&\gamma=\gamma_0;\\
\chi_{-3},&\gamma=\gamma_1.\end{cases}\]
Therefore
\[
\frac{I_{[\gamma]}(f)}{\vol(K') }=\frac{2c_\gamma}{L(1,\eta)}
\frac {\|\phi_3\|^2 }{L(1,\pi_3,\Ad)} 
\frac{2^{\omega(N)}}N
\frac{\Gamma(2k-1)^4}{(2\pi)^{2k}\Gamma(k)^3\Gamma(3k-1)}
\sum_{\Omega}
L_{\fin}(\frac 12,(\pi_3)_E\otimes \Omega)
\mathbb{I}_{2k}^{(m)}(\gamma);
\]
in particular we have $L(1,\chi_{-4})=1/4$, 
$L(1,\chi_{-3})=1/3\sqrt 3$
by the Dirichlet class number formula, and then
\begin{equation}\label{nontrivialorbiteqns}
\begin{split}\frac{I_{[\gamma_0]}(f)}{\vol(K') }&=
\frac {4\|\phi_3\|^2 }{L(1,\pi_3,\Ad)} 
\frac{2^{\omega(N)}}N
\frac{\Gamma(2k-1)^4}{(2\pi)^{2k}\Gamma(k)^3\Gamma(3k-1)}
\sum_{\Omega}
L_{\fin}(\frac 12,(\pi_3)_{E_0}\otimes \Omega)
\mathbb{I}_{2k}^{(m)}(\gamma_0),\\
\frac{I_{[\gamma_1]}(f)}{\vol(K') }&=
\frac {6\sqrt{3}\|\phi_3\|^2 }{L(1,\pi_3,\Ad)} 
\frac{2^{\omega(N)}}N
\frac{\Gamma(2k-1)^4}{(2\pi)^{2k}\Gamma(k)^3\Gamma(3k-1)}
\sum_{\Omega}
L_{\fin}(\frac 12,(\pi_3)_{E_1}\otimes \Omega)
\mathbb{I}_{2k}^{(m)}(\gamma_1).
\end{split}
\end{equation}

With Theorem \ref{orbdecomp} and \eqref{nontrivialorbiteqns},
we can get Theorem \ref{maingeometric}, 
the orbital decomposition of $I(f)$.

\subsection{Local calculation 
on ramified quaternion algebras}\label{section.nonarchi.ram}

In the following two sections we will explicitly
calculate the local integrals in \eqref{nontriv.orbit}
at non-Archimedean places.
The local constants $\alpha_p$ have been studied in
Lemmas \ref{nonarchiunram2} and \ref{nonarchramify1}.

\begin{Lemma}\label{nonarchramify2}
For $p\mid \disc(D)$,
\[\int_{E_{p}^\times \backslash D_p^\times}\varphi_{\gamma,p}(h_p) \ dh_p
=\int_{E_{p}^\times \backslash D_p^\times}\ dh_p
=\vol(E_{p}^\times \backslash D_p^\times).\]
\end{Lemma}
\begin{proof}
In this case $D_p$
is a division algebra.
Recall that
$\cO_p^\times=\{x\in D_p^\times: v_p(N_{D_p}(x))=0\}$
by Proposition \ref{maxorder}.
Clearly for any $h_p\in D_p$,
$N_{D_p}(h_{p}^{-1}\gamma h_{p})=N_{D_p}(\gamma)=1$
for $\gamma=\gamma_0$ or $\gamma_1$.
So the condition $h_{p}^{-1}\gamma h_{p}\in \cO_p^\times$ is trivial,
i.e. $\varphi_{\gamma,p}(h_p)\equiv 1$.
\end{proof}

Together with Lemma \ref{nonarchramify1} we know, 
when $\alpha_p\neq 0$,
\[\int_{E_{p}^\times \backslash D_p^\times}
\varphi_{\gamma,p}(h_p) \alpha_p\ dh_p
=(1-p^{-1})\vol(\Q_{p}^\times \backslash E_p^\times)
\vol(E_{p}^\times \backslash D_p^\times)
=(1-p^{-1})\vol(G'_p).\]
Notice that $N_{D_p}(\Q_{p}^\times)=(\Q_p^\times)^2$
and $N_{D_p}(D_{p}^\times)=\Q_p^\times$ 
(cf. \cite[Lemma 13.4.9]{voight2017quaternion}).
Since in this case $\cO_p^\times=\{g_p\in D_p^\times:
v_p(N_{D_p}(g_p))=0\}$,
we can write
\[\begin{split}G_p'&=\Q_{p}^\times \backslash D_p^\times
=\Q_{p}^\times \backslash \Big(\{g_p: 2\mid v_p(N_{D_p}(g_p))\}
\sqcup\{g_p: 2\nmid v_p(N_{D_p}(g_p))\}\Big)\\
&=\Q_{p}^\times \backslash \Q_p^\times \cO_p^\times
\sqcup\Q_{p}^\times \backslash j\Q_p^\times \cO_p^\times
=K_p\sqcup jK_p.\end{split}\]
(Here $v_p(N_{D_p}(j))=v_p(N)=1$ for $j\in D_p=(\frac{-q,-N}{\Q_p})$.
Recall that $D_p=(\frac{-q,-N}{\Q_p})$
for $q=1$ or $3$, by Lemma \ref{AnotherPresentation}.)
And therefore
$\vol(G_p')
=2\vol(K_p)$.

In conclusion, for $p\mid N$ and $\chi_{-d}(p)\neq 1$,
\begin{equation}\label{nonarchramify3}
\alpha_p\int_{E_{p}^\times \backslash D_p^\times}
\varphi_{\gamma,p}(h_p) \ dh_p
=\begin{cases}
2(1-p^{-1})\vol(K_p),&\Omega_p=\overline{\delta_p\circ N_{D_p}};\\
0,&\text{otherwise}.
\end{cases}\end{equation}
In particular if $\Omega_p=\1$ then
\begin{equation}\label{nonarchramify4}
\alpha_p\int_{E_{p}^\times \backslash D_p^\times}
\varphi_{\gamma,p}(h_p) \ dh_p\\
=(1-\chi_{-d}(p))(1-p^{-1})\vol(K_p)\cdot\begin{cases}
1,&\chi_{-d}(p)=-1;\\
(1-\varepsilon(\frac 12,(\pi_3)_p)),&\chi_{-d}(p)=0.
\end{cases}\end{equation}

\subsection{Local calculation 
on split quaternion algebras}\label{section.nonarchi.unram}

When $p\nmid\disc(D)$,
we have $D_p\cong M(2,\Q_p)$ and
$\pi'_{3,p}\cong\pi_{3,p}$ are both unramified.
Take the maximal order $\cO_p$ to be the preimage of $M(2,\Z_p)$.
Recall that in Section \ref{section.waldspurger} 
we fix an isomorphism $D_p\cong M(2,\Q_p)$
so that $\cO_{E_p}= E_p\cap \cO_p$ for $E=\Q(\gamma)$
(as defined in Theorem \ref{orbdecomp}).

Here $h_p$ satisfies
$h_p^{-1}\gamma h_p\in K_p$ for all $p$.
(Our choice of $\gamma$ has $N_{D}(\gamma)=1$, so
$h_p^{-1}\gamma h_p\in \GL(2,\Z_p)$.)
Let $\phi_p$ be the normalized spherical vector 
in $\pi_{3,p}$.
For $\phi'_p=\phi''_p=\pi_p(h_p)\phi_p$
we care about
\begin{itemize}
	\item $\int_{E_p^\times/\Q_p^\times}
	B_p(\pi_p(t_p)\phi'_p,\overline{\phi''_p})
	\Omega_p(t_p)\ dt_p$
	for $h_p^{-1}\gamma h_p\in K_p$; and
	\item $\int_{E_p^\times\backslash D_p^\times }
	\1_{K_p}(h_p^{-1}\gamma h_p) \ dh_p$.
\end{itemize}
For the second integral we have the following lemma.
Notice that with this lemma the first integral has been studied 
in Lemma \ref{nonarchiunram2}.

\begin{Lemma}\label{nonarchiunram1}
Assume that $p\nmid\disc(D)$.
Let $\gamma\in E$ be such that 
$\gamma=\gamma_0=\sqrt{-1}$ or $\gamma=\gamma_1=\frac{1+\sqrt{-3}}2$.
\begin{enumerate}[(1)]
\item For $h\in D_p^\times$,
\[\text{(i)} \quad h\in E_p^\times\GL(2,\Z_p)\quad
\text{if and only if}\quad \text{(ii)}\quad h^{-1}\gamma h\in\GL(2,\Z_p).\]
\item
\[\int_{E_p^\times\backslash D_p^\times }
	\1_{K_p}(h_p^{-1}\gamma h_p) \ dh_p
=\frac{\vol(\cO_p^\times)}{\vol(\cO_{E_p}^\times)}.\]
\end{enumerate}
\end{Lemma}

\begin{proof}
The second statement can be shown 
by applying $H=\cO_p^\times$ and $K=E_p^\times$ in Lemma \ref{productmeasure}:
\[\int_{E_p^\times\backslash D_p^\times }
	\1_{K_p}(h_p^{-1}\gamma h_p) \ dh_p
=\int_{E_p^\times\backslash E_p^\times \cO_p^\times}\ dh_p
=\vol((\cO_p^\times\cap E_p^\times)\backslash\cO_p^\times)
=\frac{\vol(\cO_p^\times)}{\vol(\cO_{E_p}^\times)}.\]
For the first statement,
(i)$\Rightarrow$(ii) is obvious
since $\gamma\in\GL(2,\Z_p)$ and $\gamma$ commutes with $E_p^\times$,
and (ii)$\Rightarrow$(i) will be shown
case by case as follows.
\end{proof}

\subsubsection{}

When $E_p/\Q_p$ is split, 
$E_p=\Q_p(\gamma)\cong\Q_p\oplus\Q_p$ 
is embedded in $M(2,\Q_p)$
as the set of all the diagonal matrices.
In particular, $\gamma=\gamma_0$ (or $\gamma_1$)
can be embedded in $\GL(2,\Q_p)$
as $(\begin{smallmatrix}\gamma&\\&\bar \gamma\end{smallmatrix})$.

The Iwasawa decomposition gives that
\[h=c_h\left(\begin{array}{cc}p^r& \\ &1\end{array}\right)
\left(\begin{array}{cc}1&x_h\\ &1\end{array}\right)k_h,\quad
c_h\in\Q_p^\times,\ r\in\Z,\ x_h\in\Q_p,\ k_h\in\GL(2,\Z_p).\]
One can simplify that
\[h^{-1}\gamma h=
k_h^{-1}\left(\begin{array}{cc}\gamma&\left(\gamma-\bar\gamma\right)x_h\\
&\bar \gamma\end{array}\right)k_h.\]
Therefore $h^{-1}\gamma h\in\GL(2,\Z_p)$
if and only if $\left(\gamma-\bar \gamma\right)x_h\in\Z_p$.
This implies $v_p(x_h)\geq 0$ when $p\neq 2$ (or $p\neq 3$ respectively).
Then $h\in c_h(\begin{smallmatrix}p^r& \\ &1\end{smallmatrix})\GL(2,\Z_p)$.
This proves Lemma \ref{nonarchiunram1} when $E_p/\Q_p$ is split.

\subsubsection{}

When $E_p/\Q_p$ is non-split and unramified, 
it is well-known that 
$\cO_{E_p}=\Z_p[\gamma]$ holds for both $E=E_0$ and $E_1$.
Then 
Lemma \ref{nonarchiunram1} can be shown by 
taking $\tau=\gamma$ in the following lemma. 

\begin{Lemma}
Suppose that $E_p/\Q_p$ is a quadratic field extension which is unramified,
and that $\cO_{E_p}=\Z_p[\tau]$ for some $\tau\in\cO_{E_p}^\times$.
Assume that either $p$ is an odd prime, or $p=2$ but $\Tr_{E_p/\Q_p}(\tau)\notin 2\Z_2$.
Fix an embedding $E_p\hookrightarrow M(2,\Q_p)$ by
\[a+b\tau\longmapsto
\left(\begin{array}{cc}a+b\Tr_{E_p/\Q_p}(\tau)&bN_{E_p/\Q_p}(\tau)\\
-b&a\end{array}\right)\]
(which is the same as in Section \ref{section.waldspurger}).
Then we have, 
for $h\in \GL(2,\Q_p)$,
\[h\in E_p^\times\GL(2,\Z_p)\quad
\text{if and only if} \quad h^{-1}\tau h\in\GL(2,\Z_p).\]
(In fact it is equivalent to say $h\in \Q_p^\times\GL(2,\Z_p)$.)
\end{Lemma}

\begin{proof}
The necessity is easy to check.
To show the sufficiency,
write $h=\left(\begin{smallmatrix}y&x\\ &1\end{smallmatrix}\right)k_h$
for some $k_h\in Z_p\GL(2,\Z_p)$
using Iwasawa decomposition.
Then $h^{-1}\tau h\in\GL(2,\Z_p)$ implies
\begin{multline*}
\left(\begin{array}{cc}y&x\\ &1\end{array}\right)^{-1}
\left(\begin{array}{cc}\Tr_{E_p/\Q_p}(\tau)&N_{E_p/\Q_p}(\tau)\\-1&0\end{array}\right)
\left(\begin{array}{cc}y&x\\ &1\end{array}\right)\\
=\left(\begin{array}{cc}\Tr_{E_p/\Q_p}(\tau)+x&
(x^2+x\Tr_{E_p/\Q_p}(\tau)+N_{E_p/\Q_p}(\tau))y^{-1}\\
-y&-x\end{array}\right)
\in\GL(2,\Z_p),\end{multline*}
and hence 
\begin{equation}\label{conjugatecondition}
x,y\in\Z_p,\quad
v_p(x^2+x\Tr_{E_p/\Q_p}(\tau)+N_{E_p/\Q_p}(\tau))\geq v_p(y).
\end{equation}
But we notice that $x^2+x\Tr_{E_p/\Q_p}(\tau)+N_{E_p/\Q_p}(\tau)=0$
has two solutions $x=-\tau$ and $-\bar\tau$ 
(here $\bar\cdot$ is the nontrivial element in $\operatorname{Gal}(E_p/\Q_p)$) 
in $E_p=\Q_p(\tau)$
and therefore has no solution in $\Q_p$.
Using Hensel's Lemma 
(this is where the assumption that $\Tr_{E_p/\Q_p}(\tau)\notin 2\Z_2$ for $p=2$ works),
$x^2+x\Tr_{E_p/\Q_p}(\tau)+N_{E_p/\Q_p}(\tau)\in\Z_p^\times$ for any $x\in\Z_p$.
Then \eqref{conjugatecondition} implies $y\in\Z_p^\times$, 
that is $h\in Z_p\GL(2,\Z_p)\subseteq E_p^\times\GL(2,\Z_p)$.
\end{proof}

\subsubsection{}

When $p$ ramifies in $E$, i.e. 
when $p=2$ for $\gamma=\gamma_0$
or $p=3$ for $\gamma=\gamma_1$,
the necessity can be seen by direct calculation.
To show the sufficiency,
we again use Iwasawa decomposition and
write $h=\left(\begin{smallmatrix}y&x\\ &1\end{smallmatrix}\right)k_h$
for some $k_h\in Z_p\GL(2,\Z_p)$.

For $\gamma=\gamma_0$, $E_0=\Q(\sqrt{-1})$ and $p=2$,
we choose $\tau_2=\varpi_2=1+\sqrt{-1}$
to be the uniformizer in $E_p$ so that 
$\cO_{E_p}=\Z_p[\tau_p]$.
By our choice of embedding $E_p\hookrightarrow M(2,\Q_p)$
defined in Section \ref{section.waldspurger},
\[\varpi_2\mapsto
\left(\begin{array}{cc}2&2\\
-1&0\end{array}\right),
\quad \gamma_0=-1+\varpi_2\mapsto
\left(\begin{array}{cc}1&2\\
-1&-1\end{array}\right).\]
Then 
\[h^{-1}\gamma_0 h=k_h^{-1}\left(\begin{array}{cc}1+x&((x+1)^2+1)y^{-1}\\
-y&-x-1\end{array}\right)k_h
\in\GL(2,\Z_2)\]
implies that $x,y\in \Z_2$ and
$v_2((x+1)^2+1)\geq v_2(y)$.

Notice that 
\[v_2((x+1)^2+1)=0\Leftrightarrow x\in\Z_2^\times 
\quad\text{and}\quad 
v_2((x+1)^2+1)=1\Leftrightarrow x\in 2\Z_2.\]
If $x\in\Z_2$ and $y\in\Z_2^\times$ 
we directly get  
$h\in Z_p\GL(2,\Z_p)\subseteq E_p^\times\GL(2,\Z_p)$.
The only case that is not covered above is 
when $v_2(y)=1$ (and then $v_2(x)\geq 1$).
In this case 
\[h=\left(\begin{array}{cc}2&2\\
-1&0\end{array}\right)
\left(\begin{array}{cc}0&-1\\
\frac y2&\frac x2+1\end{array}\right)k_h
\in \varpi_2 Z_p\GL(2,\Z_p)
\subseteq E_p^\times\GL(2,\Z_p).\]
These give a proof of Lemma \ref{nonarchiunram1} 
when $E=E_0$ and $p=2$.

For $\gamma=\gamma_1$, $E_1=\Q(\sqrt{-3})$ and $p=3$,
we choose $\tau_3=\varpi_3=\sqrt{-3}$ 
(one can check that $\cO_{\Q(\sqrt{-3})_3}=\Z_3[\sqrt{-3}]$).
The proof is similar, noticing that
\[\varpi_3\mapsto
\left(\begin{array}{cc}0&3\\-1&0\end{array}\right),
\quad \gamma_1=\frac {1+\varpi_3}2\mapsto
\frac 12\left(\begin{array}{cc}1&3\\-1&1\end{array}\right)\]
under our embedding $E_p\hookrightarrow M(2,\Q_p)$.

\subsection{Compatibility of two nontrivial orbits}
\label{section.compatibility}

When $N$ has no prime factor $\equiv 1\pmod 4$ and
none $\equiv 1\pmod 3$
(i.e. all prime factors of $N$ are either $2$, $3$,
or $\equiv 11\pmod {12}$),
$\gamma_0$ and $\gamma_1$
both appear in the quaternion algebra $D$,
i.e. two quadratic fields $E_0$ and $E_1$
both can be embedded in $D$.
Locally, for $p\nmid \disc(D)$,
the embeddings are given in 
Section \ref{section.Waldspurger.split}.
But different embeddings in this case 
would lead to different isomorphisms $D_p\cong M(2,\Q_p)$,
and moreover lead to local maximal orders 
that differ by conjugation by an element in $\GL(2,\Q_p)$.
But the spherical vector $\phi_{3,p}$
in the definition of the distribution $I(f)$
depends on the choice of maximal order of $D_p$.

Recall that we choose $\phi_3$ in Lemma \ref{autoset} so that
\[\phi_3\in \CC X_{3}^{2k-2} \otimes(\pi_{3,\fin}')^{K_{\fin}}\]
is a new-line vector in $\pi_3'$.
It depends on the definition of $X_{3}^{2k-2}$
(i.e. the way we construct $\pi_{3,\infty}'$,
which will be fixed in Section \ref{section.geo.archi})
and of $K_{\fin}=\prod_{p<\infty} K_p$
(i.e. the choice of maximal orders $\cO_p$ at every finite place).
For $p\mid \disc(D)$ there is
a unique maximal order $\cO_p$ of $D_p$
and $\phi_{3,p}$ is a constant multiple of
$\delta_p\circ N_{D_p}$ for both cases.

For each $p\nmid \disc(D)$, 
let $\cO_p$ be the maximal order given by the preimage 
of $M(2,\Z_p)$ under the fixed isomorphism $D_p\cong M(2,\Q_p)$.
We choose $\phi_{p}$ to be the normalized $\cO_p^\times$-invariant vector in
the spherical representation $\pi_{p}$.
Lemma \ref{nonarchiunram1} and \ref{nonarchiunram2} show that
\[\int_{E_p^\times \backslash D_p^\times}
\alpha_p\varphi_{\gamma,p}(h)\ dh=
\begin{cases}\vol(K_p),
&\Omega_p\text{ unramified,}\\0,&\Omega_p\text{ ramified}
\end{cases}\]
holds for 
\[\alpha_p=\alpha_p(\pi_p(h)\phi_{p},
\overline{\pi_p(h)\phi_{p}};\Omega_p),\quad 
\varphi_{\gamma,p}(h)=\1_{K_p}(h^{-1}\gamma h)\quad
\text{with }K_p=Z(\Q_p)\cO_p^\times.\]

Clearly the local constant $\alpha_p$ coming from Waldspurger's formula
depends on the choice of maximal order
(essentially on the embedding $E_p\hookrightarrow M(2,\Q_p)$).
But in this section we show that the above integral is independent 
of the choice of the maximal order $\cO_p$.
That is to say:

\begin{Lemma}\label{lemma.compatibility}
When $p\nmid \disc(D)$, 
let $\cO_p'$ be another maximal order of $D_p\cong M(2,\Q_p)$,
and $\phi_{p}'$ be the normalized $(\cO_p')^\times$-invariant vector in $\pi_{p}$.
Then 
\[\int_{E_p^\times \backslash D_p^\times}
\alpha_p'
\varphi_{\gamma,p}'(h)\ dh
=\int_{E_p^\times \backslash D_p^\times}
\alpha_p\varphi_{\gamma,p}(h)\ dh\]
with \[\alpha_p'=\alpha_p(\pi_p(h)\phi_{p}',
\overline{\pi_p(h)\phi_{p}'};\Omega_p),\quad
\varphi_{\gamma,p}'(h)=\1_{K_p'}(h^{-1}\gamma h),\]
$K_p'=Z(\Q_p)(\cO_p')^\times$, and $\alpha_p$, $\varphi_{\gamma,p}$
defined as above.
\end{Lemma}

\begin{proof}
By Proposition \ref{maxorder}
there exists $T\in\GL(2,\Q_p)$ such that 
$\cO_p'=T\cO_p T^{-1}$.

On one hand, Lemma \ref{nonarchiunram1} shows that
$h\in E_p^\times \cO_p^\times$
if and only if $h^{-1}\gamma h\in \cO_p^\times$; and
\[\int_{E_p^\times\backslash D_p^\times }
	\1_{K_p}(h_p^{-1}\gamma h_p) \ dh_p
=\frac{\vol(\cO_p^\times)}{\vol(\cO_{E_p}^\times)}.\]
We work on another maximal order $\cO_p'=T\cO_pT^{-1}$:
\[h^{-1}\gamma h\in(\cO_p')^\times\Leftrightarrow
T^{-1}h^{-1}\gamma hT\in\cO_p^\times\Leftrightarrow
hT\in E_p^\times\cO_p^\times\Leftrightarrow
h\in E_p^\times\cO_p^\times T^{-1};\]
moreover
\[\int_{E_p^\times\backslash D_p^\times }
	\1_{K_p'}(h^{-1}\gamma h) \ dh
=\int_{E_p^\times\backslash D_p^\times }
\1_{E_p^\times\cO_p^\times T^{-1}}(h_p) \ dh_p
=\frac{\vol(\cO_p^\times T^{-1})}{\vol(\cO_{E_p}^\times)}
=\frac{\vol(\cO_p^\times)}{\vol(\cO_{E_p}^\times)}.\]

On the other hand,
one can check that
$\pi_p(T^{-1})\phi_p'$ is $\cO_p^\times$-invariant.
By Lemma \ref{nonarchiunram2},
\[\alpha_p'=
\alpha_p(\pi_p(hT)(\pi_p(T^{-1})\phi_{p}'),
\overline{\pi_p(hT)(\pi_p(T^{-1})\phi_{p}')};
\Omega_p)=\begin{cases}
\frac{\vol(\cO_{E_p}^\times)}{\vol(\Z_p^\times)},
&\Omega_p\text{ unramified,}\\0,&\Omega_p\text{ ramified,}
\end{cases}\]
whenever $hT\in E_p^\times\cO_p^\times$.
This completes the proof.
\end{proof}

\subsection{Local calculation: Archimedean}\label{section.geo.archi}

Recall that when $v=\infty$,
in Section \ref{RepresentationSU2} we realize
$\pi_v'\cong \pi'_{2k}$ on the space
of homogeneous polynomials in $X,Y$ of degree $2k-2$ 
with
\[\pi'_{2k} (g)P(X,Y)=P((X,Y)g)\det(g)^{1-k}
\quad \text{for}\quad g\in
\left\{\left(\begin{matrix}\alpha&-\beta\\\bar\beta&\bar\alpha\end{matrix}\right)\in\GL(2,\CC)\right\}
\cong D^\times(\R).\]
This representation is determined by the last isomorphism.
To make the action on $E_0^\times$ and $E_1^\times$ consistent
(so that the two nontrivial orbits are compatible
as we mentioned in the previous section),
we fix the following isomorphism for $q=1$ and $3$:
\begin{equation}\label{archi.embed}\begin{gathered}
D(\R)=\left(\frac{-q,-N}{\R}\right)\iso
\left\{\left(\begin{matrix}\alpha&-\beta\\
\bar\beta&\bar\alpha\end{matrix}\right)\in M(2,\CC)\right\},\\
i_D\mapsto\left(\begin{matrix}\sqrt{-q}&\\&-\sqrt{-q}\end{matrix}\right),\quad
j_D\mapsto\left(\begin{matrix}&-\sqrt{N}\\\sqrt{N}&\end{matrix}\right).
\end{gathered}\end{equation}
Then the image of $E_\infty=\R(\gamma)$ are the same
if $\left(\frac{-1,-N}{\Q}\right)$ and $\left(\frac{-3,-N}{\Q}\right)$
express the same quaternion algebra.

The last lemma gives a formula to calculate 
the integral in \eqref{nontriv.orbit} at infinity.
In particular, it is a number 
depending only on $\gamma$, $k$ and $m$. 

\begin{Lemma}\label{geom.archi.coeff}
For $\phi'_\infty=\|\phi_3\|\pi_{2k}'(h)X^{2k-2}$, 
$\phi''_\infty=\dfrac{\|\phi_3\|}{\|\mathbb{P}_{2k}\|^2}
\pi_{2k}'(h)e_{\gamma}\in\pi_{2k}'$, $|\gamma|=1$,
define $\mathbb{I}_{2k}^{(m)}(\gamma)$ by 
\[
\int_{\CC^\times\backslash D^\times_\infty}\int_{\CC^\times/\R^\times}
B_\infty(\pi_{2k}'(t)\phi'_\infty,\overline{\phi''_\infty})\sgn^{2m}(t)\ dt\ dh
=\vol(G'_\infty)\frac{\|\phi_3\|^2}{2k-1}
\frac{\Gamma(2k-1)^3}{\Gamma(k)^3\Gamma(3k-1)}\cdot \mathbb{I}_{2k}^{(m)}(\gamma).
\]
Then we have, \begin{enumerate}[(i)]
\item for $0\leq r\leq 2k-2$,
\[\mathbb{I}_{2k}^{(r-k+1)}(\gamma)=\binom{2k-2}r^{-1}
\sum_{\substack{0\leq i,j\leq 2k-2\\i+j=3(k-1)-r}}\gamma^{2(k-1-i)}
\binom{2k-2}i^{-1}\binom{2k-2}j^{-1}|C_{i,j,r}|^2,\]
where $C_{i,j,r}$ is the coefficient in $\mathbb{P}_{2k}$ of
$X_1^{2k-2-i}Y_1^i\otimes X_2^{2k-2-j}Y_2^j
\otimes X_3^{2k-2-r}Y_3^{r}$.
In particular
\[\mathbb{I}_{2k}^{(0)}(\gamma)=
\frac{\Gamma(k)^2}{\Gamma(2k-1)}
\sum_{i=0}^{2k-2}\gamma^{2(k-1-i)}
\binom{2k-2}i^{-2}|C_{i,2k-2-i,k-1}|^2.\]
\item $\mathbb{I}_{2k}^{(m)}(\gamma)$ does not depend on $N$.
Moreover 
\[\sum_{m=-(k-1)}^{k-1}|\mathbb{I}_{2k}^{(m)}(\gamma)|
\leq\|\mathbb{P}_{2k}\|^2.\]
\end{enumerate}
\end{Lemma}

\begin{proof}
(i) By the proof of Lemma \ref{geom.archi.vanish} we have
\begin{equation}\label{1}
\int_{\CC^\times/\R^\times}
\langle\pi_{2k}'(t)\phi'_v,\phi''_v\rangle_{2k}
\sgn^{-2(k-1-r)}(t)\ dt
=c_{r}'\overline{c_{r}''}
\| X^{2k-2-r}Y^{r}\|^2
\vol(\CC^\times/\R^\times)\end{equation}
with
\begin{equation}\label{2}
c_{r}'\overline{c_{r}''}
=\frac{\langle \phi'_\infty,X^{2k-2-r}Y^{r}\rangle_{2k}
\overline{\langle \phi''_\infty,X^{2k-2-r}Y^{r}\rangle_{2k}}}
{\| X^{2k-2-r}Y^{r}\|^4}.\end{equation}

We first deal with $\langle \phi''_\infty,X^{2k-2-r}Y^{r}\rangle$.
With the proposition of $e_{\gamma}$ shown in Lemma \ref{def.edelta}, 
we write $\phi''_\infty$ back as an integral:
\[\phi''_\infty=\frac {\|\phi_3\|}{\vol(G'_\infty)\|w^\circ_{2k}\|^2}
\int_{G'_\infty}
\langle \pi'_{2k}\otimes\pi'_{2k}
(h^{-1}\gamma g,h^{-1}g) w^\circ_{2k}, w^\circ_{2k}\rangle
\pi'_{2k}(g)X^{2k-2}\ dg.\]
Then we have
\begin{multline}\label{3}
\langle \phi''_\infty,X^{2k-2-r}Y^{r}\rangle=
\frac {\|\phi_3\|}{\vol(G'_\infty)\|w^\circ_{2k}\|^2}\\
\cdot\int_{G'_\infty}
\langle \pi'_{2k}\otimes\pi'_{2k} (\gamma g,g) w^\circ_{2k},
\pi'_{2k}\otimes\pi'_{2k}(h,h) w^\circ_{2k}\rangle
\langle\pi'_{2k}(g)X^{2k-2}
,X^{2k-2-r}Y^{r}\rangle\ dg.
\end{multline}
Noticing that $\{X_1^{2k-2-i}Y_1^i\otimes X_2^{2k-2-j}Y_2^j\}$
forms an orthogonal basis of $\pi'_{2k}\otimes\pi'_{2k}$,
the first matrix coefficient
$\langle \pi'_{2k}\otimes\pi'_{2k} (\gamma g,g) w^\circ_{2k},
\pi'_{2k}\otimes\pi'_{2k}(h,h) w^\circ_{2k}\rangle$
is equal to
\[\begin{split}
&\ \sum_{0\leq i,j\leq 2k-2}
\langle\pi'_{2k}\otimes\pi'_{2k} (\gamma g,g) w^\circ_{2k},
\frac{X_1^{2k-2-i}Y_1^i\otimes X_2^{2k-2-j}Y_2^j}
{\|X_1^{2k-2-i}Y_1^i\|\|X_2^{2k-2-j}Y_2^j\|}\rangle\\
&\phantom{\sum_{0\leq i,j\leq 2k-2}\frac{X_1^{2k-2-i}Y_1^i\otimes X_2^{2k-2-j}Y_2^j}{\|X_1^{2k-2-i}Y_1^i\|\|X_2^{2k-2-j}Y_2^j\|}}
\cdot\overline{\langle\pi'_{2k}\otimes\pi'_{2k}(h,h) w^\circ_{2k},
\frac{X_1^{2k-2-i}Y_1^i\otimes X_2^{2k-2-j}Y_2^j}
{\|X_1^{2k-2-i}Y_1^i\|\|X_2^{2k-2-j}Y_2^j\|}\rangle}\\
=&\ \sum_{0\leq i,j\leq 2k-2}
\langle\pi'_{2k}\otimes\pi'_{2k}
(g,g) w^\circ_{2k},\pi'_{2k}\otimes\pi'_{2k}(\gamma^{-1},1)
(X_1^{2k-2-i}Y_1^i\otimes X_2^{2k-2-j}Y_2^j)\rangle\\
&\phantom{\sum_{0\leq i,j\leq 2k-2}}
\cdot\overline{\langle\pi'_{2k}\otimes\pi'_{2k}(h,h) w^\circ_{2k},
X_1^{2k-2-i}Y_1^i\otimes X_2^{2k-2-j}Y_2^j\rangle}
\|X_1^{2k-2-i}Y_1^i\|^{-2}\|X_2^{2k-2-j}Y_2^j\|^{-2}.
\end{split}\]
By \eqref{weightvector}, when $|\gamma|=1$,
\[\pi'_{2k}(\gamma^{-1}) X_1^{2k-2-i}Y_1^i
=\gamma^{-2(k-1-i)}X_1^{2k-2-i}Y_1^i.\]
Therefore
\begin{multline}\label{4}
\langle \pi'_{2k}\otimes\pi'_{2k}
(\gamma g,g) w^\circ_{2k},\pi'_{2k}\otimes\pi'_{2k}(h,h) w^\circ_{2k}\rangle
=\sum_{0\leq i,j\leq 2k-2}
\|X_1^{2k-2-i}Y_1^i\|^{-2}\|X_2^{2k-2-j}Y_2^j\|^{-2}
\gamma^{-2(k-1-i)}\\
\cdot\langle\pi'_{2k}\otimes\pi'_{2k}
(g,g) w^\circ_{2k},X_1^{2k-2-i}Y_1^i\otimes X_2^{2k-2-j}Y_2^j\rangle
\overline{\langle\pi'_{2k}\otimes\pi'_{2k}(h,h) w^\circ_{2k},
X_1^{2k-2-i}Y_1^i\otimes X_2^{2k-2-j}Y_2^j\rangle}.
\end{multline}

Combining \eqref{1} \eqref{2} \eqref{3} \eqref{4}
we have
\[\begin{split}
&\ \int_{\CC^\times/\R^\times}
\langle\pi_{2k}'(t)\phi'_\infty,\phi''_\infty\rangle_{2k}
\sgn^{-2(k-1-r)}(t)\ dt\\
=&\ \frac{\vol(\CC^\times/\R^\times)}{\vol(G'_\infty)}
\frac {\|\phi_3\|^2}{\|w^\circ_{2k}\|^2}\| X^{2k-2-r}Y^{r}\|^{-2}
\sum_{0\leq i,j\leq 2k-2}\gamma^{2(k-1-i)}
\| X^{2k-2-i}Y^{i}\|^{-2}\| X^{2k-2-j}Y^{j}\|^{-2}\\
&\phantom{\sum}\cdot\langle \pi'_{2k}(h)X^{2k-2},X^{2k-2-r}Y^{r}\rangle_{2k}
\langle\pi'_{2k}\otimes\pi'_{2k}(h,h) w^\circ_{2k},
X_1^{2k-2-i}Y_1^i\otimes X_2^{2k-2-j}Y_2^j\rangle\\
&\phantom{\sum\sum}\cdot\int_{G'_\infty} \overline{
\langle\pi'_{2k}\otimes\pi'_{2k}
(g,g) w^\circ_{2k},X_1^{2k-2-i}Y_1^i\otimes X_2^{2k-2-j}Y_2^j\rangle
\langle\pi'_{2k}(g)X^{2k-2},X^{2k-2-r}Y^{r}\rangle_{2k}}\ dg\\
=&\ \frac{\vol(\CC^\times/\R^\times)}{\vol(G'_\infty)}
\frac {\|\phi_3\|^2}{\|w^\circ_{2k}\|^2}\| X^{2k-2-r}Y^{r}\|^{-2}
\sum_{0\leq i,j\leq 2k-2}\gamma^{2(k-1-i)}
\| X^{2k-2-i}Y^{i}\|^{-2}\| X^{2k-2-j}Y^{j}\|^{-2}\\
&\phantom{\sum}\cdot
\langle\big({\pi'_{2k}}^{\otimes 3}\circ \Delta_3(h) \big)
w^\circ_{2k}\otimes X_3^{2k-2},
X_1^{2k-2-i}Y_1^i\otimes X_2^{2k-2-j}Y_2^j
\otimes X_3^{2k-2-r}Y_3^{r}\rangle\\
&\phantom{\sum\sum}\cdot\int_{G'_\infty} \overline{
\langle{\big(\pi'_{2k}}^{\otimes 3}\circ \Delta_3(g)\big)
w^\circ_{2k}\otimes X_3^{2k-2},
X_1^{2k-2-i}Y_1^i\otimes X_2^{2k-2-j}Y_2^j
\otimes X_3^{2k-2-r}Y_3^{r}\rangle}\ dg.
\end{split}\]
Here we denote by $\Delta_3$
the diagonal embedding from $G'_\infty$
to three copies of it.
By the definition of quotient measure on
$\CC^\times\backslash D^\times_\infty\cong
(\R^\times\backslash \CC^\times)\backslash G'_\infty$,
we have
\begin{multline*}
\int_{\CC^\times\backslash D^\times_\infty}
\int_{\CC^\times/\R^\times}
\langle\pi_{2k}'(t)\phi'_\infty,\phi''_\infty\rangle_{2k}
\sgn^{2m}(t)\ dt\ dh\\
=\frac{1}{\vol(\CC^\times/\R^\times)}
\int_{G'_\infty}
\int_{\CC^\times/\R^\times}
\langle\pi_{2k}'(t)\phi'_\infty,\phi''_\infty\rangle_{2k}
\sgn^{2m}(t)\ dt\ dh,
\end{multline*}
and then
\begin{multline*}\begin{aligned}
&\ \int_{\CC^\times\backslash D^\times_\infty}
\int_{\CC^\times/\R^\times}
\langle\pi_{2k}'(t)\phi'_\infty,\phi''_\infty\rangle_{2k}
\sgn^{-2(k-1-r)}(t)\ dt\ dh\\
=&\ \frac{1}{\vol(G'_\infty)}
\frac {\|\phi_3\|^2}{\|w^\circ_{2k}\|^2}\| X^{2k-2-r}Y^{r}\|^{-2}
\sum_{0\leq i,j\leq 2k-2}\gamma^{2(k-1-i)}
\| X^{2k-2-i}Y^{i}\|^{-2}\| X^{2k-2-j}Y^{j}\|^{-2}\end{aligned}\\
\cdot\left|\int_{G'_\infty}
\langle\big({\pi'_{2k}}^{\otimes 3}\circ \Delta_3(g)\big)
 w^\circ_{2k}\otimes X_3^{2k-2},
X_1^{2k-2-i}Y_1^i\otimes X_2^{2k-2-j}Y_2^j
\otimes X_3^{2k-2-r}Y_3^{r}\rangle\ dg\right|^2.
\end{multline*}
Recall that
$\mathbb{P}_{2k}$ is the only
$G'_\infty$-invariant vector
in $(\pi'_{2k})^{\otimes 3}\cdot\Delta_3$
up to a constant multiple.
Then Lemma \ref{Schur}
gives that
\begin{multline*}
\int_{G'_\infty}
\langle\big({\pi'_{2k}}^{\otimes 3}\circ \Delta_3(g)\big)
w^\circ_{2k}\otimes X_3^{2k-2},
X_1^{2k-2-i}Y_1^i\otimes X_2^{2k-2-j}Y_2^j
\otimes X_3^{2k-2-r}Y_3^{r}\rangle\ dg\\
=\vol(G'_\infty)
\langle w^\circ_{2k}\otimes X_3^{2k-2},\frac{\mathbb{P}_{2k}}
{\|\mathbb{P}_{2k}\|}\rangle
\overline{\langle X_1^{2k-2-i}Y_1^i\otimes X_2^{2k-2-j}Y_2^j
\otimes X_3^{2k-2-r}Y_3^{r},\frac{\mathbb{P}_{2k}}
{\|\mathbb{P}_{2k}\|}\rangle}.
\end{multline*}
Lemma \ref{Choosew} shows that
$\langle w^\circ_{2k}\otimes X_3^{2k-2},\mathbb{P}_{2k}\rangle
=\|w^\circ_{2k}\|^2$.
So \[\begin{split}
&\ \int_{\CC^\times\backslash D^\times_\infty}
\int_{\CC^\times/\R^\times}
\langle\pi_{2k}'(t)\phi'_\infty,\phi''_\infty\rangle_{2k}
\sgn^{-2(k-1-r)}(t)\ dt\ dh\\
=&\ \vol(G'_\infty)
\frac {\|\phi_3\|^2\|w^\circ_{2k}\|^2}{\|\mathbb{P}_{2k}\|^4}\| X^{2k-2-r}Y^{r}\|^{-2}
\sum_{0\leq i,j\leq 2k-2}\gamma^{2(k-1-i)}
\| X^{2k-2-i}Y^{i}\|^{-2}\| X^{2k-2-j}Y^{j}\|^{-2}\\
&\phantom{\vol(G'_\infty)\|\phi_3\|^2\|w^\circ_{2k}\|^2\| X^{2k-2-r}Y^{r}\|^{-2}\sum}
\cdot\left|
\langle \mathbb{P}_{2k},X_1^{2k-2-i}Y_1^i\otimes X_2^{2k-2-j}Y_2^j
\otimes X_3^{2k-2-r}Y_3^{r}\rangle\right|^2\\
=&\ \vol(G'_\infty)
\frac {\|\phi_3\|^2\|w^\circ_{2k}\|^2}{\|\mathbb{P}_{2k}\|^4}\| X^{2k-2-r}Y^{r}\|^{2}
\sum_{\substack{0\leq i,j\leq 2k-2\\i+j=3(k-1)-r}}\gamma^{2(k-1-i)}
\| X^{2k-2-i}Y^{i}\|^{2}\| X^{2k-2-j}Y^{j}\|^{2}
|C_{i,j,r}|^2,
\end{split}\]
where $C_{i,j,r}$ is the coefficient in $\mathbb{P}_{2k}$ of 
$X_1^{2k-2-i}Y_1^i\otimes X_2^{2k-2-j}Y_2^j
\otimes X_3^{2k-2-r}Y_3^{r}$.
With \eqref{innerproduct}
and Lemma \ref{length} we 
complete the proof.

(ii) The above result implies that, for $|\gamma|=1$,
\begin{equation}\label{eqn.estimate}
\sum_{m=-(k-1)}^{k-1}|\mathbb{I}_{2k}^{(m)}(\gamma)|
\leq\sum_{\substack{0\leq i,j,r\leq 2k-2\\i+j+r=3(k-1)}}
\binom{2k-2}i^{-1}\binom{2k-2}j^{-1}\binom{2k-2}r^{-1}|C_{i,j,r}|^2.
\end{equation}
To show the right hand side is equal to $\|\mathbb{P}_{2k}\|^2$, 
we write
\[\mathbb{P}_{2k}=\sum_{i,j,r}C_{i,j,r}
X_1^{2k-2-i}Y_1^i\otimes X_2^{2k-2-j}Y_2^j\otimes X_3^{2k-2-r}Y_3^r.\]
Recall that $\{X^{2k-2-r}Y^r\}$ 
forms an orthogonal basis of $V_{\pi_{2k}'}$
which are eigenvectors under the action of $\CC^\times$. 
Analogous to what we have done in Lemma \ref{geom.archi.vanish},
\[\begin{split}
&\ \int_{\CC^\times/\R^\times}
\langle\pi_{2k}'\otimes\pi_{2k}'\otimes\pi_{2k}'(t)\mathbb{P}_{2k},
\mathbb{P}_{2k}\rangle
\sgn^{2m}(t)\ dt\\
=&\ \int_{\CC^\times/\R^\times}
\langle\sum_{i,j,r}^{2k-2}
\sgn^{2(k-1-i)}(t)\sgn^{2(k-1-j)}(t)\sgn^{2(k-1-r)}(t)\\
&\hspace{5cm}
\cdot C_{i,j,r}
X_1^{2k-2-i}Y_1^i\otimes X_2^{2k-2-j}Y_2^j\otimes X_3^{2k-2-r}Y_3^r,
\mathbb{P}_{2k}\rangle
\sgn^{2m}(t)\ dt\\
=&\ \sum_{i,j,r}|C_{i,j,r}|^2
\|X_1^{2k-2-i}Y_1^i\otimes X_2^{2k-2-j}Y_2^j\otimes X_3^{2k-2-r}Y_3^r\|^2
\int_{\CC^\times/\R^\times}\sgn^{2(3(k-1)-i-j-r)}(t)\sgn^{2m}(t)\ dt\\
=&\ \begin{cases}
\sum\limits_{\substack{0\leq i,j,r\leq 2k-2\\i+j+r=3(k-1)-m}}|C_{i,j,r}|^2
\binom{2k-2}i^{-1}\binom{2k-2}j^{-1}\binom{2k-2}r^{-1}
\vol(\CC^\times/\R^\times),
\quad\text{if }m\in\Z,\ |m|\leq 3(k-1);&\\
0,\qquad\text{otherwise.}&\end{cases}
\end{split}\]

On the other hand, $\mathbb{P}_{2k}$ is 
$G_\infty'$-invariant and hence $\CC^\times/\R^\times$-invariant.
Then 
\[\int_{\CC^\times/\R^\times}
\langle\pi_{2k}'\otimes\pi_{2k}'\otimes\pi_{2k}'(t)\mathbb{P}_{2k},
\mathbb{P}_{2k}\rangle
\sgn^{2m}(t)\ dt
=\begin{cases}
\vol(\CC^\times/\R^\times)
\langle \mathbb{P}_{2k},\mathbb{P}_{2k}\rangle,&\text{if }m=0;\\
0,&\text{otherwise.}
\end{cases}\]
Therefore
\[\sum\limits_{\substack{0\leq i,j,r\leq 2k-2\\i+j+r=3(k-1)-m}}
\binom{2k-2}i^{-1}\binom{2k-2}j^{-1}\binom{2k-2}r^{-1}|C_{i,j,r}|^2
=\begin{cases}
\langle \mathbb{P}_{2k},\mathbb{P}_{2k}\rangle&\text{if }m=0;\\
0,&\text{otherwise.}
\end{cases}\]
In particular we show the right hand side of \eqref{eqn.estimate}
is equal to $\|\mathbb{P}_{2k}\|^2$, which completes the proof.
\end{proof}


\section{Examples of small weights}\label{examples246}

In this section 
we simplify our Main Theorem,
particularly $I_0$ and $I_1$, 
for small weights.

\subsection{Weight 2 or 4}\label{ExampleWeight24}
In both cases the only character $\Omega$ contributing 
to the sum on the right hand side of the Main Theorem 
is the trivial character $\1$,
according to Lemma \ref{nonvanishOmega}.
When $2k=2$, the polynomial $\mathbb{P}_2=1$ is trivial,
and hence $\mathbb{I}_{2}^{(0)}(\gamma)=1$ 
as defined in Lemma \ref{geom.archi.coeff}. 
Following the notations in Theorem \ref{maingeometric},
we have
\[I_0(\1)=
\begin{cases}
1,&\text{if }2\nmid N,\\
\frac 12(1+a_2(h))
,&\text{if }2\mid N;\\
\end{cases}\quad
I_1(\1)=
\begin{cases}
1,&\text{if }3\nmid N,\\
\frac 12(1+a_3(h))
,&\text{if }3\mid N.\\
\end{cases}\]
The Main Theorem simplifies to the following identity.
\begin{Example}\label{example.maineqwt2}
For any $h\in\cF_{2}(N)$,
\begin{multline}\label{maineqwt2}
\frac{N}{2^{8} \pi^{5}}
\sum_{\substack{f,g\in\cF_{2}(N)\\
\varepsilon_p=-1,\ \forall p\mid N}}
\frac{L_{\fin}(2,f\times g\times h)}
{(f,f)(g,g)(h,h)}
=\frac{1-24/\varphi(N)}
{2^{\omega(N)}}
+\frac{1}{16\pi^{2}(h,h)}\\
\cdot\left(4L_0 \cdot 2^{\ord_2(N)}
\prod\limits_{p\mid N}\frac{1-\chi_{-4}(p)}{2}
+6\sqrt 3 L_1 \cdot 2^{\ord_3(N)}
\prod\limits_{p\mid N}\frac{1-\chi_{-3}(p)}{2}\right)
,
\end{multline}
where
\[\begin{split}L_0&=
L_{\fin}(1, h)L_{\fin}(1, h\otimes\chi_{-4})\cdot
\begin{cases}
1,&\text{if }2\nmid N,\\
\frac 12(1+a_2(h))
,&\text{if }2\mid N;\\
\end{cases},\\
L_1&=
L_{\fin}(1, h)L_{\fin}(1, h\otimes\chi_{-3})\cdot
\begin{cases}
1,&\text{if }3\nmid N,\\
\frac 12(1+a_3(h))
,&\text{if }3\mid N.\\
\end{cases}\end{split}\]
\end{Example}
Recall that when $p\mid N$, 
$a_p(h)p^{-(k-1)}=
-\varepsilon_p(\frac 12,\pi_h)=\pm 1$ for $h\in\cF_{2k}(N)$.
So the terms $\frac 12(1+a_2(h))$ and $\frac 12(1+a_3(h))$
can either be $0$ or $1$.

When $N$ is a prime with $N=11$ or $N>13$, this reproves
\cite[Theorem 1.1]{feigon2010exact}.
(The Petersson inner product $(f,f)$,
defined in \eqref{def.Petersson}, is normalized 
differently in \cite{feigon2010exact},
so the main formula there is differed by some constants 
comparing with \eqref{maineqwt2}.)

When $2k=4$, \[\begin{split}
\mathbb{P}_{4}&=(X_1Y_2-X_2Y_1)(X_2Y_3-X_3Y_2)(X_3Y_1-X_1Y_3)\\
&=(-Y_1^2X_2^2+X_1^2Y_2^2)X_3Y_3+\text{other terms}.
\end{split}\]
The coefficients $C_{i,2-i,1}$ in $\mathbb{P}_{4}$ of
$X_1^{2-i}Y_1^i\otimes X_2^{i}Y_2^{2-i}
\otimes X_3Y_3$ are given by
\[C_{0,2,1}=1,\quad C_{1,1,1}=0,\quad C_{2,0,1}=-1.\]
Applying Lemma \ref{geom.archi.coeff}, we have
\[\mathbb{I}_{4}^{(0)}(\gamma)
=\frac 12 (\gamma^{2}+\gamma^{-2})
=\begin{cases}
-1,&\gamma=\gamma_0=\sqrt{-1};\\
-\frac 12,&\gamma=\gamma_1=\frac 12(1+\sqrt{-3}).
\end{cases}\]
Using the notations in Theorem \ref{maingeometric},
we have
\[I_0(\1)=
\begin{cases}
-1,&\text{if }2\nmid N,\\
-\frac 12(1+\frac{a_2(h)}2)
,&\text{if }2\mid N;\\
\end{cases}\quad
I_1(\1)=
\begin{cases}
-\frac 12,&\text{if }3\nmid N,\\
-\frac 14(1+\frac{a_3(h)}3)
,&\text{if }3\mid N.\\
\end{cases}\]
The Main Theorem now becomes the following. 
\begin{Example}
For any $h\in\cF_{4}(N)$,
\begin{multline}\label{maineqwt4}
\frac{3N}{2^{17} \pi^{11}}
\sum_{\substack{f,g\in\cF_{4}(N)\\
\varepsilon_p=-1,\ \forall p\mid N}}
\frac{L_{\fin}(5,f\times g\times h)}
{(f,f)(g,g)(h,h)}
=\frac{1}{2^{\omega(N)}}+
\frac{1}{(4\pi)^{4}(h,h)}\\
\cdot\left(
8L_0 \cdot 2^{\ord_2(N)}
\prod\limits_{p\mid N}\frac{1-\chi_{-4}(p)}{2}
+12\sqrt 3 L_1 \cdot 2^{\ord_3(N)}
\prod\limits_{p\mid N}\frac{1-\chi_{-3}(p)}{2}
\right),
\end{multline}
where
\[\begin{split}L_0&=
-L_{\fin}(2, h)L_{\fin}(2, h\otimes\chi_{-4})\cdot
\begin{cases}
1,&\text{if }2\nmid N,\\
\frac 12(1+\frac{a_2(h)}2)
,&\text{if }2\mid N;\\
\end{cases},\\
L_1&=
-\frac 12 L_{\fin}(2, h)L_{\fin}(2, h\otimes\chi_{-3})\cdot
\begin{cases}
1,&\text{if }3\nmid N,\\
\frac 12(1+\frac{a_3(h)}3)
,&\text{if }3\mid N.\\
\end{cases}\end{split}\]
\end{Example}
Notice that $L_0$, $L_1$ are non-positive now, 
unlike that in the weight $2$ case.

\subsection{Weight 6}\label{ExampleWeight6}
In this case nontrivial Hecke characters 
$\Omega$ of $E_0=\Q(\sqrt{-1})$ might 
contribute to the orbital decomposition.
Let $\xi$ be the character 
on $E_0^\times\backslash \A_{E_0}^\times$
which is trivial on $\A_{\Q}^\times$,
unramified everywhere,
and satisfies $\xi_\infty(z)=(z/\bar z)^2$.
With the notations in Theorem \ref{orbdecomp} 
and \eqref{nontrivialorbiteqns},
Lemma \ref{nonvanishOmega} shows that
only $\xi$, $\xi^{-1}$ and the trivial character $\1$
contribute to the nontrivial orbit $I_{[\gamma_0]}$,
and only $\1$ contributes to the orbit $I_{[\gamma_1]}$.

Recall that
\[\mathbb{P}_{6}=(X_1Y_2-X_2Y_1)^2
(X_2Y_3-X_3Y_2)^2(X_3Y_1-X_1Y_3)^2.\]
The coefficient $C_{i,j,r}$ in $\mathbb{P}_{6}$ of
$X_1^{4-i}Y_1^i\otimes X_2^{4-j}Y_2^{j}
\otimes X_3^{4-r}Y_3^r$ vanishes unless $i+j+r=6$, in which case
\begin{center}\begin{tabular}{lllll}
 & & $C_{0,4,2}=1$, & $C_{0,3,3}=-2$, & $C_{0,2,4}=1$, \\
 & $C_{1,4,1}=-2$, & $C_{1,3,2}=2$, & $C_{1,2,3}=2$, & $C_{1,1,4}=-2$, \\
$C_{2,4,0}=1$,  & $C_{2,3,1}=2$, & 
$C_{2,2,2}=-6$, & $C_{2,1,3}=2$, & $C_{2,0,4}=1$, \\
$C_{3,3,0}=-2$,  & $C_{3,2,1}=2$, & $C_{3,1,2}=2$, & $C_{3,0,3}=-2$, & \\
$C_{4,2,0}=1$,  & $C_{4,1,1}=-2$, & $C_{4,0,2}=1$. & & \\
\end{tabular}\end{center}
Applying Lemma \ref{geom.archi.coeff}, we have
\[\mathbb{I}_6^{(0)}(\gamma)=6^{-1}(\gamma^4+ \gamma^0+\gamma^{-4}+4^{-1}
(\gamma^2 +\gamma^{-2}))
=\begin{cases}\frac 5{12},&\gamma=\gamma_0=\sqrt{-1};\\
-\frac 1{24},&\gamma=\gamma_1=\frac 12(1+\sqrt{-3}).\end{cases}
\]
And for $\gamma=\gamma_0$ we have 
\[\mathbb{I}_6^{(-2)}(\gamma_0)=
6^{-1}(\gamma_0^0+\gamma_0^{-4})+
4^{-1}\gamma_0^{-2}=\frac 1{12};
\]
and $\mathbb{I}_6^{(2)}(\gamma_0)=\frac 1{12}$.

With the notations in Theorem \ref{maingeometric},
we have
\begin{gather*}
I_0(\1)=
\begin{cases}
\frac 5{12},&\text{if }2\nmid N,\\
\frac 5{24}(1+\frac{a_2(h)}{2^2})
,&\text{if }2\mid N;\\
\end{cases}\quad
I_1(\1)=
\begin{cases}
-\frac 1{24},&\text{if }3\nmid N,\\
-\frac 1{48}(1+\frac{a_3(h)}{3^2})
,&\text{if }3\mid N;\\
\end{cases}\\
I_0(\xi)=I_0(\xi^{-1})=
\begin{cases}
\frac 1{12},&\text{if }2\nmid N,\\
\frac 1{24}(1-\frac{a_2(h)}{2^2})
,&\text{if }2\mid N.\\
\end{cases}
\end{gather*}
The Main Theorem now becomes, 
\begin{Example}
Let $E_0=\Q(\sqrt{-1})$, 
$\xi$ be the character 
on $E_0^\times\backslash \A_{E_0}^\times$
which is trivial on $\A_{\Q}^\times$,
unramified everywhere,
and satisfies $\xi_\infty(z)=(z/\bar z)^2$.
For any $h\in\cF_{6}(N)$,
\begin{multline}\label{maineqwt6}
\frac{N}{2^{23} \pi^{17}}
\sum_{\substack{f,g\in\cF_{6}(N)\\
\varepsilon_p=-1,\ \forall p\mid N}}
\frac{L_{\fin}(8,f\times g\times h)}
{(f,f)(g,g)(h,h)}
=\frac{1}{2^{\omega(N)}135}
+\frac 1{2^{7}105\pi^{6}(h,h)}\\
\cdot\left(
4L_0 \cdot 2^{\ord_2(N)}
\prod\limits_{p\mid N}\frac{1-\chi_{-4}(p)}{2}
+6\sqrt 3 L_1 \cdot 2^{\ord_3(N)}
\prod\limits_{p\mid N}\frac{1-\chi_{-3}(p)}{2}\right),
\end{multline}
where
\[\begin{split}
L_0 &= 
\begin{cases}
\frac 5{12} L_{\fin}(3, h)L_{\fin}(3, h\otimes\chi_{-4})
+\frac 1{12}
\left(L_{\fin}(\frac {11}2, h\times \Theta_\xi)
+L_{\fin}(\frac {11}2, h\times \Theta_{\xi^{-1}})
\right),\quad
\text{if }2\nmid N,&\\
\frac 5{12}L_{\fin}(3, h)L_{\fin}(3, h\otimes\chi_{-4}),
\hspace{2.6cm} 
\text{if }2\mid N,\ a_2(h)=2^2,&\\
\frac 1{12}
\left(L_{\fin}(\frac {11}2, h\times \Theta_\xi)
+L_{\fin}(\frac {11}2, h\times \Theta_{\xi^{-1}})
\right),
\quad \text{if }2\mid N,\ a_2(h)=-2^2;&
\end{cases}\\
L_1&= -\frac 1{24}L_{\fin}(3, h)L_{\fin}(3, h\otimes\chi_{-3})
\begin{cases}
1,&\text{if }3\nmid N,\\
\frac 12(1+\frac{a_3(h)}{3^2})
,&\text{if }3\mid N;
\end{cases}
\end{split}\]
and $\Theta_\xi$, $\Theta_{\xi^{-1}}$
(defined before Theorem \ref{mainthm})
are the CM modular forms 
arise from Hecke characters $\xi$, $\xi^{-1}$ correspondingly.
\end{Example}


\section{Applications}

\subsection{Sum over three forms}\label{section.sumover3}

The size of $\cF_{2k}(N)$ is given by \cite{martin2005dimensions}:
For any integer $k\geq 1$ and $N$
a square-free integer with an odd number of prime factors, the dimension
of the space of weight $2k$ newforms on $\Gamma_0(N)$ is
\begin{equation}\label{eqn:numberofnewforms}
\#\cF_{2k}(N)=\frac{2k-1}{12}\varphi(N)
-c_2(2k)\prod_{p\mid N}(1-\chi_{-4}(p))
-c_3(2k)\prod_{p\mid N}(1-\chi_{-3}(p))
- \delta(k),
\end{equation}
where $\delta(k)=\begin{cases}1,&\text{if }k=1,\\
0,&\text{otherwise;}\end{cases}$ and
$c_2,c_3$ are defined by
\[c_2(n)=\frac 14+\lfloor\frac n4\rfloor-\frac n4
=\begin{cases}1/4,&n\equiv 0\pmod 4, \\
-1/4,&n\equiv 2\pmod 4;\end{cases}\]
\[c_3(n)=\frac 13+\lfloor\frac n3\rfloor-\frac n3
=\begin{cases}1/3,&n\equiv 0\pmod 3, \\
0,&n\equiv 1\pmod 3, \\
-1/3,&n\equiv 2\pmod 3.\end{cases}\]
In particular,
\[\#\cF_{2}(N)=\frac{\varphi(N)}{12}
+\frac 14\prod_{p\mid N}(1-\chi_{-4}(p))
+\frac 13\prod_{p\mid N}(1-\chi_{-3}(p))
- 1;\]
if $N$ has a prime factor $\equiv 1\pmod 4$ and one $\equiv 1\pmod 3$,
\[\#\cF_{2k}(N)=\frac{2k-1}{12}\varphi(N)- \delta(k).\]

The next result is a direct corollary of Theorem \ref{mainthm} 
by taking the average over $h\in\cF_{2k}(N)$,
and of the size of $\cF_{2k}(N)$:

\begin{Corollary}
Let $N$ be a square-free integer with an odd number of prime factors.
If $N$ has a prime factor $\equiv 1\pmod 4$
and one $\equiv 1\pmod 3$,
we have
\[\frac{N}{ 2^8\pi^{5}}\sum_{\substack{f,g,h\in\cF_{2}(N)\\\varepsilon_p=-1, \forall p\mid N}}
\frac{L_{\fin}(2,f\times g\times h)}{(f,f)(g,g)(h,h)}
=\frac{(\varphi(N)-24)(\varphi(N)-12)}{2^{\omega(N)}12\varphi(N)};\]
\[\frac{N}{2^{12k-4}\pi^{6k-1}}
\sum_{\substack{f,g,h\in\cF_{2k}(N)\\
\varepsilon_p=-1,\ \forall p\mid N}}
\frac{L_{\fin}(3k-1,f\times g\times h)}{(f,f)(g,g)(h,h)}=
\frac{\varphi(N)}
{2^{\omega(N)}12\Gamma(2k-1)^3},\quad
\text{if }2k>2.\]
\end{Corollary}

When all prime factors of $N$ are $\equiv 3\pmod 4$ 
(not equal to $3$) or all prime factors $\equiv 5\pmod 6$,
the right hand side of \eqref{maineq1} contains some weighted averages for 
central values of twisted quadratic base change $L$-functions:
\[\sum_{h\in\cF_{2k}(N)}\frac{L_{\fin}(k,h)L_{\fin}(k,h,\chi_{d})}{(h,h)}\ 
(d=-4\text{ or }-3)
\quad\text{and}\quad
\sum_{h\in\cF_{2k}(N)}
\frac{L_{\fin}(k+|m|+\frac 12,h\times\Theta_{\Omega})}{(h,h)},\]
which have been calculated
by Feigon and Whitehouse in \cite[Theorem 6.10]{feigon2009averages}. 
Notice that, when $2,3\nmid N$,
$I_0$ and $I_1$ (defined in Theorem \ref{maingeometric})
do not depend on $h\in\cF_{2k}(N)$.
Then one can obtain 
an exact average formula of $L_{\fin}(3k-1,f\times g\times h)$
as all three forms run through $\cF_{2k}(N)$.

\subsection{The nonvanishing problems}\label{section.nonvanishing}

When the weight $2k=2$,
we can get a similar result of the nonvanishing problem
as \cite[Corollary 5.2]{feigon2010exact}.
\begin{Corollary}
Let $N$ be the product of an odd number of distinct primes
such that $\varphi(N)>24$.
For each $h\in\cF_2(N)$, there exist $f,g\in\cF_2(N)$
such that $L_{\fin}(2,f\times g\times h)\neq 0$;
moreover,
\[\#\{(f,g)\in\cF_{2}(N)\times\cF_{2}(N):
L_{\fin}(2,f\times g\times h)\neq 0\}\gg_\epsilon N^{3/4-\epsilon}.\]
\end{Corollary}
\begin{proof}
The first statement comes from Example \ref{example.maineqwt2} and
the non-negativity of $L(1,h)L(1,h,\chi_{-d})$.

For $f\in\cF_{2}(N)$ we know
from Hoffstein--Lockhart \cite{hoffstein1994coefficients} that
$(f,f)\gg N(\log N)^{-1}$.
Applying this to \eqref{maineqwt2}
together with the non-negativity of $L(1,h)L(1,h,\chi_{-d})$,
we have
\[\sum_{f,g\in\cF_{2}(N)}L_{\fin}(2,f\times g\times h)\gg
\sum_{\substack{f,g\in\cF_{2}(N)\\\varepsilon_p=-1, \forall p\mid N}}
L_{\fin}(2,f\times g\times h)
\gg \frac{1-24/\varphi(N)}{2^{\omega(N)}}N^2(\log N)^{-3}.\]
Robin \cite{robin1983estimation} shows that, for $N>2$,
\begin{equation}\label{omega}
\omega(N)=\sum_{p\mid N}1\ll \log N/\log\log N.
\end{equation}
So $2^{\omega(N)}\ll N^{C/\log\log N}$.
When $\varphi(N)>24$ we have
\[\sum_{f,g\in\cF_{2}(N)}L_{\fin}(2,f\times g\times h)\gg
N^{2-C/\log\log N}(\log N)^{-3}.\]
Moreover, for any weight $2k$ we have the convexity bound
$L_{\fin}(3k-1,f\times g\times h)\ll_{\epsilon} (N^5k^8)^{1/4+\epsilon}$
(cf. \cite{iwaniec2004analytic}). Therefore
\[
\#\{(f,g)\in\cF_{2}(N)\times\cF_{2}(N):
L_{\fin}(2,f\times g\times h)\neq 0\}
\gg_\epsilon
\frac{N^{3/4-\epsilon}}{N^{C/\log\log N}(\log N)^{3}}
\gg_\epsilon N^{3/4-\epsilon}.
\]

\end{proof}

Analogously we can get a nonvanishing result for higher weight.
Now the coefficients on the right hand side
are not always non-negative 
(see \eqref{maineqwt4} and \eqref{maineqwt6}),
we need more work on the nontrivial orbits on the geometric side.

Recall that, 
when $\Omega$ is unramified everywhere,
the convexity bound of central value of 
base change $L$-function is 
\[L_{\fin}(\frac 12, \pi_E,\Omega)\ll_{\epsilon}(N^2k^4)^{1/4+\epsilon}.\]
Followed by \eqref{nontrivialorbiteqns} we have
\[\begin{split}
\frac{I_{[\gamma]}(f)}{\vol(K')\|\phi_3\|^2 }
=&\ \frac{2c_\gamma}{L(1,\eta)}
\frac{2^{\omega(N)}}{N\cdot L(1,\pi_3,\Ad)} 
\frac{\Gamma(2k-1)^4}{(2\pi)^{2k}\Gamma(k)^3\Gamma(3k-1)}
\cdot \sum_{\Omega}
L_{\fin}(\frac 12,(\pi_3)_E\otimes \Omega)
\mathbb{I}_{2k}^{(m)}(\gamma)\\
\ll_{\epsilon}&\ 
\frac{2^{\omega(N)}}{N\cdot L(1,\pi_3,\Ad)} 
\frac{\Gamma(2k-1)^4}{(2\pi)^{2k}\Gamma(k)^3\Gamma(3k-1)}
\cdot \sum_{\Omega}(N^2k^4)^{1/4+\epsilon}
|\mathbb{I}_{2k}^{(m)}(\gamma)|.
\end{split}\]
Lemma \ref{geom.archi.coeff} and Lemma \ref{length} show that
\[\sum_\Omega |\mathbb{I}_{2k}^{(m)}(\gamma)|\leq 
\|\mathbb{P}_{2k}\|^2=
\frac{\Gamma(k)^3\Gamma(3k-1)}{\Gamma(2k-1)^3}.\] 
Then
\begin{equation}\label{estnontrivorb}
\frac{I_{[\gamma]}(f)}{\vol(K') \|\phi_3\|^2}
\ll_{\epsilon} \frac{2^{\omega(N)}N^{-1/2+2\epsilon}}{L(1,\pi_3,\Ad)} 
\frac{k^{1+4\epsilon}\Gamma(2k-1)}{(2\pi)^{2k}}.
\end{equation}
Recall that $2^{\omega(N)}\ll N^{C/\log\log N}$ and
$L_\infty(1,\pi_3,\Ad)=\frac 2{\pi}(2\pi)^{-2k}\Gamma(2k)$.
For $\pi_3\in\cF(N,2k)$ we know
from Hoffstein--Lockhart \cite{hoffstein1994coefficients} 
and Nelson \cite{nelson2011equidistribution} that
\[L(1,\pi_3,\Ad)\gg \frac{\Gamma(2k)}{(2\pi)^{2k}\log(Nk)}.\]
So 
\begin{equation}\label{eqn.estimate.Igammaf}
\frac{I_{[\gamma]}(f)}{\vol(K') \|\phi_3\|^2}
\ll_{\epsilon} N^{C/\log\log N}N^{-1/2+2\epsilon}
\frac{k^{1+4\epsilon}}{2k-1} \log(Nk)
\ll_{\epsilon}N^{-\frac 12}(Nk)^\epsilon.
\end{equation}

Now we can show that:
\begin{Corollary}
Let $N$ be the product of an odd number of distinct primes.
For $h \in \cF_{2k}(N)$,
\[\#\{(f,g)\in\cF_{2k}(N)\times\cF_{2k}(N):
L_{\fin}(3k-1,f\times g\times h)\neq 0\}
\gg_{k,\epsilon} N^{3/4-\epsilon}.\]
In particular, if $N$ has a prime factor $\equiv 1\pmod 4$ and
one $\equiv 1\pmod 3$,
then for each $h\in\cF_{2k}(N)$, there exist $f,g\in\cF_{2k}(N)$
such that $L_{\fin}(3k-1,f\times g\times h)\neq 0$.
\end{Corollary}

\begin{proof}
Recall that 
\[L(1/2,\pi_{\mathrm{dis}}^{2k}\otimes \pi_{\mathrm{dis}}^{2k}
\otimes \pi_{\mathrm{dis}}^{2k})=
2^{4}(2\pi)^{1-6k}\Gamma(k)^3\Gamma(3k-1).\]
For $2k>2$,
rewrite \eqref{maineqnadelic} as
\[2^{4}\pi
\frac{2^{\omega(N)}}{N^2}\Gamma(2k-1)^3
\sum_{\substack{f,g\in\cF_{2k}(N)\\
\varepsilon_p=-1,\ \forall p\mid N}}
\frac{L_{\fin}(3k-1,f\times g\times h)}
{(2\pi)^{6k}L(1,\pi_f\otimes\pi_g\otimes\pi_h,\Ad)}
=\frac 1{2k-1}+\frac{I_{[\gamma_0]}+I_{[\gamma_1]}}
{\|\phi_3\|^2 \vol(K')}
\]
using the notations in Theorem \ref{maingeometric}.
By the above calculation we know that
the right hand side
$\gg_{k,\epsilon}1$ and
$(2\pi)^{2k}L(1,\pi_3,\Ad)\gg \Gamma(2k)\log(Nk)^{-1}$.
Therefore
\[\sum_{\substack{f,g\in\cF_{2k}(N)\\
\varepsilon_p=-1,\ \forall p\mid N}}
L_{\fin}(3k-1,f\times g\times h)\gg_{k,\epsilon}
\frac{N^2}{2^{\omega(N)}}(\log N)^{-3}
\gg N^{2-C/\log\log N}(\log N)^{-3}.\]
The rest of proof is the same as in the weight $2$ case,
while the convexity bound of $L_{\fin}(3k-1,f\times g\times h)$
does not change with respect to $N$.
\end{proof}

\begin{remark}
A nonvanishing result involving $k$
cannot be derived following the above proof.
For $2k>2$,
rewrite \eqref{maineqnadelic} as
\begin{multline*}2^{4}\pi
\frac{2^{\omega(N)}}{N^2}\Gamma(2k-1)^3
\sum_{\substack{f,g\in\cF_{2k}(N)\\
\varepsilon_p=-1,\ \forall p\mid N}}
\frac{L_{\fin}(3k-1,f\times g\times h)}
{(2\pi)^{4k}L(1,\pi_f\otimes\pi_g,\Ad)}\\
=\frac {(2\pi)^{2k}L(1,\pi_h,\Ad)}{2k-1}+
(2\pi)^{2k}L(1,\pi_h,\Ad)
\frac{I_{[\gamma_0]}+I_{[\gamma_1]}}
{\|\phi_3\|^2 \vol(K')}
\end{multline*}
By \eqref{estnontrivorb} we have 
\[(2\pi)^{2k}L(1,\pi_h,\Ad)
\frac{I_{[\gamma]}(f)}{\vol(K') \|\phi_3\|^2}
\ll_{\epsilon} 2^{\omega(N)}N^{-1/2+2\epsilon}
k^{1+4\epsilon}\Gamma(2k-1)\]
But the main term 
\[\frac {(2\pi)^{2k}L(1,\pi_h,\Ad)}{2k-1}\gg
\frac{\Gamma(2k-1)}{\log(Nk)}\]
is not big enough.

But, 
if $N$ has a prime factor $\equiv 1\pmod 4$ and
one $\equiv 1\pmod 3$, we have
\[2^{4}\pi
\frac{2^{\omega(N)}}{N^2}\Gamma(2k-1)^3
\sum_{\substack{f,g\in\cF_{2k}(N)\\
\varepsilon_p=-1,\ \forall p\mid N}}
\frac{L_{\fin}(3k-1,f\times g\times h)}
{(2\pi)^{6k}L(1,\pi_f\otimes\pi_g\otimes\pi_h,\Ad)}
=\frac 1{2k-1}\]
for $2k>2$.
We can get that
\[\begin{split}\sum_{\substack{f,g\in\cF_{2k}(N)\\
\varepsilon_p=-1,\ \forall p\mid N}}
L_{\fin}(3k-1,f\times g\times h)&\gg_{\epsilon}
\frac{N^2}{2^{\omega(N)}}
\frac 1{\Gamma(2k-1)^3}
\Gamma(2k)^3\log(Nk)^{-3}\frac 1{2k-1}\\
&=\frac{N^2}{2^{\omega(N)}}
\frac {(2k-1)^2}{\log(Nk)^{3}}
\gg \frac{N^{2-C/\log\log N}k^2}{\log(Nk)^{3}}.
\end{split}\]
The convexity bound
$L_{\fin}(3k-1,f\times g\times h)\ll_{\epsilon} (N^5k^8)^{1/4+\epsilon}$
gives
\[
\#\{(f,g)\in\cF_{2k}(N)\times\cF_{2k}(N):
L_{\fin}(3k-1,f\times g\times h)\neq 0\}
\gg_\epsilon N^{3/4-\epsilon}k^{-\epsilon}.
\]
A subconvexity bound of weight aspect 
for $L_{\fin}(3k-1,f\times g\times h)$
might have some application.
More precisely, if $L_{\fin}(3k-1,f\times g\times h)
\ll_{N,\epsilon} k^{2-\delta+\epsilon}$
for some $\delta>0$, we can show that
\[\#\{(f,g)\in\cF_{2k}(N)\times\cF_{2k}(N):
L_{\fin}(3k-1,f\times g\times h)\neq 0\}
\gg_{N,\epsilon} k^{\delta-\epsilon}.\]
\end{remark}


Even though the nonvanishing result involving $k$
cannot be shown,
we can still derive a Lindel\"of-on-average upper bound,
with respect to both the level $N$ and the weight $2k$,
for some
$\sum_{f,g}L_{\fin}(3k-1,f\times g\times h)$.

It is well known that $\varphi(N)\gg N/\log\log N$ 
(cf. \cite[Theorem 6.26]{leveque1996fundamentals}) while $\varphi(N)\leq N$.
By \eqref{eqn:numberofnewforms} we know
\[\frac {kN}{\log\log N}\ll\#\cF_{2k}(N)<\frac{2kN}{12}
.\]
The Generalized Lindel\"of Hypothesis on average
gives an upper bound 
\[\sum_{f,g\in\cF_{2k}(N)}
L_{\fin}(3k-1,f\times g\times h)\ll_{\epsilon}
(kN)^{2+\epsilon}.\]

\begin{Corollary}\label{LindelofOnAverage}
Let $N$ be prime.
Then \[\sum_{f,g\in\cF_{2k}(N)}
L_{\fin}(3k-1,f\times g\times h)
\ll_{\epsilon}k^{3+\epsilon}N^{2+\epsilon}.\]
In particular, if $N$ is a prime $\equiv 1\pmod {12}$,
\[\sum_{f,g\in\cF_{2k}(N)}
L_{\fin}(3k-1,f\times g\times h)
\ll_{\epsilon}(kN)^{2+\epsilon},\]
where the upper bound is 
of the same quality as the Generalized Lindel\"of Hypothesis on average.
\end{Corollary}

\begin{proof}
By the discussion in Remark \ref{remark.mainthm},
when $N$ is a prime, we can remove the condition 
of root numbers $\varepsilon_p=-1$ in the sum
and rewrite \eqref{maineqnadelic} as
\[\frac{2^{2}\pi^4}{N^2 (2k-1)^3}
\sum_{f,g\in\cF_{2k}(N)}
\frac{L_{\fin}(3k-1,f\times g\times h)}
{L_{\fin}(1,\pi_f\otimes\pi_g\otimes\pi_h,\Ad)}
=\frac 1{2k-1}+\frac{I_{[\gamma_0]}+I_{[\gamma_1]}}
{\|\phi_3\|^2 \vol(K')}
\]
for $2k>2$.

We have shown in \eqref{eqn.estimate.Igammaf} that
\[
\frac{I_{[\gamma]}(f)}{\vol(K') \|\phi_3\|^2}
\ll_{\epsilon}N^{-\frac 12}(kN)^\epsilon
\]
for $\gamma=\gamma_0$ or $\gamma_1$.
Then the right hand side (RHS) of the above identity is $\ll_{\epsilon}(kN)^\epsilon$.
Recall that Iwaniec \cite[Theorem 2]{iwaniec1990small} shows 
$L_{\fin}(1,\pi_f,\Ad)\ll_\epsilon (kN)^\epsilon$ for $f\in\cF_{2k}(N)$.
Then we have that
\[\sum_{f,g\in\cF_{2k}(N)}
L_{\fin}(3k-1,f\times g\times h)
\ll_{\epsilon}
N^2 (2k-1)^3(kN)^\epsilon\cdot\text{RHS}
\ll_{\epsilon}N^2k^3(kN)^\epsilon.\]

When $N$ is a prime $\equiv 1\pmod {12}$, 
there is no nontrivial orbit in the geometric side, 
and we can simply replace the RHS with $\frac{1}{2k-1}$
to complete the proof.

\end{proof}

At last we study the nonvanishing modulo suitable primes $p$
of the algebraic part of triple product $L$-values.
Given $f,g,h\in\cF_{2k}(N)$, we define
\[L^{\alg}(3k-1,f\times g\times h)=\Gamma(3k-1)\Gamma(k)^3
\frac {2^{\omega(N)}N}{2^{8k-3}
\pi^{6k-1}}
\frac{L_{\fin}(3k-1,f\times g\times h)}
{(f,f)(g,g)(h,h)}.\]
According to \cite[Theorem 5.7]{bocherer1993central}
(a revised version of this theorem can be found
in \cite{bocherer2003arithmetic},
in the proof of Proposition 2.1),
$L^{\alg}(3k-1,f\times g\times h)$ lies in the subfield of $\CC$
generated by the Fourier coefficients of $f$, $g$ and $h$
and hence is algebraic.

\begin{Corollary}\label{nonvanishmodp}
Let $p$ be a prime such that $p\geq 3k-1$ and $p\neq 2$,
and $\fp$ be a place in $\overline{\Q}$ above $p$.
Let $N$ be a square-free integer with an odd number of prime factors
which has a prime factor $\equiv 1\pmod 4$ and one $\equiv 1\pmod 3$.
Then, when $2k>2$, for any $h\in\cF_{2k}(N)$,
there exist $f,g\in\cF_{2k}(N)$ such that
\[L^{\alg}(3k-1,f\times g\times h)\not\equiv 0\pmod{\fp}.\]
This holds for $2k=2$ too if in addition $p\nmid \varphi(N)-24$.
\end{Corollary}

\begin{proof}
The corollary is obvious,
noting from \eqref{maineq2} that,
for any $h\in\cF_{2k}(N)$,
\[\begin{split}
\sum_{\substack{f,g\in\cF_{2k}(N)\\
\varepsilon_p=-1,\ \forall p\mid N}}
L^{alg}(3k-1,f\times g\times h)&=
(1-\frac{24}{\varphi(N)}\delta(k))2^{4k-1}
\frac{\Gamma(3k-1)\Gamma(k)^3}
{\Gamma(2k-1)^2\Gamma(2k)}\\
&=\begin{cases}
2^3\dfrac{\varphi(N)-24}{\varphi(N)},&2k=2;\\
2^{4k-1}\dfrac{\Gamma(3k-1)\Gamma(k)^3}
{\Gamma(2k-1)^2\Gamma(2k)},&2k>2.\end{cases}
\end{split}\]
\end{proof}



\section*{Acknowledgements}

This paper contains the main result of
the author's doctoral dissertation \cite{guan2020averages}
at CUNY the Graduate Center. 
The cases with weight larger than 4 are completed 
during his time at Shandong University, 
while the author was supported by 
the National Key Research and Development Program of China (No. 2021YFA1000700).

The author would like to thank his advisor,
Professor Brooke Feigon 
for her patient guidance and encouragement,
without whom the completion of this work 
would not have been possible.
The author is also thankful to 
Shih--Yu Chen, Bingrong Huang, Peter Humphries, 
Krzysztof Klosin, Carlos Moreno,
Yiannis Sakellaridis, Lawrence Vu, and Liyang Yang
for many helpful discussions and comments
regarding this work.
At last the author thanks the anonymous referees 
for making very detailed and helpful comments on an earlier version
which led to improvement of the exposition.

\bibliographystyle{alpha}
\bibliography{Averages_Nonvanishing_Triple_L}
\end{document}